\documentclass{article}
\usepackage{arxiv}
\usepackage[utf8]{inputenc} 
\usepackage[T1]{fontenc}    
\usepackage{hyperref}       
\usepackage{url}            
\usepackage{booktabs}       
\usepackage{amsfonts}       
\usepackage{nicefrac}       
\usepackage{microtype}      
\usepackage{graphicx}
\usepackage{ifthen}
\usepackage{tikz}
\usepackage{eso-pic}
\usepackage{hyperref}
\usepackage{color}
\usepackage{float}
\usepackage{multicol}
\usepackage{setspace}
\usepackage{mathptmx}
\usepackage[theorems]{tcolorbox}
\usepackage{mdframed}
\usepackage{amsmath,amsthm,amssymb,mathtools,flexisym,bbm}
\usepackage{xcolor,paralist,hyperref,fancyhdr,etoolbox}
\usepackage{mathrsfs}
\usepackage{booktabs}
\usepackage{pgfplots}
\usepackage{xcolor}
\usepackage[shortlabels]{enumitem}
\usepackage{algorithm}
\usepackage[noend]{algpseudocode}
\usepackage{cite}
\usepackage{multirow}
\usepackage[font=scriptsize]{caption}
\usepackage{subcaption}
\usepackage{comment}
\usepackage{xspace}
\usepackage{environ}
\usepackage{longtable}
\usepackage{indentfirst,csquotes}
\usepackage{array}

\newtheorem{theorem}{Theorem}[section]
\newtheorem{definition}[theorem]{Definition}
\newtheorem{example}[theorem]{Example}
\newtheorem{lemma}[theorem]{Lemma}
\newtheorem{proposition}[theorem]{Proposition}
\newtheorem{corollary}[theorem]{Corollary}

\newtheorem*{claim}{Claim}

\newcolumntype{H}{>{\setbox0=\hbox\bgroup}c<{\egroup}@{}}
\hypersetup{colorlinks=true, linkcolor=black, filecolor=black, urlcolor=black }

\newcommand{\PI}{\textbf{Phase~I}\xspace}
\newcommand{\PII}{\textbf{Phase~II}\xspace}

\newcommand{\R}{\mathbb{R}}

\newcommand{\C}{\mathbb{C}}

\newcommand{\Diag}{\operatorname{Diag}}

\newcommand{\diag}{\operatorname{diag}}

\newcommand{\symm}{\mathbb{S}}
\renewcommand{\symm}{\mathbb{S}}
\newcommand{\PSD}{\mathbb{S}_{+}}
\newcommand{\PD}{\mathbb{S}_{++}}
\newcommand{\rank}{\operatorname{rank}}

\newcommand{\calV}{\mathcal{V}}
\newcommand{\SDP}{\text{SDP}}
\newcommand{\LP}{\text{LP}}
\DeclareMathOperator{\Null}{Null}
\DeclareMathOperator{\range}{Range}

\DeclareMathOperator{\conv}{conv}

\DeclareMathOperator{\CLQ}{CLQ}
\DeclareMathOperator{\thbody}{TH}
\DeclareMathOperator{\STAB}{STAB}

\renewcommand{\subset}{\subseteq}

\DeclareMathOperator*{\argmax}{arg\,max}

\tikzset{every picture/.style={line width=0.75pt}} 
\title{Rounding the Lov\'asz Theta Function with a\\Value Function Approximation} 
\author{Rui Gong\thanks{
\{rgong44,diego.cifuentes,atoriello\}@gatech.edu\\
\indent~~ Georgia Institute of Technology, Atlanta, GA 30332, USA\\
\indent~~ The authors' work was partially funded by the U.S.\ Office of Naval Research, N00014-23-1-2631.} \And 
Diego Cifuentes\footnotemark[2] \And 
Alejandro Toriello\footnotemark[2]}

\begin{document}
\maketitle
\begin{abstract}
The Lov{\'a}sz theta function is a semidefinite programming (SDP) relaxation for the maximum weighted stable set problem, which is tight for perfect graphs.
However, even for perfect graphs, there is no known rounding method guaranteed
to extract a maximum weighted stable set from the SDP solution. 
In this paper, we develop a novel rounding scheme for the theta function
that constructs a value function approximation 
from the SDP solution
and then generates a stable set using dynamic programming.
Our method provably recovers a maximum weighted stable set in several sub-classes of perfect graphs,
including generalized split graphs,
which asymptotically cover almost all perfect graphs.
To the best of our knowledge, this is the only known rounding strategy for the theta function that recovers a maximum weighted stable set for large classes of perfect graphs. 
Our rounding scheme relies on simple linear algebra computations;
we only solve one SDP. 
In contrast, existing methods 
that leverage the theta function require solving multiple SDPs. 
Computational experiments show that our method produces good solutions even on imperfect graphs.
\end{abstract} 

\section{Introduction}
\label{sec:intro}
Semidefinite programming (SDP) relaxations provide a tractable approach for tackling hard combinatorial optimization problems.
Two of the most studied cases are SDP relaxations for the \emph{maximum weighted stable set (MWSS) problem} , known as the \emph{Lov\'asz theta function}, and for the maximum cut problem.
If the relaxation is tight and the optimal solution has rank one, it is easy to recover a solution of the combinatorial problem from the SDP solution.
However, if the optimal solution has higher rank, 
a \emph{rounding} procedure may be needed, even if the upper bound from the relaxation is tight. 
Devising good rounding procedures is hence of great importance both in theory and in practice.
For example, Goemans and Williamson \cite{10.1145/227683.227684} introduced a randomized rounding algorithm for the maximum cut SDP, which produces a solution with an approximation factor of roughly $0.878$. 
While some heuristic rounding strategies have been proposed for MWSS problem using the theta function \cite{Yildirim-Fan,sublinear-time-Alizadeh},
none of them have theoretical guarantees.
In this paper, we introduce a rounding scheme for this SDP 
that provably finds an MWSS in several important subclasses of perfect graphs, which asymptotically include almost all perfect graphs. 
Moreover, our rounding scheme performs well in practice, even in imperfect graphs.

Given a simple graph $G=(N,E)$ and a weight function $w:N\to\R_{++}$, the MWSS problem seeks the stable set 
$S\subseteq N $ that maximizes $ \sum_{i\in S} w_i $;
a set $S$ is \emph{stable} (or \emph{independent} or a \emph{packing}) if $ij \not\in E$ for every pair of vertices $i, j \in S$.
The stability number $\alpha(G;w)$ is the optimal value of the problem;
computing $\alpha(G;w)$ is NP-hard for general graphs~\cite{NP-hard,Karp1972}.
Lov\'asz introduced a quantity $\vartheta(G;w)$,
the {theta function},
which upper bounds $\alpha(G;w)$ 
and can be efficiently computed by solving an SDP \cite{1055985,ellipsoid_MSS}.
Indeed, $\vartheta(G;w)$ is the optimal value of the following primal-dual pair of SDPs \cite{Grötschel1993,ellipsoid_MSS,GROTSCHEL1984325}:
\begin{multicols}{2}
    \noindent
    \begin{align}
\label{SDP-P}
\tag{SDP-P}
\begin{split}
    \max\quad &w^\top x\\
    \text{s.t.}\quad &X_{ii}=x_i, \forall i\in N\\
    &X_{ij}=0,\forall ij\in E\\
    &X-xx^\top\succeq 0
\end{split}
\end{align}
\columnbreak
\begin{align}
\label{SDP-D}
\tag{SDP-D}
\begin{split}
    \min\quad &t\\
    \text{s.t.}\quad
    &Q_{ii}=-2q_i-w_i,\forall i \in N\\
    &Q_{ij}=0,\forall ij\notin E\\
    &tQ-qq^\top\succeq 0 .
\end{split}
\end{align}
\end{multicols}
\noindent
Our rounding method relies on an optimal solution of the primal-dual pair above.

The MWSS problem can be solved in polynomial time for an important class of graphs, the \emph{perfect} graphs; 
this follows because $\alpha(G;w) = \vartheta(G;w)$ for these graphs \cite{Grötschel1993}.
Hence, in a perfect graph we can determine whether a vertex $i \in N$ belongs to an MWSS by checking if $\vartheta(G;w) = \vartheta(G|_{N \setminus i};w)$.
We can thus extract an MWSS 
by solving $\mathcal{O}(n)$ SDPs~\cite{GROTSCHEL1984325};
see \cite{Yildirim-Fan} for an improved scheme
that uses $\min\{\alpha(G;w), n/3\}$ SDPs. 
Alternatively, Alizadeh \cite{sublinear-time-Alizadeh} considers a randomized algorithm that requires perturbing the weight vector $w$ up to $\mathcal{O}(\log (1/\epsilon))$ times in order to guarantee that \eqref{SDP-P} has a unique rank-one optimal solution with probability $1 - \epsilon$. 
All previous polynomial-time methods for the MWSS problem in perfect graphs require solving multiple SDPs.
In addition, they do not rely on rounding an SDP solution, but instead either use the SDP optimal value or require an exact, rank-one optimal solution;
in contrast, our rounding procedure requires solving \eqref{SDP-D} once.

To describe the main idea behind our rounding method, 
let $(x,X), (t,q,Q)$ be optimal solutions of \eqref{SDP-P}, \eqref{SDP-D} respectively.
We consider $\calV_{\SDP} : 2^N \to \R_{+}$ defined as
\begin{align}
    \label{eq:V_SDP}
    \calV_{\SDP}(S) & := q_S^\top  Q_S^{\dagger}  q_S, \quad S \neq \emptyset, 
& \calV_{\SDP}(\emptyset) & := 0,
\end{align}
where $q_S$ and $Q_S$ are respectively the sub-vector of $q$ and principal sub-matrix of $Q$ indexed by $S$, 
and $Q_S^\dagger$ denotes the Moore–Penrose pseudo-inverse of $Q_S$. 
$\calV_{\SDP}$ can be computed either using an eigenvalue decomposition or using equivalent characterizations we introduce in later sections. 
Our rounding algorithm evaluates $\calV_{\SDP}$ at most $\mathcal{O}(n^2)$ times, and produces a stable set $S \subseteq N$ for an arbitrary graph~$G$.
Notice that $q$ and $Q$ are fixed and given by solving ~\eqref{SDP-D} once. 
The function $\calV_{\SDP}$ is monotone and satisfies the following important property,
\begin{equation}
\label{eq:valuefun}
    \calV(S)- \calV( S \setminus (\{i\}\cup\delta_i) ) \geq w_i, \qquad
     S \subseteq N, i \in S,
\end{equation}
where $\delta_i$ denotes the set of neighbors of $i$ in $G$.
More generally, we call $\calV: 2^N \to \R_+$ a \emph{value function approximation} (VFA) for the MWSS problem if the function is monotone, $\calV(\emptyset) = 0$, and it satisfies \eqref{eq:valuefun}.
The term \emph{value function }comes from dynamic programming (DP). 
If we view the construction of a stable set as a sequential decision process, the \emph{state} is the set of available vertices $S \subseteq N$. 
Selecting a vertex $i \in S$ yields an immediate reward of $w_i$, and transitions the state to the remaining available vertices $S \setminus (\{i\} \cup \delta_i)$. 
In this context, the optimal value function—which represents the maximum possible weight achievable from state $S$—is exactly the stability number $\mathcal{V}^*(S) := \alpha(G|_S; w)$. 
By the Bellman equation, this optimal value function naturally satisfies \eqref{eq:valuefun}.

Our rounding procedure starts by discarding all vertices $i$ with $x_i = 0$.
We then arbitrarily select a remaining vertex to be included in the stable set and discard its neighbors.
We keep discarding vertices with
\(
\calV_{\SDP}(S) - \calV_{\SDP}( S \setminus ( \{i\} \cup \delta_i ) ) > w_i
\)
and selecting until no vertices are left, and then return the selected set of vertices.
This process may sometimes lead to choosing a vertex incorrectly, 
so we use a DP technique known as a one-step \emph{look-ahead} to prevent this from occurring: we simulate a vertex choice and check if it leads to a suboptimal stable set; if so, we backtrack and delete that vertex.  
The output of our procedure is always a stable set; we show that it is indeed an MWSS for several important subclasses of perfect graphs. 

The analysis of our rounding scheme uses the clique linear programming (LP) relaxation of the MWSS problem.
Consider the primal-dual pair of LPs
\begin{multicols}{2}
\noindent
\begin{align}
\label{LP-P}
\tag{LP-P}
\begin{split}
    \max_{x\geq 0} &\ w^{\top}x\\
    \text{s.t.} &\ \sum_{i\in C} x_i\leq 1,\forall C\in \mathcal{C}(G),
\end{split}
\end{align}
\columnbreak
\begin{align}
\label{LP-D}
\tag{LP-D}
\begin{split}
    \min_{\mu\geq 0}&\ \sum_{C\in \mathcal{C}(G)}\mu_{C}\\
    \text{s.t.} &\ \sum_{C\ni i}\mu_C\geq w_i,\forall i\in N,
\end{split}
\end{align}
\end{multicols}
\noindent
where $\mathcal{C}(G)$ is the set of all cliques in $G$; this relaxation is tight and the primal polyhedron of \eqref{LP-P} is integral for perfect graphs \cite{Grötschel1993}. 
Given a solution of \eqref{LP-D},
we can construct another VFA $\calV_{\LP} : 2^N \to \R_+$ and apply our algorithm with $\calV_{\LP}$.
We emphasize that our algorithm does not solve \eqref{LP-D}, but we use it to analyze our rounding scheme with \eqref{SDP-D}.
Our main results are stated next.
\begin{theorem}[Informal]
\label{thm1}
Our rounding algorithm based on either \eqref{LP-D} or \eqref{SDP-D} outputs a maximum weighted stable set for generalized split, chordal, and co-chordal graphs.
\end{theorem}
The precise statement of our main theorem is given in Section~\ref{sec:preliminaries}. 
Note that almost all perfect graphs are generalized split graphs in an asymptotic sense \cite{Prömel_Steger_1992}.
To the best of our knowledge, our methods give the only known rounding strategy for the Lov\'asz theta function that provably works for large subclasses of perfect graphs.

The rest of the paper is organized as follows. Section~\ref{sec:preliminaries} provides background, properties of a VFA, our algorithm and the formal statement of the main results.
Section~\ref{sec:LPVFA} analyzes our algorithm with $\calV_{\LP}$, discusses its combinatorial interpretation, and highlights the necessity of the look-ahead. 
Then Section~\ref{sec:SDPVFA} presents the analysis of our algorithm with $\calV_{\SDP}$ and proves the main results.  
Section~\ref{sec:comp_exp} describes our computational study.

\subsection{Brief Literature Review}

Beyond the work mentioned above, there is a huge body of literature on the MWSS problem. 
While it is NP-hard for general graphs~\cite{Karp1972,NP-hard}, there are polynomial-time combinatorial algorithms for various graph classes, including circular graphs and their complements~\cite{Gavril1973AlgorithmsFA}, circular arc graphs and their complements~\cite{Gavril1974AlgorithmsOC,DBLP:journals/networks/GuptaLL82}, graphs without long odd cycles~\cite{wo_long_odd_cycle}, claw-free graphs and $\ell$-claw-free graphs~\cite{10.1145/2629600,10.5555/2095116.2095218}, and graphs without two disjoint odd cycles~\cite{2odd-cycle-free}.
For graphs without holes of length at least five, the results in \cite{doi:10.1137/20M1383732} imply an $n^{\mathcal{O}(k)}$ algorithm, where $k$ is an upper bound for the treewidth. 
Recently, \cite{walteros_why_2020} gave a fixed-parameter-tractable algorithm for the cardinality (unweighted) objective that depends exponentially on $g:=(d+1)-\alpha$, where $d,\alpha$ are respectively the degeneracy and unweighted stability number. There are polynomial-time combinatorial algorithms for the cardinality objective on some sub-classes of perfect graphs, including chordal graphs \cite{10.1137/0201013} and generalized split graphs \cite{ESCHEN2014195}, which we also study. Unlike these works, which are tailored to a graph class, our algorithm is generic, works with arbitrary weights, and can be applied to any perfect graph; it furthermore exhibits favorable computational performance. To the best of our knowledge, there is no known polynomial-time combinatorial algorithm for general perfect graphs.

There are also exact approaches for the MWSS problem based on integer programming techniques. These include branch-and-bound \cite{CARRAGHAN1990375,Babel,Babel_Tinhofer,doi:10.1137/0215075,Balas1991MinimumWC,PARDALOS1992363,Mannino_Sassano,NemSigismondi}, which often derive bounds from clique covers, and cutting-plane algorithms, e.g.\ \cite{BalasSamuelsson,NemSigismondi,NemTrotter}. 
There are also highly effective heuristic and meta-heuristic methods for the stable set problem, such as tabu search \cite{FRIDEN1990437}. The algorithm proposed in \cite{BensonYe} is particularly relevant for us, as it uses \eqref{SDP-P}; despite having no theoretical guarantees, it performs quite well in our computational experiments.
For further study on relevant algorithms, we refer the reader to~\cite{WU2015693,marino2024shortreviewnovelapproaches} and the references therein.

Our algorithm can be adapted to efficiently output all maximum weighted stable sets (rather than one of them) for unipolar, chordal and co-chordal graphs. 
Algorithms that enumerate all optimal solutions of a problem have a long history in combinatorial optimization, beginning with ranked-enumeration frameworks for assignments \cite{doi:10.1287/opre.16.3.682}, and the $K$-best procedure \cite{doi:10.1287/mnsc.18.7.401}, which shows how to systematically recover an entire family of optimal solutions rather than a single optimum.
For stable sets, however, all-optima results are much rarer than single-solution algorithms to our knowledge: \cite{OKAMOTO2008229} show that all maximum stable sets of a chordal graph can be enumerated in constant amortized time per output, and \cite{143921} consider the problem of finding all maximum weighted stable sets in interval and circular-arc graphs. 
By contrast, standard references for unipolar graphs, e.g.\ \cite{ESCHEN2014195}, focus on computing one maximum stable set efficiently rather than enumerating the full family of optimal solutions.

\subsection{Contributions}
\begin{itemize}
    \item  We introduce a novel rounding method for extracting a stable set from the optimal solution of the Lov\'asz theta function SDP. 
    Our rounding procedure relies on a VFA constructed from the dual variables.
    \item Our rounding procedure is guaranteed to extract an MWSS for generalized split, chordal, and co-chordal graphs, and can be modified to efficiently output all MWSS for unipolar, chordal and co-chordal graphs.
    To our knowledge, this is the first rounding procedure for the theta function SDP that provably works for any of these classes.
    \item We test our method in several perfect and imperfect graphs with up to 60,000 nodes.
    The experiments show that our method produces near-optimal solutions and is computationally efficient; the majority of the runtime is taken up by solving the SDP.
\end{itemize}
%

\subsection{Notation}
We let $\R$ be the real numbers.  
For a finite set $N$, we denote the set of $N\times 1$ vectors, non-negative vectors and positive vectors by $\R^{N}$, $\R^{N}_{+}$, $\R_{++}^{N}$ respectively; the sets of $N\times N$ symmetric matrices, positive semidefinite (PSD) matrices, and positive definite matrices are $\symm^N$, $\PSD^N$, and $\PD^{N}$, respectively. 
For $A,B\in\symm^N$, the partial order $A\succeq B$ is defined by $A-B\in\PSD^{N}$, and similarly, $A\succ B$ is defined by $A-B\in\PD^N$.
For $N=[n]$ or when $|N|=n$, we sometimes use $\R^{n}$ instead of $\R^{N}$ for convenience, and the same applies to other sets with superscript~$N$.
For any $q\in\R^{N},Q\in\symm^N$ and $S\subseteq N$, we respectively let $q_S$, $Q_S$ denote the sub-vector and principal sub-matrix of $q,Q$ indexed by~$S$. 
Given a matrix $D \in \symm^n$,
let $\diag (D) = (D_{11},\dots, D_{nn})\in \R^n$ be the vector consisting of its diagonal entries.
Conversely, given $d \in \R^n$,
let $\Diag (d) \in \symm^n$ be the diagonal matrix with $d$ in its diagonal.

Throughout the paper, assume that $G=(N,E)$ is a simple undirected graph, where $N$ is the set of vertices, $E$ is the set of edges, and $n:=|N|$, $m:=|E|$. 
Denote the complement of $G$ by $\overline{G}$, $N\setminus S$ by $\overline{S}$, and let $G|_{S}$ be the induced subgraph of $G$ on $S\subseteq N$.
We let $w:N\to\R_{++}$ denote a weight function over~$N$.
For a vertex $i\in N$, we let $\delta_i$ denote the set of its neighbors in $G$, and let $\delta_S$ denote the set of vertices $j$ such that $j\in \overline{S}$ and $j\in\delta_i$ for some $i\in S$.
A set $S\subseteq N$ is \emph{stable} in $G$ if no two vertices in $S$ are adjacent, and is a \emph{clique} if every two vertices in $S$ are adjacent.
A stable set (clique) $S$ is a \emph{maximum weighted stable set (clique)} if it maximizes $\sum_{i\in S} w_i$ among all stable sets (cliques) in $G$. 
For $ w $ and $ S \subseteq N $, we let $ \alpha(S) := \alpha(G \vert_S; w )$ denote the weight of a max stable set contained in $S$.

\section{Preliminaries and Algorithm}
\label{sec:preliminaries}

In this section, we detail necessary background, 
then introduce our algorithms and some properties of the VFA.  
Finally, we introduce the VFAs based on LP and SDP, and state the main results.

\subsection{Perfect Graphs and Convex Relaxations}

A graph $G$ is \emph{perfect} if the chromatic number $\chi(G')$ equals the clique number $\omega(G')$ for every induced subgraph $G'$ of~$G$.
Equivalently, $G$ is perfect if and only if it does not contain a chordless odd cycle (an \emph{odd hole}) or its complement (an \emph{odd anti-hole}) as an induced subgraph \cite{StrongPerfectGraph}.

The MWSS problem can be formulated as an LP over the \emph{stable set polytope},
\[
\STAB(G):=\conv\{x\in\{0,1\}^N:x\text{ is an incidence vector of a stable set in }G\}.
\]
Let $\CLQ(G) \subset \R^N$ be the feasible set of~\eqref{LP-P},
and let $\thbody(G) \subset \R^N$ be the projection of the feasible set of \eqref{SDP-P} onto the $x$ variables; 
$\thbody(G)$ is known as the \emph{theta body} of $G$~\cite{thetabody}. 
The following relations hold~\cite{Grötschel1993,ellipsoid_MSS,GROTSCHEL1984325}:
\begin{align}
\label{relation:perfect graph}
& \STAB(G)\subset\thbody(G)\subset\CLQ(G); & 
& \STAB(G)=\thbody(G)=\CLQ(G),
\text{ if } G \text{ is perfect.}    
\end{align}
Thus, the problems \eqref{LP-P} and \eqref{SDP-P} are convex relaxations of the MWSS problem, and both relaxations are tight for perfect graphs. 

Problem~\eqref{SDP-P} can be solved in polynomial time using interior point methods~\cite{nesterov1994interior}. 
Interior point methods return a point in the relative interior of the \emph{optimal face} of \eqref{SDP-P};
see, e.g., \cite[Theorem 3.7]{doi:10.1137/0805002}. 
We call such a point a \emph{relative interior solution}.
This property of interior point methods contrasts with methods such as simplex,
which return extreme points of the optimal face.

\subsection{Retrieving a Stable Set from a VFA}

Let $G=(N,E)$ be a simple undirected graph and $w:N\to\R_{++}$ be a weight function over $N$. 
For a set of vertices $I\subseteq N$, we let $\alpha(I)$ denote the weighted stability number of the induced subgraph $G|_I$. 
Interpreted as a set function $ \alpha: 2^N \to \R_+ $, it is the optimal value function of the MWSS problem. 
A VFA is a function $\calV:2^N\to\R_{+}$ with: 
\begin{enumerate}[leftmargin=*,label=(\roman*)]
    \item $\calV(\emptyset)=0.$
    \item $\calV$ is monotone, $\calV(I)\leq \calV(J)$ for $I\subseteq J$.
    \item $\calV$ satisfies \eqref{eq:valuefun}.
\end{enumerate}
We say a VFA is \emph{tight} if $\calV(N) = \alpha(N)$.
We now show that a VFA upper bounds the optimal value function, $\alpha(I)$.
\begin{lemma}
\label{lem:VFA_Bound}
    Given a VFA $\calV$, 
    $\calV(I)\geq\alpha(I)$ for all $I\subseteq N$. 
\end{lemma}
\begin{proof}
    For $I=\{1,\ldots,s,{s+1},\ldots, |I|\}$, without loss of generality, assume $S^*=\{1,\ldots,s\}$ is an MWSS of $G|_I$. 
    Then by applying~\eqref{eq:valuefun} repeatedly,
    \[
    \calV(I)\geq w_1+\calV(I\setminus(\delta_{1}\cup \{1\}))\geq \dotsb \geq \sum_{i\in S^*}w_i+\calV(I\setminus (\delta_{S^*}\cup S^*))=\alpha(I).
    \]
    The last inequality follows because $S^*$ is an MWSS of $G|_I$, and thus $I \subseteq (\delta_{S^*}\cup S^*) $. 
\end{proof}
The following are sufficient conditions for a set function to be a VFA.
\begin{lemma}\label{lem:VFAconditions}
    Let $\calV : 2^N \to \R_+$ satisfy
    (i) $\calV(\emptyset) = 0$,
    (ii) $\calV$ is monotone,
    (iii) $\calV(\{i\}) \geq w_i$ for $i\in N$,
    and
    (iv) $\calV(I \cup J) = \calV(J) + \calV(I)$
    for all disjoint $I, J \subseteq N$ with no edge between them.
    Then $\calV$ is a VFA.
\end{lemma}
\begin{proof}
Suppose all the conditions above hold.
    Since there is no edge between $i$ and $J:=I\setminus(\{i\}\cup\delta_i)$, by (iv) we have that $\calV(i\cup J)=\calV(i)+\calV(J)$. 
    Using (ii) and (iii) we get that
    \[\calV(I)\geq \calV(i\cup J)=\calV(i)+\calV(J)\geq w_i+\calV(J).\]
    Hence, $\calV$ is a VFA. 
\end{proof}

Algorithm \ref{alg:look ahead}
constructs a stable set from any VFA.
The algorithm maintains a stable set $S \subset N$
and a set $I \subset N$ of vertices yet to be processed.
When a vertex $i$ is selected to be in $S$,
it is discarded from $I$ together with its neighbors.
Algorithm \ref{alg:look ahead} is designed to work with a tight VFA;
our main theoretical contribution is the analysis of Algorithm \ref{alg:look ahead} on several classes of perfect graphs, where we can construct tight VFAs because of \eqref{relation:perfect graph}. 

\begin{algorithm}[htb]
\begin{algorithmic}[1]
\State\textbf{Input:} $G=(N,E)$, a weight function $w$, optimal $x^*\in\STAB(G)$, and a tight VFA $\calV$.
\item[\textbf{Phase I:}]
\State $S \gets \emptyset$
\Comment{Start with the empty stable set.}
\State $I \gets \{ i \in N : x_i^* >0 \}$\label{alg_step:subopt by primal}
\Comment{Discard vertices not in any optimal set.}
\item[\textbf{Phase II:}]
\While{$I\neq\emptyset$}
    \State $I' \gets I$
    \State $I' \gets I'\setminus(\{i\}\cup\delta_i)$ for an arbitrary $i\in I'$\label{alg_step:select}
    \Comment{Tentatively select $i$ and discard $\delta_i$.}
    \While{$\exists j\in I'$ with $\calV(I')-\calV(I'\setminus(\{j\}\cup\delta_j))> w_j$}\label{alg_step:discard}
        \State $I' \gets I'\setminus \{j\}$
        \Comment{Discard vertices that cannot be in an optimal set with $i$.}
    \EndWhile
    \If{$\calV(I)>\calV(I')+w_i$}\label{alg_step:lookahead start}
    \Comment{Check if $i$ is a suboptimal choice.}
        \State $I \gets I\setminus\{i\}$
    \Else  
        \State $I \gets I'$, $S \gets S\cup\{i\}$
        \Comment{Confirm choice of $i$ and continue.}
    \EndIf\label{alg_step:lookahead end}
\EndWhile
\State\textbf{Output:} Return $S$, a stable set of $G$.
\end{algorithmic}
\caption{Retrieving a stable set from a tight VFA with one-step look ahead} 
\label{alg:look ahead}
\end{algorithm}

\begin{definition}
    During each iteration of~\PII~of Algorithm~\ref{alg:look ahead}, if the set of selected vertices $S$ is a subset of an MWSS of $G$, we say we are on an \emph{optimal trajectory}.
    If an iteration starts with the remaining set of vertices $I$ and $S\cup\{i\}$ is a subset of an MWSS of $G$ for some $i\in I$, then we say $i$ is an \emph{optimal choice} of this iteration or $i$ is \emph{on an optimal trajectory}, otherwise, we say $i$ is \emph{suboptimal} or not on an optimal trajectory.
\end{definition}

After making a copy $I'$ of $I$ and tentatively selecting a vertex $i$ in \PII, Algorithm~\ref{alg:look ahead} tentatively discards vertices that are not in any MWSS of the current graph $G|_{I'}$. 
It then checks if $i$ is indeed an optimal choice: 
if not, Algorithm~\ref{alg:look ahead} returns to $I$, deletes $i$, and continues;
otherwise, Algorithm~\ref{alg:look ahead} adds $i$ to $S$, updates $I$ as $I'$ and continues.
The following lemma justifies these claims.
\begin{lemma}
\label{lem:fund_VFA_prop}
    Let $I \subseteq N$, $i \in I$.
    A VFA $\calV$ has the following properties:
    \begin{enumerate}[leftmargin=*,label=(\roman*)]
        \item If $\calV(I)=\alpha(I)$ and $\calV(I)-\calV(I\setminus(\{i\}\cup\delta_i))>w_i$, then $i$ is not in any MWSS of $G|_{I}$.\label{item:fund_VFA_prop_1}
        \item If~$\calV(I)=\alpha(I)$ and either $\calV(I)-\calV(I\setminus(\{i\}\cup\delta_i))>w_i$ or $\alpha(I)=\alpha(I\setminus\{i\})$,
        then $\calV(I\setminus\{i\})=\alpha(I\setminus\{i\})$.\label{item:fund_VFA_prop_2}
        \item Suppose  $\calV$ is tight.
        At the start of any iteration of~\PII~of Algorithm~\ref{alg:look ahead}, if $S$ is on an optimal trajectory, then $\calV(I) {=} \alpha(I)$ and for any MWSS $S_I$ of $G|_I$, $S\cup S_I$ is an MWSS of~$G$. \label{item:fund_VFA_prop_3}
    \end{enumerate}
\end{lemma}
\begin{proof}
\begin{enumerate}[leftmargin=*,label=(\roman*)]
    \item For contradiction, suppose $i$ is in an MWSS $S^*$ of $G|_{I}$. 
    Without loss of generality, suppose $i=1$; by~\eqref{eq:valuefun},
    \[
    \calV(I)> w_1+\calV(I\setminus(\delta_{1}\cup \{1\}))\geq\ldots\geq \sum_{i\in S^*}w_i+\calV(I\setminus (\delta_{S^*}\cup S^*))=\sum_{i\in S^*}w_i+0=\alpha(I),
    \]
    which contradicts $\calV(I)=\alpha(I)$.
    
\item By the first property of this lemma and Lemma~\ref{lem:VFA_Bound},  $ \calV(I) = \alpha(I) = \alpha(I\setminus\{i\}) \leq \calV(I\setminus\{i\})$.
Since $\calV$ is monotone, $\alpha(I\setminus\{i\})=\calV(I)=\calV(I\setminus\{i\})$.
    
\item We proceed by induction.
    The base case follows from the tightness assumption. 
    Assume the set of selected vertices $S$ is a subset of an MWSS of $G$, and the remaining set of vertices $I$ satisfies $\calV({I})=\alpha({I})$. 
    Suppose $i\in {I}$ is selected in the current iteration and ${S}\cup\{i\}$ is a subset of an MWSS of $G$.
    Let us show that the statement holds for the next iteration.
    
We claim that $i$ belongs to an MWSS of $G|_{{I}}$.
    Suppose not; then  
    \[
    \alpha({I})>\alpha({I}\setminus(\{i\}\cup\delta_i))+w_i\implies
    \alpha(N)>\sum_{j\in S\cup\{i\}}w_j+\alpha({I}\setminus(\{i\}\cup\delta_i)),
    \]
    so ${S}\cup\{i\}$ is not a subset of an MWSS of $G$, a contradiction. 
    Set $I' \gets {I}\setminus(\{i\}\cup\delta_i)$ and let $S_{{I}}$ be an MWSS of $G|_{{I}}$ containing $i$.
    Then $S_{{I}}\setminus\{i\}$ is an MWSS of $I'$ and
    \[\calV(I')=\calV({I}\setminus(\{i\}\cup\delta_i))=\calV({I})-w_i=\alpha({I}\setminus(\{i\}\cup\delta_i)).\]
    After discarding all vertices $j$ with $\calV(I')-\calV(I'\setminus(\{j\}\cup\delta_j))>w_j$ and updating $I'$, by the second property of this lemma, $\calV({I}\setminus(\{i\}\cup\delta_i))=\calV(I')$,
    which implies $\calV(I')=\alpha({I}\setminus(\{i\}\cup\delta_i))$.
    Again by the first property of this lemma, $\alpha({I}\setminus(\{i\}\cup\delta_i))=\alpha(I')$, hence $\calV(I')=\alpha(I')$.
    By the induction assumption and the tightness of $\calV$, 
    \[
    \calV(N)=\sum_{j\in S}w_j+\calV({I})=\sum_{j\in S}w_j+w_i+\alpha({I}\setminus(\{i\}\cup\delta_i))=\sum_{j\in S}w_j+w_i+\alpha(I')=\alpha(N).
\]
Thus, for any MWSS $S_{I'}$ of $G|_{I'}$, $S_{I'}\cup (S\cup\{i\})$ is an MWSS of $G$. 
\end{enumerate}
\end{proof}
\begin{corollary}
\label{cor:fund_VFA_prop}
    Suppose a VFA $\calV$ is tight.
        At the beginning of any iteration of \PII~of Algorithm~\ref{alg:look ahead}, if $S$ is a subset of an MWSS of $G$, then $\sum_{i\in S}w_i+\calV(I)=\alpha(N)$. 
\end{corollary}

Lemma~\ref{lem:fund_VFA_prop} provides important properties of our algorithm. 
First, given $\calV(I)=\alpha(I)$, discarding vertices with $\calV(I) - \calV( I \setminus ( \{i\} \cup \delta_i ) ) > w_i $ preserves all MWSS of $G|_{I}$. 
When we discard a vertex that is not in any MWSS, the VFA remains equal to the stability number.
Also, if our algorithm is following an optimal trajectory at every iteration, $\calV(I)$ preserves the stability number of $G|_{I}$, and Algorithm~\ref{alg:look ahead} outputs an MWSS when it terminates.
Hence, with any VFA, our goal is to stay on an optimal trajectory. 

\subsection{VFAs from LP/SDP Relaxations}

Consider the primal-dual pair of LP relaxations \eqref{LP-P}, \eqref{LP-D}.
Given a feasible $\mu$ for~\eqref{LP-D}, we define a VFA 
\begin{equation}
\label{eq:V_LP}
\calV_{\LP}(S):=\sum_{C\in \mathcal{C}(G),C\cap S\neq\emptyset}\mu_C, \quad S\subseteq N. 
\end{equation}
The following lemmas verify that $\calV_{\LP}$, $\calV_{\SDP}$ are VFAs that are tight when the graph is perfect.
\begin{lemma}
\label{lem:V_LP_is_VFA}
    If $\mu$ is feasible for \eqref{LP-D}, $\calV_{\LP}$
    satisfies the conditions from Lemma~\ref{lem:VFAconditions},
    and hence is a VFA.
    If $\mu$ is optimal for \eqref{LP-D} and $G$ is perfect, $\calV_{\LP}$ is a tight VFA.
\end{lemma}
\begin{proof}
    We verify the four conditions from Lemma~\ref{lem:VFAconditions}.
        (i) Follows from the definition.
        (ii) Follows from the definition and the fact that $\mu \geq 0$.
        (iii) $\calV_{\LP}(i) = \sum_{C \ni i} \mu_C$,
        which is greater than or equal to $w_i$ since $\mu$ is feasible to~\eqref{LP-D}.
        (iv) If $I,J$ share no edges, a clique $C\in \mathcal{C}(G)$ can intersect at most one of $I, J$.
        Then
        $\calV_{\LP}(I \cup J) = \calV_{\LP}(I) + \calV_{\LP}(J)$ by definition.
    The tightness for perfect graphs follows from~\eqref{relation:perfect graph} and the optimality of $\mu$.
\end{proof}

For the primal-dual pair of SDP relaxations \eqref{SDP-P}, \eqref{SDP-D},
given a feasible $(t,q,Q)$ for~\eqref{SDP-D}, consider $\calV_{\SDP}$ from \eqref{eq:V_SDP}.
The following lemma provides some alternative forms of~$\calV_{\SDP}$, which are useful for analysis and computation.
\begin{lemma}
\label{lem:comp equiv}
    Let ${q} \in \R^N,{Q} \in \symm_+^{N}$, and define $\calV_{\SDP}$ with \eqref{eq:V_SDP}. 
    For $\emptyset \neq S \subset N$, 
    \begin{align*}
        \calV_{\SDP}(S)
        \quad:=\quad
        q_S^\top\; Q_S^{\dagger}\; q_S
        \quad=\quad
        \min_{t\in\R}
            \left\{ t : 
            \begin{psmallmatrix} t&q_S^\top\\ q_S&Q_S \end{psmallmatrix}
            \succeq 0 \right\}
        \quad=\quad
        -\min_{y\in\R^S}(y^\top Q_S y - 2 q_S^\top y).
    \end{align*}
\end{lemma}
\begin{proof}
    We first consider the SDP formulation.
    By taking the Schur complement, the PSD constraint is equivalent to $t\geq 0, tQ_S\succeq q_{S}q_S^\top$. 
    Therefore, the SDP is feasible if and only if $Q_S\succeq 0$ and $q_S$ is in the range of $Q_S$ by~\cite[Theorem 1]{pseudo-inv}.
    Assume that this is the case, and we show $t^*:=q_S^\top  Q_S^{\dagger} q_S$ is feasible, i.e., $t^*Q_S\succeq {q_S}q_S^\top$.
    By applying singular value decomposition and changing the basis, we may assume $Q_S=\Diag(a_1,\ldots,a_r,0,\ldots,0)$ with $a_i>0$ and ${q_S}=(b_1,\ldots,b_r,0,\ldots,0)$ with $b_i\geq 0$, so ${q_S}$ lies in the range of $Q_S$. 
    Then $Q_S^\dagger=\Diag(a_1^{-1},\ldots,a_r^{-1},0,\ldots,0)$. For any vector $x\in\R^{n}$, we have,
    \begin{align*}
        x^\top(t^*Q_S-{q_S}q_S^\top)x=(q_S^\top Q_S^\dagger {q_S})(x^\top Q_Sx)-(q_S^\top x)^2
        =\biggl(\sum_{i=1}^{r}b_i^2 a_i^{-1}\biggr)\biggl(\sum_{i=1}^{r}x_i^2 a_i\biggr)-\biggl(\sum_{i=1}^r b_ix_i\biggr)^2,
    \end{align*}
    which is nonnegative by Cauchy-Schwarz inequality. Hence, $t^*$ is feasible. 
    Moreover, when $x=Q_S^\dagger q_S$, the above inequality becomes equality, so $t^*$ is the optimum. 
    Hence, $
        q_S^\top  Q_S^{\dagger} q_S
        =
        \min_{t\in\R}
            \left\{ t : 
            \begin{psmallmatrix} t&q_S^\top\\ q_S&Q_S \end{psmallmatrix}
            \succeq 0 \right\}.$
    
    For the other equivalence, since $Q_S$ is PSD, the optimum is obtained at $y$ for $ Q_S y=q_S $,
    which is satisfied when $y=Q_S^\dagger q_S$, as we show above. 
    Then
    $y^\top Q_S y-2q_S^\top y=-q_S^\top Q^\dagger q_S$,
    which implies $-\min_{y}(y^\top Q_S y-2q_S^\top y)=q_S^\top Q_S^\dagger q_S$.
\end{proof}

We proceed to show that $\calV_{\SDP}$ is a VFA.
\begin{lemma}
\label{lem:VFA prop}
If $(t,q,Q)$ is feasible for~\eqref{SDP-D},  $\calV_{\SDP}$ is a VFA.
    If $(t,q,Q)$ is optimal for~\eqref{SDP-D} and $G$ is perfect, $\calV_{\SDP}$ is a tight VFA.
\end{lemma}
\begin{proof}
    We proceed to verify the four conditions from Lemma~\ref{lem:VFAconditions}.
    \begin{enumerate}[leftmargin=*,label=(\roman*)]
        \item Follows from the definition of $\calV_{\SDP}$.
        \item Let $I\subset J$ and let $t_J = \calV_{\SDP}(J)$.
        By the previous lemma, we have that $\begin{psmallmatrix}
            t_J&{q}_J^\top\\
            {q}_J&{Q}_J
        \end{psmallmatrix}\succeq 0$. 
        Since a principal sub-matrix of a PSD matrix is also PSD, we have $\begin{psmallmatrix}
            t_J &{q}_I^\top\\
            {q}_I&{Q}_I
        \end{psmallmatrix}\succeq 0$, and hence $\calV_{\SDP}(I)\leq t_J=\calV_{\SDP}(J)$ by Lemma~\ref{lem:comp equiv}. 
        \item Suppose there is no edge between $I,J$; then the matrix ${Q}_{I\cup J}$ is a block diagonal matrix, i.e. ${Q}_{I\cup J}=\Diag({Q}_I,{Q}_J)$. 
        Then its pseudo-inverse is also block diagonal, ${Q}^\dagger_{I\cup J}=\Diag({Q}^\dagger_{I},{Q}^\dagger_{J})$. 
        Hence,
        \[\calV_{\SDP}(I\cup J)={q}^\top_{I\cup J}{Q}^\dagger_{I\cup J}{q}_{I\cup J}={q}_I^\top {Q}^\dagger_{I}{q}_I+{q}_J^\top {Q}^\dagger_{J}{q}_J=\calV_{\SDP}(I)+\calV_{\SDP}(J).\]
        \item We have ${Q}_{ii}=-2{q}_i-w_i$, by feasibility in \eqref{SDP-D}. If ${Q}_{ii}=0$, then again by feasibility, ${q}_i=0$, $w_i=-2{q}_i-{Q}_{ii}=0$, and $\calV_{\SDP}(i)=0\geq w_i=0$.
        If $Q_{ii}>0$, then by the definition,
        \[
            \calV_{\SDP}(i):={Q}_{ii}^{-1}{q}_i^2=\frac{1}{4}{Q}_{ii}^{-1}({Q}_{ii}+w_i)^2\implies
            {Q}_{ii}(\calV_{\SDP}(i)-w_i)=\frac{1}{4}({Q}_{ii}-w_i)^2\geq 0,
        \]
        which implies $\calV_{\SDP}(i)\geq w_i$.
    \end{enumerate}
    The tightness follows from~\eqref{relation:perfect graph} and the optimality of $(t,q,Q)$.
 \end{proof}

\subsection{Statement of main results}

Our main contribution is to show that Algorithm~\ref{alg:look ahead}, applied with either $\calV_{\LP}$ or $\calV_{\SDP}$, finds an MWSS for several important subclasses of perfect graphs. 
In Section~\ref{sec:LPVFA}, we provide a combinatorial interpretation of Algorithm~\ref{alg:look ahead} with $\calV_{\LP}$, to prove the optimality of the returned stable set.
Subsequently, in Section~\ref{sec:SDPVFA} we analyze Algorithm~\ref{alg:look ahead} with $\calV_{\SDP}$, and give a modification that efficiently outputs all MWSS for unipolar, chordal and co-chordal graphs.

\begin{definition}
\label{defn:subclasses_of_perfect}
A graph $G$ is \emph{unipolar} if its vertex set $N$ can be partitioned as $N = A \cup \overline{{A}}$, such that the graph $G|_A$, called the \emph{center}, is complete, and the connected components of $G|_{\overline{A}}$, called \emph{clusters}, are also complete.
A graph is \emph{co-unipolar} if its complement is unipolar.
A \emph{generalized split graph} is a graph that is either unipolar or co-unipolar. 

A graph is \emph{chordal} if it does not contain induced cycles of length at least four.  
A graph is \emph{co-chordal} if its complement is chordal.
\end{definition}
%
Almost all perfect graphs are generalized split graphs in an asymptotic sense \cite{Prömel_Steger_1992}. 
We now state the main results: Algorithm~\ref{alg:look ahead} returns an MWSS for the subclasses of perfect graphs in Definition~\ref{defn:subclasses_of_perfect},  while~\PII~may need to be run twice only for co-unipolar graphs.
The proofs are in Subsection~\ref{subsec:look-ahead} and Section~\ref{sec:SDPVFA}, respectively.
\begin{theorem}
\label{thm:main_LP}
    Let $x^*$, $\mu^*$ respectively be optimal solutions of \eqref{LP-P} and \eqref{LP-D}, 
    both in the relative interior of the optimal face of their respective feasible regions,
    and let $\calV_{\LP}$ be the VFA constructed from $\mu^*$ via \eqref{eq:V_LP}.
    Consider Algorithm~\ref{alg:look ahead} executed with $\calV_{\LP}$.
    \begin{enumerate}[leftmargin=*,label=(\alph*)]
        \item\label{thm:main:unipolar}  If $G$ is unipolar, chordal, or co-chordal, then Algorithm~\ref{alg:look ahead} returns an MWSS.
        \item\label{thm:main:counipolar}
        Assume that $G$ is co-unipolar.
        Let $i,j$ be any adjacent vertices remaining after \PI.
        Then either
        Algorithm~\ref{alg:look ahead} returns an MWSS when \PII starts by selecting~$i$
        or it returns an MWSS when \PII starts by selecting~$j$.
    \end{enumerate}
\end{theorem}

\begin{theorem}
\label{thm:main_SDP}
    Let $(x^*,X^*)$, $(t^*,q^*,Q^*)$ respectively be optimal solutions of \eqref{SDP-P} and \eqref{SDP-D}, both in the relative interior of the optimal face of their respective feasible regions, and let $\calV_{\SDP}$ be the VFA constructed from $(t^*,q^*,Q^*)$ via \eqref{eq:V_SDP}. 
    Then Algorithm~\ref{alg:look ahead} executed with $\calV_{\SDP}$ returns an MWSS under the same assumptions as in Theorem~\ref{thm:main_LP}.
\end{theorem}

\begin{corollary}
\label{cor:Time Complexity}
Algorithm \ref{alg:look ahead} executes $\mathcal{O}(n^2)$ VFA evaluations.
Therefore, under the conditions of Theorem~\ref{thm:main_SDP}, it returns an MWSS in polynomial time.
\end{corollary}
\begin{proof}
In every iteration, $\mathcal{O}(n)$ VFA evaluations are required. 
    In the worst case, each iteration requires a look-ahead, and 
    there can be at most $n$ look-ahead steps.
    So there are at most $\mathcal{O}(n)$ iterations, and the overall number of VFA evaluations is $\mathcal{O}(n^2)$.  
\end{proof}

Above we account for the number of VFA evaluations, and distinguish this complexity from the SDP solve time and the complexity of a single VFA evaluation, as the latter two depend on user choice. 
In particular, for computing the VFA, Lemma \ref{lem:comp equiv} gives three equivalent forms for $\calV_{\SDP}$ that result in different complexities. For example, computing a pseudo-inverse requires $\mathcal{O}(n^3)$ time. 
In practice, we observe that applying gradient descent to the quadratic programming form is more efficient, but not as numerically stable.

The assumption that the optimal dual solution is in the relative interior of the optimal face cannot be relaxed; the algorithm may return a suboptimal stable set otherwise. 
Furthermore, the look-ahead is also necessary to guarantee the algorithm's success.
In Section~\ref{sec:LPVFA}, we discuss how Algorithm~\ref{alg:look ahead} may fail on the graph in Figure \ref{fig:c-fig-0} when either of the assumptions fails.

\section{Analysis via the LP VFA}\label{sec:LPVFA}

In this section, we analyze Algorithm~\ref{alg:look ahead} with $\calV_{\LP}$. 
We first give a combinatorial interpretation, which provides intuition about how and why a vertex is discarded.
Then we discuss why the look-ahead is necessary, and provide an example.
After that, we prove Algorithm~\ref{alg:look ahead} with $\calV_{\LP}$ returns an MWSS for generalized split, chordal, and co-chordal graphs.
Finally, we give another example showing that even with look-ahead, Algorithm~\ref{alg:look ahead} may fail to return an optimal set on some perfect graphs.

In this section, for an iteration of \PII, we let $S$ denote the set of selected vertices, $I$ denote the set of remaining vertices at the beginning of the iteration, and $I'$ be the set of vertices before step~\ref{alg_step:lookahead start}.

\subsection{Combinatorial Interpretation}
\label{sec:select and discard}

Let $G=(N,E)$ be a perfect graph. 
By~\eqref{relation:perfect graph}, one can obtain the incidence vector of an MWSS by solving for an optimal extreme point of~\eqref{LP-P}. 
The number of cliques in $G$ may grow exponentially, so it is expensive to solve such LPs directly.
Nonetheless, we can still extract a useful combinatorial interpretation.

Let $(x^*,\mu^*)$ be a pair of strictly complementary optimal solutions of~\eqref{LP-P} and~\eqref{LP-D}.
Note that $x^*\in\STAB(G)$ since the graph is perfect. 
Furthermore, strict complementarity for the LP is equivalent to being in the relative interior of the optimal face.
Therefore,
\begin{align*}
\sum_{C\ni i}\mu_{C}^{*}>w_i
\iff
x_i^*=0
\iff
&i\text{ is not in any MWSS of }G,
\\ 
\mu^*_C>0
\iff
\sum_{i\in C}x^*_i=1
\iff
&\text{ every MWSS of }G\text{ contains a vertex in }C.
\end{align*}
Recall that \PI discards all vertices $i$ with $x_i^* = 0$.
By the first equivalence above, it preserves all MWSS, ensuring that each remaining vertex belongs to at least one MWSS.
The second equivalence motivates the following definition.
\begin{definition}
    \label{defn:ess clique}
Given a graph $G$, let $\mathcal{C}(G)$ be the set of its cliques. 
For $C\in\mathcal{C}(G)$, if every MWSS of $G$ contains a vertex in $C$, we say $C$ is an \emph{essential clique} of $G$ or $C$ is \emph{essential}. In addition, if every vertex in $C$ belongs to an MWSS of $G$, 
$C$ is \emph{strictly essential}.
\end{definition}

Combining the two equivalences above, we provide a combinatorial interpretation of~\eqref{eq:valuefun} in terms of essential cliques and MWSS.
For $I\subseteq N$, we have 
\begin{align*}
    \calV_{\LP}(I)-\calV_{\LP}(I\setminus(\{i\}\cup\delta_i))&=
    [\calV_{\LP}(I\setminus\delta_i)-\calV_{\LP}(I\setminus(\{i\}\cup\delta_i))]
    +[\calV_{\LP}(I)- \calV_{\LP}(I\setminus\delta_i)]
    \\
    &=
    \sum_{C\ni i}\mu^*_C
    + \sum_{C\subseteq \delta_i \cup (N \setminus I), C \cap I \cap \delta_i \neq \emptyset}\mu^*_C
    \geq w_i + 0.
\end{align*}
Moreover, for every $I\subseteq N$,
\begin{align*}
    \calV_{\LP}(I)-\calV_{\LP}(I\setminus(\{i\}\cup\delta_i))&=w_i
    \iff
    \begin{matrix}
    i\text{ belongs to some MWSS of }G\text{ and }\\
    C\text{ is not essential }\forall C \subseteq \delta_i \cup (N \setminus I), C \cap I \cap \delta_i \neq \emptyset.
    \end{matrix}
\end{align*}
Notice that the equivalences only provide direct links to the MWSS of $G$ but not necessaryly $G|_I$.
In particular, there may exist $i,I$ with $\calV_{\LP}(I)-\calV_{\LP}(I\setminus(\{i\}\cup\delta_i)) 
= w_i$, but where $i$ is not in any MWSS of $G \vert_I $. 
Next, we characterize how essential cliques translate to induced subgraphs.

\begin{lemma}
\label{lem:ess_clique_shrink}
Let $I\subsetneq N$.
Suppose Algorithm~\ref{alg:look ahead} is on an optimal trajectory
and reaches $G|_I$ during some iteration.
For every essential clique $C$ of $G$, $C\cap I$ is an essential clique of $G|_I$ if it is non-empty.
\end{lemma}
In the lemma, it is necessary to assume we are on an optimal trajectory; otherwise, $C\cap I$ might not contain any vertex in an MWSS of $G|_I$.
\begin{proof}
Let $S$ be the set of selected vertices by Algorithm~\ref{alg:look ahead}.
Note that $C\cap S=\emptyset$, otherwise $C\cap I=\emptyset$, a contradiction.
Since $C$ is essential in $G$ and $S$ is on an optimal trajectory, every MWSS of $G|_I$ should contain a vertex in $C\cap I$ by the third property of Lemma~\ref{lem:fund_VFA_prop}.   
\end{proof}

Recall that after \PI every vertex belongs to an MWSS.
Hence, after \PI,
\begin{align*}
    \calV_{\LP}(I)-\calV_{\LP}(I\setminus(\{i\}\cup\delta_i))>w_i 
\quad\iff\quad \exists C\text{ essential}, C\subseteq \delta_i\cup(N \setminus I), C\cap I \cap \delta_i\neq\emptyset.
\end{align*}
For any such $C$, if one is on an optimal trajectory, $C\cap I\subseteq \delta_i\cap I$ is an essential clique of $G|_I$ by Lemma~\ref{lem:ess_clique_shrink}, which implies that 
$i$ does not belong to any MWSS of $G|_I$ and should be discarded.
With this interpretation, after \PI, Algorithm~\ref{alg:look ahead} with $\calV_{\LP}$ ``shrinks'' essential cliques as it discards and selects vertices. 

\subsection{Algorithm without Look-Ahead}

In this subsection, we specify a graph sub-structure that can cause the algorithm to fail without a look-ahead, and provide an example. 
We start with some preliminary results.
\begin{lemma}
\label{lem:look ahead}
Consider an iteration of Algorithm~\ref{alg:look ahead} that starts with $S$ and $I$, where $S$ is on an optimal trajectory. 
With $i$ being selected, if $\calV(I)>\calV(I')+w_i$ at step~\ref{alg_step:lookahead start}, $i$ is not in any MWSS of $G|_I$.
\end{lemma}
\begin{proof}
    Since $S$ is on an optimal trajectory and $i$ was not previously discarded, by Lemma~\ref{lem:fund_VFA_prop} we have 
    \(\alpha(I)=\calV(I)=\calV(I\setminus(\{i\}\cup\delta_i))+w_i>\calV(I')+w_i\geq \alpha(I')+w_i.\)
    If $\alpha(I)>\alpha(I\setminus(\{i\}\cup\delta_i))+w_i$,  $i$ is not in any MWSS of $G|_I$.
    Otherwise, $\alpha(I\setminus(\{i\}\cup\delta_i))+w_i=\alpha(I)=\calV(I\setminus(\{i\}\cup\delta_i))+w_i$, and then by the first and second property of Lemma~\ref{lem:fund_VFA_prop}, $\calV(I\setminus(\{i\}\cup\delta_i))=\calV(I')$, a contradiction.
 \end{proof}

The following proposition provides a necessary and sufficient condition for Algorithm~\ref{alg:look ahead} to return an MWSS for any graph $G$.
\begin{proposition}
    \label{prop:weight_hold}
    Let $G$ be a graph and $\calV$ be a tight VFA. 
    Algorithm~\ref{alg:look ahead} returns an MWSS of $G$ if and only if at the end of every iteration of~\PII~in which we select a vertex $i$, we have $\calV(I)=\calV(I')+w_i$. 
\end{proposition}
\begin{proof}
    $(\impliedby)$ For every iteration, if $S$ is on an optimal trajectory, Lemma~\ref{lem:look ahead} ensures that the look-ahead only discards suboptimal vertices; we can assume Algorithm~\ref{alg:look ahead} terminates after $K$ iterations of \PII~without applying the look-ahead.
    Let $I^t$ be the set of available vertices at the beginning of iteration $t$, where $I^1=N$, and without loss of generality, let $t$ be the vertex selected during iteration $t$.
    Then  
    \(\alpha(N)=\calV(I^1)=\calV(I^2)+w_1=\cdots=\calV(\emptyset)+w_{K}+\cdots+w_1.\)
    Since $\{1,\ldots,K\}$ is a stable set, it is an MWSS of $G$.
    
    $(\implies)$ If $\calV(I^t)>\calV(I^{t+1})+w_t$ for some $t$, then 
    \(\alpha(N)>\calV(\emptyset)+w_K+\cdots+w_1,\)
    which implies $\{1,\ldots,K\}$ is not an MWSS of $G$. 
\end{proof}

We emphasize that Proposition~\ref{prop:weight_hold} requires the equality to hold for \emph{every iteration} in which we add a vertex to $S$.
Thus, for Algorithm~\ref{alg:look ahead} to fail, $\calV(I)>\calV(I')+w_i$ in some iteration; 
we call such $i\in I$ a \emph{bad choice} of this iteration or of $I$. 
Indeed, there may be an iteration where $\calV(I) = \calV(I') + w_i$, but $\alpha(I) > \alpha(I') + w_i$, which means $i$ is not an optimal choice but the algorithm cannot detect it. 
The proposition only guarantees that at some iteration later in the process, all remaining vertices will be bad choices; at that point we realize we made a suboptimal choice, but we will not know which vertex it was.

With respect to $\calV_{\LP}$, a bad choice $i$ is equivalent to $\calV_{\LP}(I)>\calV_{\LP}(I')+w_i=\calV_{\LP}(I'\cup\{i\})$.
By the definition of $\calV_{\LP}$, we have 
\[
\calV_{\LP}(I)-\calV_{\LP}(I'\cup\{i\})=\sum_{C\in \mathcal{C}(G),C\cap (I'\cup\{i\})=\emptyset, C\cap I\neq \emptyset}\mu^*_C>0.
\]
In other words, there is an essential clique $C\in\mathcal{C}(G)$ with $C\cap (I'\cup\{i\})=\emptyset$ and $ C \cap I \neq \emptyset$. 
If one is on an optimal trajectory, this implies that $ C \cap I $ is an essential clique of $G|_I$ that would be discarded by selecting $i$. 
However, $ C \cap I \not\subseteq \delta_i \cap I $ (otherwise the algorithm would discard $i$ in the previous iteration). 
We conclude that for the algorithm to fail, after selecting $i$ and discarding $\delta_i$, every vertex of $C \cap I\setminus(\{i\}\cup\delta_i) \neq \emptyset$ contains an essential clique in its neighbor set (checked in step \ref{alg_step:discard}), so $C$ is completely discarded.
Next we discuss a necessary sub-structure for such a bad choice to exist within Algorithm \ref{alg:look ahead}.

\begin{definition}
    A graph $G=(N,E)$ is a \emph{house} if $N=[2k]\cup \{v\}$, where $G|_{[2k]}$ is an even hole and there are $i,j\in[2k]$ such that $G|_{\{i,j,v\}}$ is a triangle.
    We call $v$ the \emph{top} of the~house.
\end{definition}

\begin{proposition}
\label{prop:house_structure}
    Consider any perfect graph $G=(N,E)$ and $\calV_{\LP}$ constructed from a strictly complementary solution of~\eqref{LP-D}.
    If in some iteration of Algorithm~\ref{alg:look ahead} there is a bad choice $i\in I$, $G|_I$ contains a house as an induced subgraph.
\end{proposition}
\begin{proof}
    At the beginning of some iteration, let $S$ be the selected set of vertices, $I$ be the set of remaining vertices.
    For the sake of contradiction, suppose that $v^*\in I$ is a bad choice.
    Let $I'$ be the set of vertices after selecting $v^*$ and step~\ref{alg_step:discard} of Algorithm~\ref{alg:look ahead}.
    As discussed above, $\calV_{\LP}(I)>\calV_{\LP}(I')+w_{v^*}$ means there is some $C\in\mathcal{C}(G)$ with $C \cap (I'\cup\{v^*\}) = \emptyset, C\cap I\neq \emptyset$.
    By Lemma~\ref{lem:ess_clique_shrink}, without loss of generality we restrict all the cliques we discuss onto $I$ ($C\gets C\cap I$).
    
    First, $C$ is not contained in the neighbor set of $v^*$; otherwise, $v^*$ would be discarded in a previous iteration. 
    Also, the vertices in $C$ are discarded sequentially, so at some point in steps~\ref{alg_step:select}-\ref{alg_step:discard}, there is only $i\in C$ remaining and there is an essential clique $V_1$ contained in the neighbor set of $i$. 
    However, $V_1$ is itself not an essential clique of $G|_I$, otherwise $i$ would have been discarded before choosing $v^*$.
    That is, there exists an essential clique $V_1'$ of $G|_I$ with 
    $V_1\subsetneq V_1'$, and all vertices in $V_1'\setminus V_1$ have been discarded.
    Let $z_1 \in V_1' \setminus V_1$.
    The argument for $i$ can be applied to $z_1$; some essential clique $V_2$ is contained in its neighbor set, and $V_2'$ can be found in the same way.
    This argument applies repeatedly until we reach some $V_{k+1}:= C_1 \setminus \delta_{v^*} \neq \emptyset$, where $C_1$ is an essential clique of $G|_I$ with $C_1 \cap \delta_{v^*} \neq \emptyset$.
    Define $S_V:=C_1\cap \delta_{v^*}$.
    This constitutes a chain with alternating discarded vertices and essential cliques,
    \(
    P_i := i \to V_1\to z_1\to V_2\to \dotsb \to z_k\to V_{k+1}\to S_V.
    \)
    Intuitively, when $v^*$ is selected, $\delta_{v^*}$ is discarded from $I$, so $V_{k+1}$ becomes an essential clique of $I\setminus(v^*\cup\delta_{v^*})$. 
    Then, vertices like $z_k$, which contain $V_{k+1}$ in their neighbor sets, are discarded; smaller cliques become essential, and the process repeats until $i$ is discarded. 

    Moreover, $i$ by itself is not an essential clique of $G|_I$ (otherwise we would discard all of its neighbors and select it). Therefore, there is some $j\in C\setminus\{i\}$; suppose this is the second-to-last vertex in $C$ that is discarded. 
    That is, before we discard $j$, $\{i,j\}$ is an essential clique.
    Also, without loss of generality, we assume that before $j$ is discarded, all other vertices in the remaining graph except for $i$ and $j$ do not contain an essential clique in their neighbor set, as the argument depends on the induced subgraph discussed below.
    The same argument for $i$ can be applied to $j$ to find another alternating chain of discarded vertices and essential cliques,
    \(
    P_j := j \to U_1\to y_1\to U_2\to \dotsb \to y_h\to U_{h+1}\to S_U,
    \)
    where $U_{h+1}:=C_2\setminus\delta_{v^*}\neq\emptyset$ and $C_2$ is an essential clique of $G|_I$ with $C_2\cap\delta_{v^*}\neq\emptyset$.
    Unlike $P_i$, which has length at least one, $P_j$ may have length $0$, that is, $j\in S_U$ is possible; see Example~\ref{ex:house-ct} below.
    
    Consider picking a vertex from each of $V_\ell,U_g$; then $P_i,P_j$ are even paths, paths with an odd number of edges.  
    Similarly, for any $\ell\geq 1$, $i\to V_1\to z_1\to\ldots\to V_\ell$ is an odd path  and $i\to V_1\to z_1\to\ldots\to z_\ell$ is an even path. The same argument applies to $j$;
    see Figure \ref{fig:house} for an example.
    For convenience, let $z_0 := i $, $y_0 := j$. 
\begin{claim}
    There is no edge between $V_1$ and $j$ (respectively, $U_1$ and $i$).
\end{claim}
\begin{proof}
    Since $\{i,j\}$ is essential, $u\in V_1$ being adjacent to $j$ implies that $\{i,j\}\subseteq \delta_u$, so $u$ would be discarded before.
    A similar argument applies to $U_1$ and $i$.
\end{proof}
    
\begin{claim}
For every pair $y_g, z_{\ell}$ with $\max\{g,\ell\}\geq 1$, we may assume without loss of generality that there is no edge between them. 
The same applies to any $u_g\in U_g, v_{\ell}\in V_\ell$.
\end{claim}
\begin{proof}
    Suppose not; then we replace $j$ by $y_g$ and $i$ by $z_\ell$ or by $u_g,v_\ell$ respectively, until no such pair exists.
    That is, $P_i, P_j$ become
\(
        P_i = z_\ell \to V_{\ell+1}\to  \dotsb \to z_k\to V_{k+1}\to S_V;~
        P_j = y_g \to U_{g+1}\to  \dotsb \to y_h\to U_{h+1}\to S_U\)
    or, 
\(
        P_i = v_\ell \to z_{\ell}\to  \dotsb \to z_k\to V_{k+1}\to S_V;~
        P_j = u_g \to y_{g}\to  \dotsb \to y_h\to U_{h+1}\to S_U
    \)
    respectively.
    So the paths $P_i, P_j$ are either both odd or even.
    Then the closed walk $v^*\to S_V\to v_{k+1}\to\dotsb z_\ell\to y_g\to\dotsb u_{h+1}\to S_U\to v^*$ is odd.
    For the rest of the proof, all we need is that no such adjacent pair exists except for $z_\ell$, $y_g$; it is not required that $z_\ell,y_g$ are essential.
\end{proof}
\begin{claim}
    Without loss of generality, we can assume that $y_g \notin V_{\ell}$ for any $ g, \ell \geq 1 $. The same applies for $z_{\ell}, U_g$.
\end{claim}
\begin{proof}
    Assume $y_g\in V_{\ell}$. 
    Then $y_g$ is adjancent to $z_{\ell-1}$, so the previous claim applies.
\end{proof}
If for some $g,\ell\geq 0$, $y_g=z_{\ell}$ or $U_g\cap V_{\ell}\neq\emptyset$, or there is an edge between $y_g$ and $V_{\ell}$, or between $z_{\ell}$ and $U_g$, then we call such a pair $(y_g,z_{\ell})$, ($U_g,V_{\ell})$, $(y_g,V_{\ell}), (z_{\ell},U_g)$ a \emph{knot}. 
We call the knot with $\min\{g,\ell\}$ minimized the \emph{first knot}.
\begin{claim}
Applying the claims above, $G|_I$ either contains a house as an induced subgraph or does not have a knot.
\end{claim}
\begin{proof}
    Suppose there is a knot. 
    We consider three cases for the first knot:
    \begin{enumerate}[leftmargin=*]
        \item The first knot $(y_g,z_{\ell})$ with $y_g=z_{\ell}$ has $\min\{g,\ell\}\geq 1$; then 
    \[y_0\to u_1\to y_1\to\ldots\to y_g=z_{\ell}\to v_\ell\to z_{\ell-1}\to v_{\ell-1}\to\cdots\to v_1\to z_0\to y_0\]
    has $2g+ 2\ell+1$ edges and forms an odd hole; or 
    \(v_\ell\to z_\ell=y_g\to u_g\to v_\ell\)
    forms a triangle when $i,j$ are replaced by $v_\ell,u_g$.
    Since the graph is perfect, there is no odd hole, so we consider the second case.
    
    Consider $z_{\ell-1},y_{g-1}$, suppose there is no edge between $z_{\ell-1}$, $u_g$ and between $y_{g-1},v_\ell$.
    Then if $\{z_{\ell-1},y_{g-1}\}\in E$, 
    $G|_{\{z_{\ell-1},v_\ell,z_\ell=y_g,u_g,y_{g-1}\}}$ is an induced house.
    If not, consider $v_{\ell-1}\in V_{\ell-1},u_{g-2}\in U_{g-2}$.
    Suppose they are adjacent, if there is no edge between $v_{\ell-1}, y_{g-1}$ or between $u_{g-1},z_{\ell-1}$, $G|_{\{v_{\ell-1},z_{\ell-1},v_\ell,z_\ell=y_g,u_g,y_{g-1},u_{g-1}\}}$ is an induced house; or we find a triangle and apply the above argument again.
    Suppose there is no edge between $V_{\ell-1}$ and $U_{g-1}$; consider $z_{\ell-2},u_{g-2}$ and apply the same argument, until we either find an induced house or reach $i,j$. 
    Then since there is no edge between $i, U_1$ and between $j,V_1$, $\{i,\ldots,v_{\ell-1},y_g,u_{g},\ldots,j\}$ is an induced subgraph.
    
    If such edges exist, say $\{z_{\ell-1},u_g\}\in E$, then if $\{z_{\ell-1},y_{g-1}\}\in E$, $\{z_{\ell-1},u_g,y_{g-1}\}$ forms a triangle, apply the above argument to it; otherwise, consider $\{u_g, y_{g-1},u_{g-1},v_{\ell-1},z_{\ell-1}\}$ for some $v_{\ell-1}\in V_{\ell_1}$, $u_{g-1}\in U_{g-1}$. 
    If $\{v_{\ell-1},u_{g-1}\}\in E$, then to avoid odd holes, either $\{z_{\ell-1},u_{g-1}\}\in E$ or $\{v_{\ell-1},y_{g-1}\}\in E$; without loss of generality, consider the first case: then $\{z_{\ell-1},u_{g-1},v_{\ell-1}\}$ forms a triangle, and the same argument above applies.
    If $\{v_{\ell-1},u_{g-1}\}\notin E$, then consider $z_{\ell-2},u_{g-2}$ and apply the same argument until reaching $\{i,j\}\in E$; then $\{i,\ldots,z_{\ell-1},u_g,u_{g-1},\ldots,j\}$ is an odd hole, contradiction.

\item The first knot $(U_g, V_\ell)$ has $u^* \in U_g \cap V_\ell$. 
Then the same argument from the previous case applies by replacing $z_\ell=y_g$ by $u^*$.
    
    \item The first knot is $(y_g, V_\ell)$ or $(z_\ell, U_g)$. Then the same argument from the first case applies by replacing $z_\ell=y_g$ by $y_g$ (or $z_\ell$ respectively).  \qed
    \end{enumerate}
\end{proof}
Now suppose $G|_I$ does not have a knot.
Consider
\[
\mathcal{H}:=v^*\to S_V\to V_{k+1}\to z_k\to\cdots\to z_0\to y_0\to\cdots\to y_h\to U_{h+1}\to S_U\to v^*,
\]
an odd cycle with at least five edges; 
since the graph is perfect, this cycle has a chord.
Chords separate the odd cycle into induced odd holes, triangles, and even holes.
However, since there is no edge between $V_1,j$ and between $U_1,i$, there is always a hole in $\mathcal{H}$. 
Since the graph is perfect, it is an even hole.
And since $\mathcal{H}$ is odd, there is always an induced triangle.
That is, there is always an induced subgraph of $\mathcal{H}$ which is a house.
\end{proof}
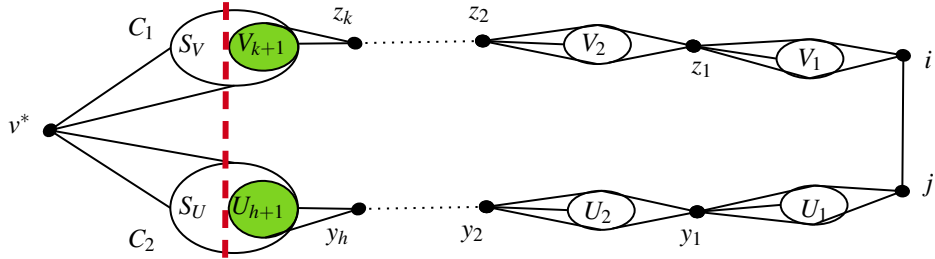
\begin{figure}[htb]
    \centering
\resizebox{0.6\columnwidth}{!}{    \begin{tikzpicture}[x=0.75pt,y=0.6pt,yscale=-1,xscale=1]

\draw    (570,247) -- (570.51,161.32) ;
\draw   (523.25,175.8) .. controls (513.95,175.74) and (506.48,170.17) .. (506.57,163.38) .. controls (506.66,156.58) and (514.28,151.13) .. (523.58,151.19) .. controls (532.89,151.26) and (540.36,156.82) .. (540.27,163.62) .. controls (540.18,170.41) and (532.56,175.87) .. (523.25,175.8) -- cycle ;
\draw    (523.58,151.19) -- (570.51,161.32) ;
\draw [shift={(570.51,161.32)}, rotate = 12.18] [color={rgb, 255:red, 0; green, 0; blue, 0 }  ][fill={rgb, 255:red, 0; green, 0; blue, 0 }  ][line width=0.75]      (0, 0) circle [x radius= 3.35, y radius= 3.35]   ;
\draw    (570.51,161.32) -- (523.25,175.8) ;
\draw    (506.57,163.38) -- (464.63,155.46) ;
\draw [shift={(464.63,155.46)}, rotate = 190.69] [color={rgb, 255:red, 0; green, 0; blue, 0 }  ][fill={rgb, 255:red, 0; green, 0; blue, 0 }  ][line width=0.75]      (0, 0) circle [x radius= 3.35, y radius= 3.35]   ;
\draw    (523.25,175.8) -- (464.63,155.46) ;
\draw [shift={(464.63,155.46)}, rotate = 199.14] [color={rgb, 255:red, 0; green, 0; blue, 0 }  ][fill={rgb, 255:red, 0; green, 0; blue, 0 }  ][line width=0.75]      (0, 0) circle [x radius= 3.35, y radius= 3.35]   ;
\draw    (523.58,151.19) -- (464.63,155.46) ;
\draw [shift={(464.63,155.46)}, rotate = 175.86] [color={rgb, 255:red, 0; green, 0; blue, 0 }  ][fill={rgb, 255:red, 0; green, 0; blue, 0 }  ][line width=0.75]      (0, 0) circle [x radius= 3.35, y radius= 3.35]   ;
\draw   (416.07,166.8) .. controls (406.77,166.74) and (399.3,161.17) .. (399.39,154.38) .. controls (399.48,147.58) and (407.1,142.13) .. (416.4,142.19) .. controls (425.71,142.26) and (433.18,147.82) .. (433.09,154.62) .. controls (432.99,161.42) and (425.38,166.87) .. (416.07,166.8) -- cycle ;
\draw    (416.4,142.19) -- (464.63,155.46) ;
\draw [shift={(464.63,155.46)}, rotate = 15.38] [color={rgb, 255:red, 0; green, 0; blue, 0 }  ][fill={rgb, 255:red, 0; green, 0; blue, 0 }  ][line width=0.75]      (0, 0) circle [x radius= 3.35, y radius= 3.35]   ;
\draw    (421.92,165.95) -- (464.63,155.46) ;
\draw    (399.39,154.38) -- (358.36,152.28) ;
\draw [shift={(358.36,152.28)}, rotate = 182.92] [color={rgb, 255:red, 0; green, 0; blue, 0 }  ][fill={rgb, 255:red, 0; green, 0; blue, 0 }  ][line width=0.75]      (0, 0) circle [x radius= 3.35, y radius= 3.35]   ;
\draw    (416.07,166.81) -- (358.36,152.28) ;
\draw [shift={(358.36,152.28)}, rotate = 194.12] [color={rgb, 255:red, 0; green, 0; blue, 0 }  ][fill={rgb, 255:red, 0; green, 0; blue, 0 }  ][line width=0.75]      (0, 0) circle [x radius= 3.35, y radius= 3.35]   ;
\draw    (416.4,142.19) -- (358.37,152.28) ;
\draw [shift={(358.37,152.28)}, rotate = 170.14] [color={rgb, 255:red, 0; green, 0; blue, 0 }  ][fill={rgb, 255:red, 0; green, 0; blue, 0 }  ][line width=0.75]      (0, 0) circle [x radius= 3.35, y radius= 3.35]   ;
\draw  [dash pattern={on 0.84pt off 2.51pt}]  (358.36,152.28) -- (293.88,153.64) ;
\draw  [fill={rgb, 255:red, 126; green, 211; blue, 33 }  ,fill opacity=1 ] (246.93,170.97) .. controls (237.31,170.59) and (229.97,163.56) .. (230.53,155.28) .. controls (231.09,147) and (239.34,140.6) .. (248.96,140.98) .. controls (258.57,141.37) and (265.92,148.39) .. (265.36,156.68) .. controls (264.79,164.96) and (256.55,171.36) .. (246.93,170.97) -- cycle ;
\draw    (248.96,140.98) -- (293.88,153.64) ;
\draw [shift={(293.88,153.64)}, rotate = 15.73] [color={rgb, 255:red, 0; green, 0; blue, 0 }  ][fill={rgb, 255:red, 0; green, 0; blue, 0 }  ][line width=0.75]      (0, 0) circle [x radius= 3.35, y radius= 3.35]   ;
\draw    (293.88,153.64) -- (265.54,153.97) ;
\draw   (233.14,180.45) .. controls (215.38,180.43) and (201,169.75) .. (201.03,156.6) .. controls (201.06,143.45) and (215.48,132.8) .. (233.24,132.82) .. controls (251.01,132.85) and (265.38,143.53) .. (265.36,156.68) .. controls (265.33,169.83) and (250.91,180.47) .. (233.14,180.45) -- cycle ;
\draw  [fill={rgb, 255:red, 126; green, 211; blue, 33 }  ,fill opacity=1 ] (246.78,240.54) .. controls (237.16,240.98) and (229.79,249.35) .. (230.31,259.23) .. controls (230.83,269.12) and (239.06,276.78) .. (248.68,276.34) .. controls (258.3,275.9) and (265.67,267.53) .. (265.15,257.65) .. controls (264.63,247.77) and (256.4,240.11) .. (246.78,240.54) -- cycle ;
\draw   (233.05,229.19) .. controls (215.28,229.17) and (200.86,241.88) .. (200.82,257.57) .. controls (200.79,273.27) and (215.16,286.01) .. (232.93,286.04) .. controls (250.69,286.06) and (265.12,273.35) .. (265.15,257.65) .. controls (265.18,241.96) and (250.81,229.21) .. (233.05,229.19) -- cycle ;
\draw    (200.82,257.58) -- (140.09,208.59) ;
\draw    (201.03,156.6) -- (140.09,208.59) ;
\draw [shift={(140.09,208.59)}, rotate = 139.53] [color={rgb, 255:red, 0; green, 0; blue, 0 }  ][fill={rgb, 255:red, 0; green, 0; blue, 0 }  ][line width=0.75]      (0, 0) circle [x radius= 3.35, y radius= 3.35]   ;
\draw    (233.04,229.19) -- (140.09,208.59) ;
\draw    (233.14,180.45) -- (140.09,208.59) ;
\draw [color={rgb, 255:red, 208; green, 2; blue, 27 }  ,draw opacity=1 ][line width=2.25]  [dash pattern={on 6.75pt off 4.5pt}]  (228.47,289) -- (228.9,127) ;
\draw   (525.33,268.1) .. controls (516.03,268.03) and (508.56,262.47) .. (508.65,255.67) .. controls (508.74,248.88) and (516.36,243.42) .. (525.66,243.49) .. controls (534.97,243.56) and (542.44,249.12) .. (542.35,255.92) .. controls (542.26,262.71) and (534.64,268.17) .. (525.33,268.1) -- cycle ;
\draw    (525.66,243.49) -- (570,247) ;
\draw [shift={(570,247)}, rotate = 4.53] [color={rgb, 255:red, 0; green, 0; blue, 0 }  ][fill={rgb, 255:red, 0; green, 0; blue, 0 }  ][line width=0.75]      (0, 0) circle [x radius= 3.35, y radius= 3.35]   ;
\draw    (570,247) -- (525.33,268.1) ;
\draw    (508.65,255.67) -- (466.55,260.06) ;
\draw [shift={(466.55,260.06)}, rotate = 174.05] [color={rgb, 255:red, 0; green, 0; blue, 0 }  ][fill={rgb, 255:red, 0; green, 0; blue, 0 }  ][line width=0.75]      (0, 0) circle [x radius= 3.35, y radius= 3.35]   ;
\draw    (525.33,268.1) -- (466.55,260.06) ;
\draw [shift={(466.55,260.06)}, rotate = 187.78] [color={rgb, 255:red, 0; green, 0; blue, 0 }  ][fill={rgb, 255:red, 0; green, 0; blue, 0 }  ][line width=0.75]      (0, 0) circle [x radius= 3.35, y radius= 3.35]   ;
\draw    (525.66,243.49) -- (466.55,260.06) ;
\draw [shift={(466.55,260.06)}, rotate = 164.34] [color={rgb, 255:red, 0; green, 0; blue, 0 }  ][fill={rgb, 255:red, 0; green, 0; blue, 0 }  ][line width=0.75]      (0, 0) circle [x radius= 3.35, y radius= 3.35]   ;
\draw   (417.99,271.41) .. controls (408.68,271.34) and (401.21,265.78) .. (401.3,258.98) .. controls (401.4,252.19) and (409.01,246.73) .. (418.32,246.8) .. controls (427.62,246.86) and (435.09,252.43) .. (435,259.22) .. controls (434.91,266.02) and (427.29,271.48) .. (417.99,271.41) -- cycle ;
\draw    (418.32,246.8) -- (466.55,260.06) ;
\draw [shift={(466.55,260.06)}, rotate = 15.38] [color={rgb, 255:red, 0; green, 0; blue, 0 }  ][fill={rgb, 255:red, 0; green, 0; blue, 0 }  ][line width=0.75]      (0, 0) circle [x radius= 3.35, y radius= 3.35]   ;
\draw    (423.84,270.56) -- (466.55,260.06) ;
\draw    (401.3,258.98) -- (360.28,256.89) ;
\draw [shift={(360.28,256.89)}, rotate = 182.92] [color={rgb, 255:red, 0; green, 0; blue, 0 }  ][fill={rgb, 255:red, 0; green, 0; blue, 0 }  ][line width=0.75]      (0, 0) circle [x radius= 3.35, y radius= 3.35]   ;
\draw    (417.99,271.41) -- (360.28,256.89) ;
\draw [shift={(360.28,256.89)}, rotate = 194.12] [color={rgb, 255:red, 0; green, 0; blue, 0 }  ][fill={rgb, 255:red, 0; green, 0; blue, 0 }  ][line width=0.75]      (0, 0) circle [x radius= 3.35, y radius= 3.35]   ;
\draw    (418.32,246.8) -- (360.28,256.89) ;
\draw [shift={(360.28,256.89)}, rotate = 170.14] [color={rgb, 255:red, 0; green, 0; blue, 0 }  ][fill={rgb, 255:red, 0; green, 0; blue, 0 }  ][line width=0.75]      (0, 0) circle [x radius= 3.35, y radius= 3.35]   ;
\draw  [dash pattern={on 0.84pt off 2.51pt}]  (360.28,256.89) -- (295.8,258.24) ;
\draw    (248.68,276.34) -- (295.8,258.24) ;
\draw [shift={(295.8,258.24)}, rotate = 338.99] [color={rgb, 255:red, 0; green, 0; blue, 0 }  ][fill={rgb, 255:red, 0; green, 0; blue, 0 }  ][line width=0.75]      (0, 0) circle [x radius= 3.35, y radius= 3.35]   ;
\draw    (295.8,258.24) -- (265.15,257.65) ;

\draw (516.11,155.57) node [anchor=north west][inner sep=0.75pt]    {$V_{1}$};
\draw (407.06,145.74) node [anchor=north west][inner sep=0.75pt]    {$V_{2}$};
\draw (462.2,163.24) node [anchor=north west][inner sep=0.75pt]    {$z_{1}$};
\draw (349.92,128.83) node [anchor=north west][inner sep=0.75pt]    {$z_{2}$};
\draw (579.66,156.91) node [anchor=north west][inner sep=0.75pt]    {$i$};
\draw (233.12,145.68) node [anchor=north west][inner sep=0.75pt]  [rotate=-358.69]  {$V_{k+1}$};
\draw (281.49,129.51) node [anchor=north west][inner sep=0.75pt]    {$z_{k}$};
\draw (203.68,146.34) node [anchor=north west][inner sep=0.75pt]    {$S_{V}$};
\draw (517.42,250.39) node [anchor=north west][inner sep=0.75pt]    {$U_{1}$};
\draw (408.19,252.19) node [anchor=north west][inner sep=0.75pt]    {$U_{2}$};
\draw (578.32,234.11) node [anchor=north west][inner sep=0.75pt]    {$j$};
\draw (456.58,269.31) node [anchor=north west][inner sep=0.75pt]    {$y_{1}$};
\draw (346.16,266.72) node [anchor=north west][inner sep=0.75pt]    {$y_{2}$};
\draw (278.12,268.77) node [anchor=north west][inner sep=0.75pt]    {$y_{h}$};
\draw (230.42,248.45) node [anchor=north west][inner sep=0.75pt]  [rotate=-359.7]  {$U_{h+1}$};
\draw (203.68,248.41) node [anchor=north west][inner sep=0.75pt]    {$S_{U}$};
\draw (117.55,196.51) node [anchor=north west][inner sep=0.75pt]  [rotate=-358.79]  {$v^{*}$};
\draw (178.54,135.34) node [anchor=north west][inner sep=0.75pt]  [rotate=-1.18]  {$C_{1}$};
\draw (178.54,269.38) node [anchor=north west][inner sep=0.75pt]    {$C_{2}$};

\end{tikzpicture}}
\caption{Paths from $i,j$ to $v^*$.}
    \label{fig:house}    
\end{figure}

Proposition~\ref{prop:house_structure} implies that Algorithm~\ref{alg:look ahead} with $\calV_{\LP}$ returns an MWSS on house-free perfect graphs without applying the look-ahead. 
And Proposition~\ref{prop:weight_hold} provides a necessary condition for the algorithm to possibly fail: in some iteration, $\calV(I)>\calV(I')+w_i$. 
As we indicate above, $\calV_{\LP}(I)>\calV_{\LP}(I')+w_i$ is equivalent to discarding all vertices of an essential clique. 
A house like the one in Figure \ref{fig:house} is required for that to happen.
\begin{corollary}
        Consider a triangle-free perfect graph (i.e.\ a bipartite graph) $G=(N,E)$, and $\calV_{\LP}$ constructed from a strictly complementary solution of~\eqref{LP-D}.
        Algorithm~\ref{alg:look ahead} returns an MWSS of $G$ without the look-ahead.
\end{corollary}
More generally, in the proof of Proposition~\ref{prop:house_structure}, the only perfect graph property we use is the fact that it does not contain an odd hole.
Therefore, Algorithm \ref{alg:look ahead} is guaranteed to return an MWSS if $\calV_{\LP}$ is tight and the graph does not contain a house or an odd hole.

For house-free graphs, since there is no bad choice for each iteration, Algorithm~\ref{alg:look ahead} returns an MWSS without the look-ahead.
However, without the look-ahead, Algorithm~\ref{alg:look ahead} can indeed fail to produce an MWSS if a perfect graph contains a house; the following example discusses this in more detail. However, we do not currently know whether the look-ahead is in fact necessary for generalized split, chordal or co-chordal graphs, as the instance in the example is not in any of these families.

\begin{example}
\label{ex:house-ct}
Consider the perfect graph $G$ in Figure \ref{fig:c-fig-0} with a cardinality objective, where the yellow cliques with $\mu_C^*=1$ are the essential cliques determined by a strictly complementary solution of~\eqref{LP-D};
in this case, $\mu^*$ is an extreme point and the unique dual optimal solution.
We show that this graph is perfect but not generalized split, chordal or co-chordal at the end of this section in Proposition \ref{prop:ct-ngsgchordalcochordal}.
Assume Algorithm~\ref{alg:look ahead} omits the look-ahead (steps \ref{alg_step:lookahead start}-\ref{alg_step:lookahead end}). 
Suppose vertex $10$ is selected; then its neighbors $9,11$ are discarded (see Figure \ref{fig:c-fig-1}).
Then vertex $8$ becomes essential, so $7$ is discarded and $\{3,4\}$ becomes essential.
The resulting graph is shown in Figure \ref{fig:c-fig-2}, where the essential clique $\{1,2\}$ is not strictly essential.
Without the look-ahead, the algorithm can't discard any remaining vertex; suppose $2$ is selected and its neighbors $1,3,6$ are discarded.
Vertices $4,5$ become essential and are also each others' neighbors; see Figure \ref{fig:c-fig-3}. One of the vertices must be discarded; if $5$ is discarded, the algorithm returns the stable set $S=\{2,4,8,10\}$. 
However, $\alpha = 5$;
there are $5$ essential cliques in $G$, but $S$ does not contain a vertex from $\{5,6\}$. 
    
Next, suppose we use the look-ahead; for the iteration starting at Figure \ref{fig:c-fig-2}, $I=\{1,\ldots,6\}$, $2$ is selected, and $I'=\{4\}$, as in Figure \ref{fig:c-fig-3}.
Then we have $3=\calV_{\LP}(I)>\calV_{\LP}(I')+w_2=2$, so Algorithm~\ref{alg:look ahead} goes back to $I$ and discards $2$, as shown in Figure \ref{fig:c-fig-4}; Algorithm~\ref{alg:look ahead} can select an arbitrary vertex, say $3$, and return an MWSS $\{1,3,5,8,10\}$, as shown in Figure \ref{fig:c-fig-5}.
    \begin{figure}[htb]
     \begin{subfigure}[t]{0.3\textwidth}
         \centering
         \includegraphics[width=\textwidth, height=2cm]{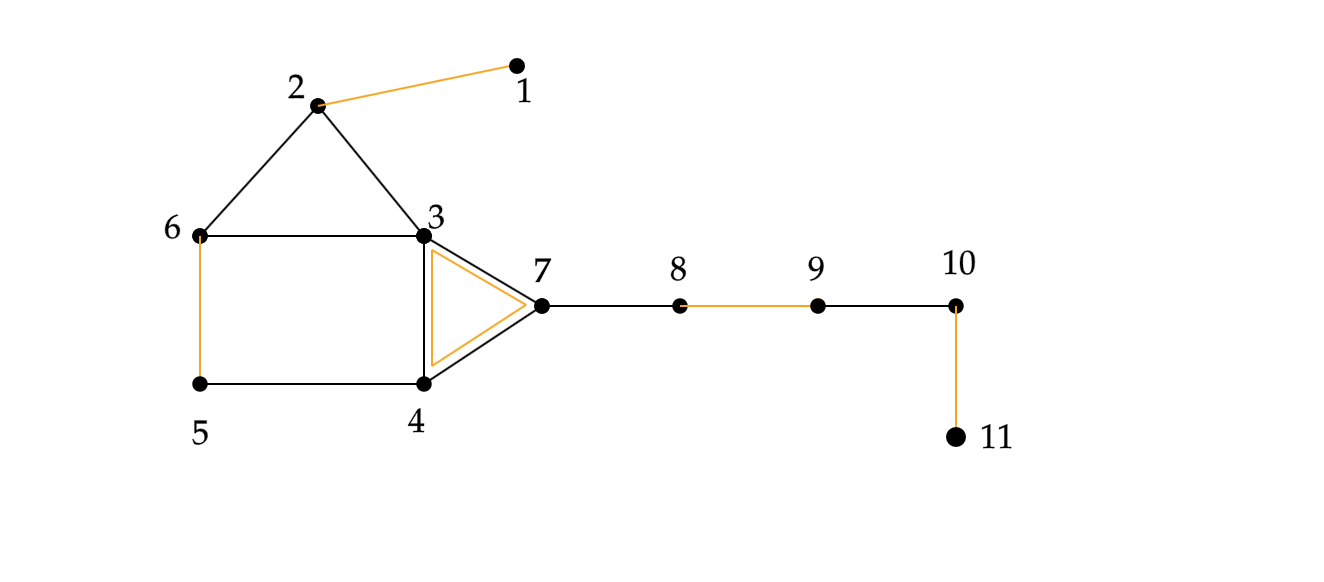}
         \caption{}
        \label{fig:c-fig-0}
     \end{subfigure}
     \begin{subfigure}[t]{0.3\textwidth}
         \centering
         \includegraphics[width=\textwidth, height=2cm]{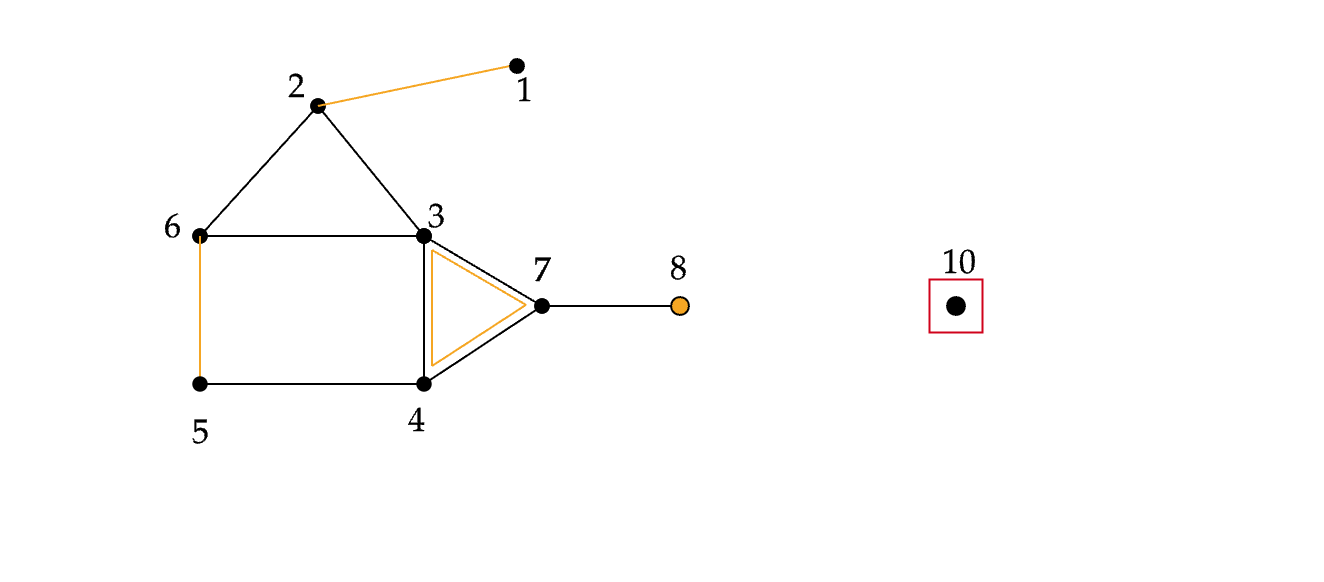}
         \caption{}
        \label{fig:c-fig-1}
     \end{subfigure}
     \begin{subfigure}[t]{0.3\textwidth}
         \centering
         \includegraphics[width=\textwidth, height=2cm]{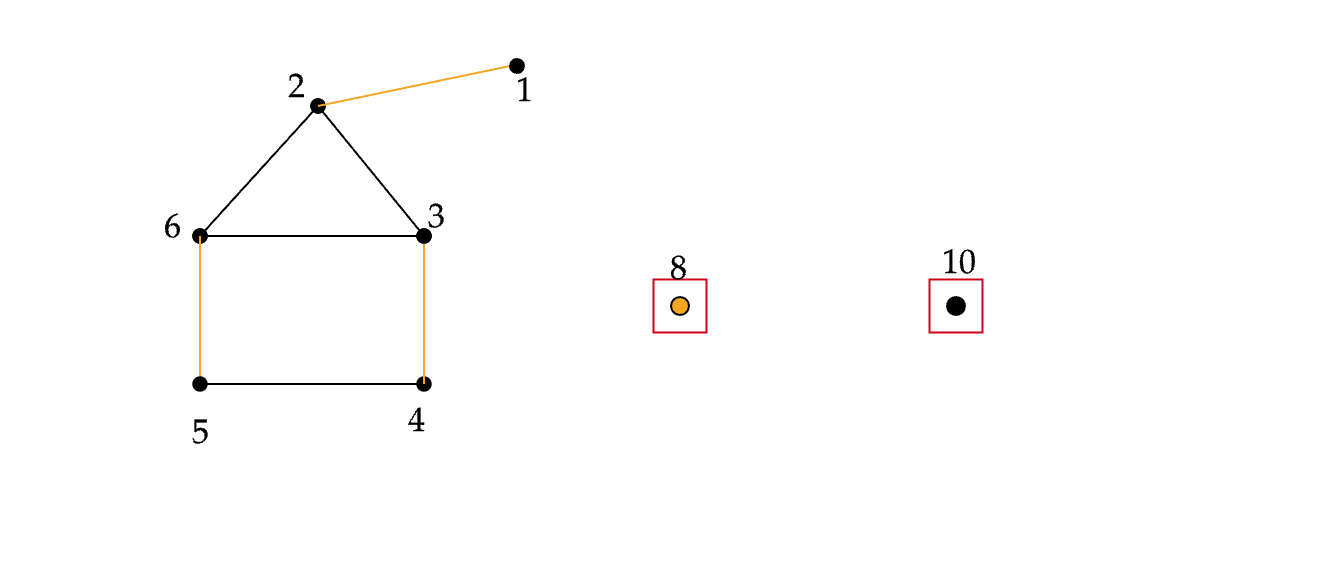}
         \caption{}
        \label{fig:c-fig-2}
     \end{subfigure}
     \begin{subfigure}[t]{0.3\textwidth}
         \centering
         \includegraphics[width=\textwidth, height=2cm]{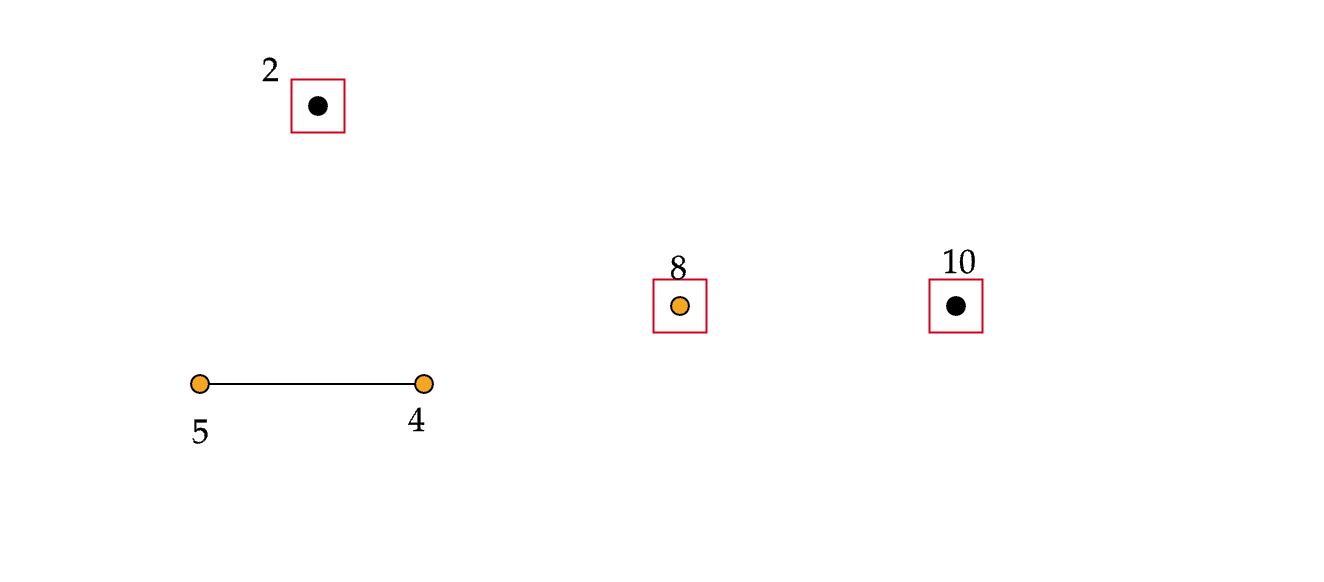}
                 \caption{}
        \label{fig:c-fig-3}
     \end{subfigure}
     \begin{subfigure}[t]{0.3\textwidth}
         \centering
         \includegraphics[width=\textwidth, height=2cm]{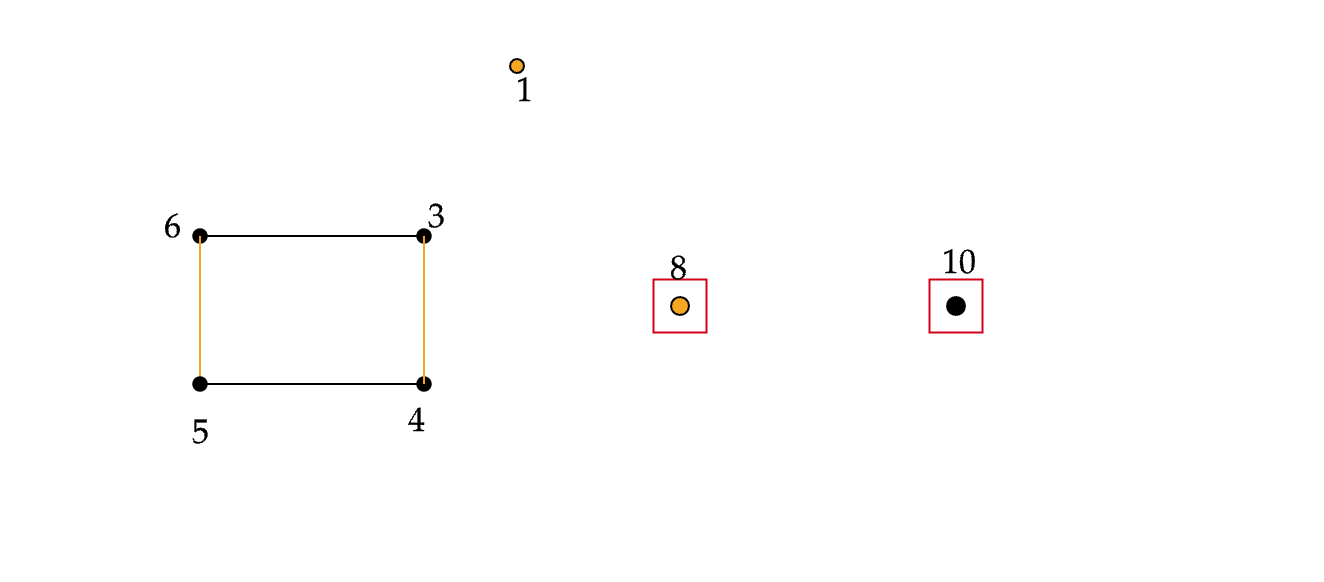}
         \caption{}
         \label{fig:c-fig-4}
     \end{subfigure}
     \begin{subfigure}[t]{0.3\textwidth}
         \centering
         \includegraphics[width=\textwidth, height=2cm]{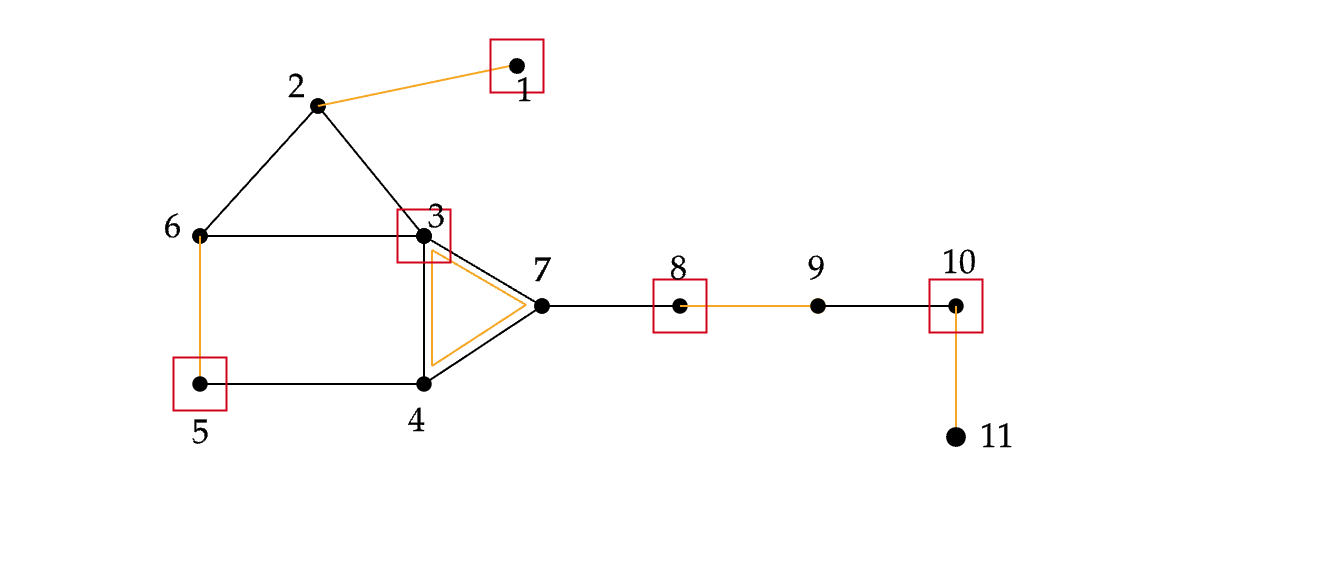}
            \caption{}
         \label{fig:c-fig-5}
     \end{subfigure}
        \centering
        \caption{}
    \end{figure}

In this example, the sub-graph induced by $\{1,\ldots,6\}$ is a house, and $2$ is a bad choice. 
In particular, $2$ and $10$ cannot be in the same MWSS, so $2$ must be discarded once $10$ is selected.
In the next section, we show that Algorithm~\ref{alg:look ahead} with $\calV_{\SDP}$ can discard any vertex that the algorithm discards with $\calV_{\LP}$ while on an optimal trajectory; therefore, a bad house is also necessary for it to fail.
Indeed, Algorithm~\ref{alg:look ahead} using $\calV_{\SDP}$ without the look-ahead can also fail on this instance, by selecting $2$ after $10$. 

The example also highlights the importance of using a solution from the relative interior of the optimal face to construct $\calV_{\LP}$. The three cliques identified in Figure \ref{fig:c-fig-2} are indeed essential for the sub-graph induced by $\{1, \dotsc, 6\}$, but they aren't the only essential cliques. The optimal dual solution that assigns unit weight to each of those cliques is optimal for this sub-graph, but it is an extreme point of \eqref{LP-D} and not in the relative interior of the optimal face. 
\end{example}

As indicated in Example~\ref{ex:house-ct}, Lemma~\ref{lem:look ahead} and the proof of Proposition~\ref{prop:weight_hold}, when $\alpha(I)>\alpha(I')+w_i$, $i$ is not on an optimal trajectory and should be discarded. 
In the following subsection we show that, for generalized split, chordal and co-chordal graphs, Algorithm~\ref{alg:look ahead} can detect and discard any such bad choice. 
    

\subsection{Look-Ahead}
\label{subsec:look-ahead}
We next prove Theorem \ref{thm:main_LP}, which guarantees the algorithm's performance on generalized split, chordal and co-chordal graphs.  
For co-unipolar graphs, Theorem~\ref{thm:main_LP} requires the algorithm to be run at most twice, starting from two neighbors. Intuitively, this ensures that at least one run of the algorithm starts from a member of one of the graph's clusters (and not its center). This ensures we obtain an MWSS, because after selecting this vertex, the remaining graph is bipartite, and we apply Proposition \ref{prop:house_structure}.

\begin{proof}[Proof of Theorem~\ref{thm:main_LP}\ref{thm:main:counipolar}]
Consider $G$ after \PI, which is still co-unipolar. 
For a co-unipolar graph $G$, we can separate $G$ into a center $A$ and clusters $B_1,\ldots,B_k$. 
Let $S=\{v_1,\ldots,v_t\}$ be the returned stable set and $v_i$ be the $i$-th selected vertex.
Without loss of generality, suppose $v_1\in B_1$; then all $B_2,\ldots, B_k$ and some vertices in $A$ are discarded, so $G|_{N\setminus (v_1\cup\delta_{v_1})}$ is a bipartite graph and $S$ is an MWSS by Proposition~\ref{prop:house_structure}. 

Now suppose $v_1\in A$;
without loss of generality, assume it has at least one neighbor $u$, which is in $B_1\cup\ldots \cup B_k$. 
If $S$ is an MWSS, we are done; otherwise, we run \PII~again, and select $u\in B_1\cup\ldots \cup B_k$ in the first iteration.
Then we are back in the previous case.   
\end{proof}

For unipolar graphs, we apply the next lemma, which shows that in every iteration of \PII~of Algorithm~\ref{alg:look ahead}, we can stay on an optimal trajectory.

\begin{lemma}
\label{lem:unipolar}
    Let $G$ be a unipolar graph.
    For each iteration of \PII of Algorithm~\ref{alg:look ahead} starting with the set of vertices $I$, there exists a vertex $u\in I$ such that 
    \(\calV_{\LP}(I)=\calV_{\LP}(I')+w_u,\)
    where $I'$ is the set of remaining vertices after selecting $u$ and step~\ref{alg_step:discard}.
\end{lemma}
\begin{proof}
    Consider $G$ after \PI; every vertex in $G$ belongs to an MWSS and $\calV_{\LP}$ is constructed from a strictly complementary optimal solution, so the first iteration of \PII~does not use the look-ahead. 
    Let $A_N$ be the center and $B_N^1,\ldots,B_N^k$ be the clusters of $G$.
    If in every iteration of \PII\ we have $\calV_{\LP}(N)=\alpha(N)=\calV_{\LP}(I)+\sum_{i\in S}w_i$, then an MWSS is returned.
    
    For the sake of contradiction, suppose that in some iteration starting with vertex set $I$, we have $\calV_{\LP}(I)>\calV_{LP}(I')+w_u$ for every $u\in I$, where $I'$ is the set of vertices remaining after $u$ is selected and vertices are discarded in step \ref{alg_step:discard}. 
    Let $A_I\subseteq A_N,B_I^1\subseteq B_N^1, B_I^2,\subseteq B_N^2,\ldots,B_I^k\subseteq B_N^k$ be the center and clusters of $G|_I$.
    \begin{claim}
        If $u\in A_I$, then $\calV_{\LP}(I)=\calV_{\LP}(I')+w_u$.
    \end{claim}
    \begin{proof}
        Suppose there is some $u\in A_I$ with $\calV_{\LP}(I)>\calV_{\LP}(I')+w_u$.
        Since $u$ has not been discarded, $ \calV_{\LP}(I) = \calV_{\LP}(I\setminus (u\cup\delta_u)) + w_u$. 
        In other words, $\calV_{\LP}(I\setminus (u\cup\delta_u))>\calV_{\LP}(I')$, which implies that some essential clique of $I$ in $I\setminus (u\cup\delta_u\cup I')$ is discarded.
        That is, after $u$ is selected and $\delta_u$ is discarded, there are two disjoint essential cliques, where each is contained in the other's neighbor set, and the algorithm discards one. 
        
        Let $C_1,C_2$ be these two essential cliques with $C_1 \cap C_2 = \emptyset $.
        Since $u\in A_I$, $C_1,C_2$ are not in $A_I$ and they belong to the same cluster, say $B_I^1$.
        Neither is an essential clique of $G \vert_I $ nor an essential clique of $G$ contained in $B_N^1$, otherwise, the other would have been previously discarded. 
        That is, there are two strict essential cliques $K_1,K_2$ of $G$ such that $C_1 \subsetneq K_1, C_2 \subsetneq K_2$ and $u \notin K_1\cup K_2$.
        
        Suppose $u$ is selected in the first iteration from $N$, and $\delta_u$ is discarded.
        Then consider $C'_1:=(B_N^1\setminus\delta_u)\cap K_1,C'_2:=(B_N^1\setminus\delta_u)\cap K_2$. 
        Without loss of generality, assume that vertices not in $C_1'\cup C_2'$ with an essential clique in their neighbor sets have been discarded.
        If $C'_1 \cap C'_2 = \emptyset $, these two essential cliques are each contained in the other's  neighbor set, so $\alpha(N) > \alpha(N\setminus (u\cup\delta_u))+w_u$, which implies $u$ is not an optimal choice, and contradicts the assumption that $u$ belongs to an MWSS.
        Hence, we assume there exists $v\in C'_1\cap C'_2$, which also implies $v\notin \delta_u$. 
        However, $v\notin C_1\cap C_2=\emptyset$.
        If $v \in C_1$, since $\{v\} \cup C_2\subseteq K_2$, $\{v\} \cup C_2$ is an essential clique instead of $C_2$ and we have a contradiction, so $v\notin C_1$; similarly, $v \notin C_2$.
        
        Hence, before $C_1,C_2$ become essential, $v$ is discarded, so it contains an essential clique $C$ in its neighbor. 
        Assume $v$ is discarded in the iteration starting with $J$, where $J\supseteq I$. 
        We consider three cases:
        
            Case 1: $C\subseteq A_J$: since $u\in A_J$, either $u\in C$ or $C\subseteq \delta_u$. For the latter case, $u$ is discarded in iteration $J$, so $u\notin I$, a contradiction.
            If $u\in C$, since $C\subseteq \delta_v$ then $v\in \delta_u$, 
            a contradiction.
            
            Case 2: $C\subseteq A_J\cup B_J^1 $ with $ C\cap A_J, C\cap B_J^1\neq \emptyset:$ there exists a strictly essential clique $K$ in $A_N\cup B_N^1$ such that $K\supseteq C$ and $v\notin K$. 
            Then $K\setminus \delta_u\neq\emptyset$, otherwise $u$ is not in any MWSS.
            Also, $u\notin K$; otherwise, since $u\notin C\subseteq \delta_v$,  $C\subseteq K\setminus \{u\}\subseteq \delta_u$, so either $C$ is discarded when $u$ is selected from $I$ then $v$ is not discarded, or $u$ is discarded together with $v$ during iteration starting with $J$, a contradiction. 
            Let $C':=(B_N^1\setminus \delta_u)\cap K\subseteq\delta_v$ by $v\in B^1_N$. 
            By the properties from above, $C'\neq\emptyset$ and $C'$ is an essential clique after $u$ is selected in the first iteration, so $v$ is discarded. 
            
            Case 3: $C\subseteq B_J^1$: there exists a strictly essential clique $K$ in $A_N\cup B_N^1$ such that $K\supseteq C$, $v\notin K$.
            If $K\cap A_N\neq\emptyset$, then this case is equivalent to the previous one.
            Suppose $K\subseteq B_N^1$; then since $v\in B_N^1$, we have $K\subseteq \delta_v$, $v$ would be discarded in \PI, and $v\notin C_1'\cup C_2'$, a contradiction.
        By the three cases above, $v$ is discarded before $C_1'$, $C_2'$ are each contained in the other's neighbor set. Since $v \in C'_1\cap C'_2$ is an arbitrary vertex, we can assume $C'_1\cap C'_2=\emptyset$. 
        Hence, $\calV_{\LP}(N)>\calV_{\LP}(N\setminus (u\cup\delta_u))+w_u$ and $u$ is not an optimal choice, a contradiction.  
\end{proof}

If every vertex $u$ in $I$ has $\calV_{\LP}(I)>\calV_{\LP}(I')+w_u$, the claim implies $A_I=\emptyset$.
That is, for $u\in B_I^i\neq\emptyset$, we have $\calV_{\LP}(I)>\calV_{\LP}(I')+w_u$. 
However, since $A_I=\emptyset$, $I'=I\setminus(\delta_u\cup u)=I\setminus B_I^i$.
Since $u$ was not previously discarded, $ \calV_{\LP}(I) = \calV_{\LP}(I\setminus(\delta_u\cup u)) + w_u = \calV_{\LP}(I') + w_u$, a contradiction.
Hence, for each iteration, there is always a vertex $u$ with $\calV_{\LP}(I)=\calV_{\LP}(I')+w_u$.
 \end{proof}

For chordal graphs, we can again apply the proof of Proposition \ref{prop:house_structure} to conclude that in such graphs we do not encounter a house. Finally, for a co-chordal graph, we use the \emph{perfect elimination ordering} \cite[page 851]{pjm/1102995572} of its complement to show that the algorithm returns an MWSS.

\begin{proof}[Proof of Theorem~\ref{thm:main_LP}\ref{thm:main:unipolar}]
    Consider first the case that $G$ is unipolar.
    Suppose we are on an optimal trajectory and reach $I$.
    If $u\in I$ gives $\calV_{\LP}(I)>\calV_{\LP}(I')+w_u$, then $u$ is not optimal and the look-ahead discards this vertex. 
    If a vertex $u$ is not optimal and not discarded by the look-ahead, then
    $\alpha(I)=\calV_{\LP}(I)=\calV_{\LP}(I')+w_u>\alpha(I')+w_u$.
    For any later iteration, suppose we have vertex set $J$ remaining. If there exists $v\in J$ with $\calV_{\LP}(J)=\calV_{\LP}(J')+w_v$, select it.
    Because $u$ is a suboptimal choice, in some iteration every vertex $v\in J$ will have
    $\calV_{\LP}(J)>\calV_{\LP}(J')+w_v$, which contradicts Lemma~\ref{lem:unipolar}.
    Hence, if a suboptimal vertex is selected, the look-ahead always detects and discards it, so Algorithm~\ref{alg:look ahead} always returns an MWSS for $G$ by Proposition~\ref{prop:weight_hold}.
    
    Suppose next that $G$ is a chordal graph.
    By definition, it does not contain a house, since it cannot contain an induced cycle of length four or more.
    Then the proof of Proposition~\ref{prop:house_structure} implies that Algorithm~\ref{alg:look ahead} has $\calV_{\LP}(I)=\calV_{\LP}(I')+w_u$ for every iteration and every valid choice $u$. 
    Hence it returns an MWSS.
    
    Finally, suppose $G$ is a co-chordal graph.
    For the sake of contradiction, assume Algorithm~\ref{alg:look ahead} fails on $G$.
    That is, during some iteration starting with $I$, $u$ is selected and 
    \[
    \alpha( I ) = \calV_{\LP}(I) = \calV_{\LP}(I') + w_u > \alpha(I') + w_u.
    \]
    As in the proof of Lemma~\ref{lem:unipolar}, we can assume Algorithm~\ref{alg:look ahead} keeps picking vertices satisfying $\calV_{\LP}(I)=\calV_{\LP}(I')+w_u$ until we reach an $I$ where every $u\in I$ has $\calV_{\LP}(I)>\calV_{\LP}(I')+w_u$. 
    Consider the \emph{perfect elimination ordering} of the chordal graph $\overline{G|_{I}}$ \cite[page 851]{pjm/1102995572}, and let $v$ be the first vertex in the ordering.
    For $G|_{I}$, consider $I'=I\setminus(\delta_v\cup \{v\})$; by the perfect elimination ordering, $\overline{G|_{I'}}$ is a clique and hence $G|_{I'}$ is a stable set.
    Hence, no neighbors of a vertex in $I'$ contain an essential clique of $G|_{I'}$.
    Algorithm~\ref{alg:look ahead} does not discard any vertex after $\delta_v$ is deleted from $I$ and 
    $\calV_{\LP}(I)=\calV_{\LP}(I')+w_v$, a contradiction.
 \end{proof}

Although Theorem~\ref{thm:main_LP} only applies to certain subclasses of perfect graphs, 
Algorithm~\ref{alg:look ahead} does not require any knowledge about the input graph beyond its perfectness,
and it can be applied to any perfect graph.
Moreover, in Section \ref{sec:comp_exp} we introduce a variant of Algorithm~\ref{alg:look ahead} that can be applied to arbitrary, possibly imperfect graphs.

\subsection{Look-Ahead May Fail in Some Perfect Graphs}
\label{sec:look-ahead_fail}
 
Theorem~\ref{thm:main_LP} shows that Algorithm~\ref{alg:look ahead} returns an MWSS for generalized split, chordal and co-chordal graphs.
It is natural to wonder whether Algorithm~\ref{alg:look ahead} works for arbitrary perfect graphs.
Unfortunately, the answer is negative, as illustrated in the following example.

\begin{example}
\label{ex:ct-lookahead}
    Consider the left perfect graph in Figure \ref{fig:lookahead_ct}. 
    Proposition \ref{prop:ct-ngsgchordalcochordal} below shows that it is perfect but not generalized split, chordal or co-chordal.
    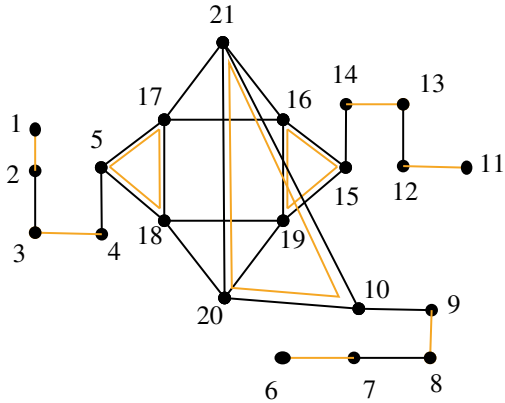
\begin{figure}[htb]
    \begin{subfigure}[t]{0.45\textwidth}
        \centering
\resizebox{0.6\textwidth}{0.5\textwidth}{\begin{tikzpicture}[x=0.6pt,y=0.6pt,yscale=-1,xscale=1]

\draw    (270.38,138.74) -- (307.28,90) ;
\draw [shift={(307.28,90)}, rotate = 307.13] [color={rgb, 255:red, 0; green, 0; blue, 0 }  ][fill={rgb, 255:red, 0; green, 0; blue, 0 }  ][line width=0.75]      (0, 0) circle [x radius= 3.35, y radius= 3.35]   ;
\draw [shift={(270.38,138.74)}, rotate = 307.13] [color={rgb, 255:red, 0; green, 0; blue, 0 }  ][fill={rgb, 255:red, 0; green, 0; blue, 0 }  ][line width=0.75]      (0, 0) circle [x radius= 3.35, y radius= 3.35]   ;
\draw    (270.38,138.74) -- (345.31,138.74) ;
\draw [shift={(345.31,138.74)}, rotate = 0] [color={rgb, 255:red, 0; green, 0; blue, 0 }  ][fill={rgb, 255:red, 0; green, 0; blue, 0 }  ][line width=0.75]      (0, 0) circle [x radius= 3.35, y radius= 3.35]   ;
\draw [shift={(270.38,138.74)}, rotate = 0] [color={rgb, 255:red, 0; green, 0; blue, 0 }  ][fill={rgb, 255:red, 0; green, 0; blue, 0 }  ][line width=0.75]      (0, 0) circle [x radius= 3.35, y radius= 3.35]   ;
\draw [color={rgb, 255:red, 0; green, 0; blue, 0 }  ,draw opacity=1 ]   (270.38,138.74) -- (270.38,202.48) ;
\draw    (270.38,202.48) -- (345.31,202.48) ;
\draw [shift={(345.31,202.48)}, rotate = 0] [color={rgb, 255:red, 0; green, 0; blue, 0 }  ][fill={rgb, 255:red, 0; green, 0; blue, 0 }  ][line width=0.75]      (0, 0) circle [x radius= 3.35, y radius= 3.35]   ;
\draw [shift={(270.38,202.48)}, rotate = 0] [color={rgb, 255:red, 0; green, 0; blue, 0 }  ][fill={rgb, 255:red, 0; green, 0; blue, 0 }  ][line width=0.75]      (0, 0) circle [x radius= 3.35, y radius= 3.35]   ;
\draw    (345.31,138.74) -- (345.31,202.48) ;
\draw    (345.31,138.74) -- (384.78,168.89) ;
\draw [shift={(384.78,168.89)}, rotate = 37.37] [color={rgb, 255:red, 0; green, 0; blue, 0 }  ][fill={rgb, 255:red, 0; green, 0; blue, 0 }  ][line width=0.75]      (0, 0) circle [x radius= 3.35, y radius= 3.35]   ;
\draw [shift={(345.31,138.74)}, rotate = 37.37] [color={rgb, 255:red, 0; green, 0; blue, 0 }  ][fill={rgb, 255:red, 0; green, 0; blue, 0 }  ][line width=0.75]      (0, 0) circle [x radius= 3.35, y radius= 3.35]   ;
\draw    (345.31,202.48) -- (384.78,168.89) ;
\draw    (384.78,168.89) -- (385,129) ;
\draw [shift={(385,129)}, rotate = 270.32] [color={rgb, 255:red, 0; green, 0; blue, 0 }  ][fill={rgb, 255:red, 0; green, 0; blue, 0 }  ][line width=0.75]      (0, 0) circle [x radius= 3.35, y radius= 3.35]   ;
\draw [shift={(384.78,168.89)}, rotate = 270.32] [color={rgb, 255:red, 0; green, 0; blue, 0 }  ][fill={rgb, 255:red, 0; green, 0; blue, 0 }  ][line width=0.75]      (0, 0) circle [x radius= 3.35, y radius= 3.35]   ;
\draw [color={rgb, 255:red, 245; green, 166; blue, 35 }  ,draw opacity=1 ]   (385,129) -- (421,129) ;
\draw [shift={(421,129)}, rotate = 0] [color={rgb, 255:red, 245; green, 166; blue, 35 }  ,draw opacity=1 ][fill={rgb, 255:red, 245; green, 166; blue, 35 }  ,fill opacity=1 ][line width=0.75]      (0, 0) circle [x radius= 3.35, y radius= 3.35]   ;
\draw    (421,129) -- (421,168) ;
\draw [shift={(421,168)}, rotate = 90] [color={rgb, 255:red, 0; green, 0; blue, 0 }  ][fill={rgb, 255:red, 0; green, 0; blue, 0 }  ][line width=0.75]      (0, 0) circle [x radius= 3.35, y radius= 3.35]   ;
\draw [shift={(421,129)}, rotate = 90] [color={rgb, 255:red, 0; green, 0; blue, 0 }  ][fill={rgb, 255:red, 0; green, 0; blue, 0 }  ][line width=0.75]      (0, 0) circle [x radius= 3.35, y radius= 3.35]   ;
\draw    (345.31,138.74) -- (307.28,90) ;
\draw [shift={(307.28,90)}, rotate = 232.04] [color={rgb, 255:red, 0; green, 0; blue, 0 }  ][fill={rgb, 255:red, 0; green, 0; blue, 0 }  ][line width=0.75]      (0, 0) circle [x radius= 3.35, y radius= 3.35]   ;
\draw [shift={(345.31,138.74)}, rotate = 232.04] [color={rgb, 255:red, 0; green, 0; blue, 0 }  ][fill={rgb, 255:red, 0; green, 0; blue, 0 }  ][line width=0.75]      (0, 0) circle [x radius= 3.35, y radius= 3.35]   ;
\draw [color={rgb, 255:red, 245; green, 166; blue, 35 }  ,draw opacity=1 ]   (421,168) -- (461,169) ;
\draw  [fill={rgb, 255:red, 0; green, 0; blue, 0 }  ,fill opacity=1 ] (464.01,169) .. controls (464.01,166.86) and (462.66,165.12) .. (461,165.12) .. controls (459.34,165.12) and (457.99,166.86) .. (457.99,169) .. controls (457.99,171.14) and (459.34,172.88) .. (461,172.88) .. controls (462.66,172.88) and (464.01,171.14) .. (464.01,169) -- cycle ;
\draw [color={rgb, 255:red, 245; green, 166; blue, 35 }  ,draw opacity=1 ]   (347.98,144.77) -- (347.98,194.73) ;
\draw [color={rgb, 255:red, 245; green, 166; blue, 35 }  ,draw opacity=1 ]   (347.98,144.77) -- (379.43,168.4) ;
\draw [color={rgb, 255:red, 245; green, 166; blue, 35 }  ,draw opacity=1 ]   (347.98,194.73) -- (379.43,168.4) ;

\draw    (270.38,138.74) -- (230.62,168.89) ;
\draw [shift={(230.62,168.89)}, rotate = 142.83] [color={rgb, 255:red, 0; green, 0; blue, 0 }  ][fill={rgb, 255:red, 0; green, 0; blue, 0 }  ][line width=0.75]      (0, 0) circle [x radius= 3.35, y radius= 3.35]   ;
\draw [shift={(270.38,138.74)}, rotate = 142.83] [color={rgb, 255:red, 0; green, 0; blue, 0 }  ][fill={rgb, 255:red, 0; green, 0; blue, 0 }  ][line width=0.75]      (0, 0) circle [x radius= 3.35, y radius= 3.35]   ;
\draw    (270.38,202.48) -- (230.62,168.89) ;
\draw    (230.62,168.89) -- (231,211) ;
\draw [shift={(231,211)}, rotate = 89.48] [color={rgb, 255:red, 0; green, 0; blue, 0 }  ][fill={rgb, 255:red, 0; green, 0; blue, 0 }  ][line width=0.75]      (0, 0) circle [x radius= 3.35, y radius= 3.35]   ;
\draw [shift={(230.62,168.89)}, rotate = 89.48] [color={rgb, 255:red, 0; green, 0; blue, 0 }  ][fill={rgb, 255:red, 0; green, 0; blue, 0 }  ][line width=0.75]      (0, 0) circle [x radius= 3.35, y radius= 3.35]   ;
\draw [color={rgb, 255:red, 245; green, 166; blue, 35 }  ,draw opacity=1 ]   (231,211) -- (189,210) ;
\draw [shift={(189,210)}, rotate = 181.36] [color={rgb, 255:red, 245; green, 166; blue, 35 }  ,draw opacity=1 ][fill={rgb, 255:red, 245; green, 166; blue, 35 }  ,fill opacity=1 ][line width=0.75]      (0, 0) circle [x radius= 3.35, y radius= 3.35]   ;
\draw    (189,210) -- (189,171) ;
\draw [shift={(189,171)}, rotate = 270] [color={rgb, 255:red, 0; green, 0; blue, 0 }  ][fill={rgb, 255:red, 0; green, 0; blue, 0 }  ][line width=0.75]      (0, 0) circle [x radius= 3.35, y radius= 3.35]   ;
\draw [shift={(189,210)}, rotate = 270] [color={rgb, 255:red, 0; green, 0; blue, 0 }  ][fill={rgb, 255:red, 0; green, 0; blue, 0 }  ][line width=0.75]      (0, 0) circle [x radius= 3.35, y radius= 3.35]   ;
\draw [color={rgb, 255:red, 245; green, 166; blue, 35 }  ,draw opacity=1 ]   (189,171) -- (189,141) ;
\draw  [fill={rgb, 255:red, 0; green, 0; blue, 0 }  ,fill opacity=1 ] (186.01,144.88) .. controls (186,142.74) and (187.34,141) .. (189,141) .. controls (190.66,141) and (192.02,142.74) .. (192.04,144.88) .. controls (192.05,147.02) and (190.71,148.75) .. (189.05,148.75) .. controls (187.39,148.75) and (186.03,147.02) .. (186.01,144.88) -- cycle ;
\draw [color={rgb, 255:red, 245; green, 166; blue, 35 }  ,draw opacity=1 ]   (267.25,144.77) -- (267.58,194.73) ;
\draw [color={rgb, 255:red, 245; green, 166; blue, 35 }  ,draw opacity=1 ]   (267.25,144.77) -- (235.96,168.4) ;
\draw [color={rgb, 255:red, 245; green, 166; blue, 35 }  ,draw opacity=1 ]   (267.58,194.73) -- (235.96,168.4) ;

\draw    (308.4,251.22) -- (270.38,202.48) ;
\draw [shift={(270.38,202.48)}, rotate = 232.04] [color={rgb, 255:red, 0; green, 0; blue, 0 }  ][fill={rgb, 255:red, 0; green, 0; blue, 0 }  ][line width=0.75]      (0, 0) circle [x radius= 3.35, y radius= 3.35]   ;
\draw [shift={(308.4,251.22)}, rotate = 232.04] [color={rgb, 255:red, 0; green, 0; blue, 0 }  ][fill={rgb, 255:red, 0; green, 0; blue, 0 }  ][line width=0.75]      (0, 0) circle [x radius= 3.35, y radius= 3.35]   ;
\draw    (345.31,202.48) -- (308.4,251.22) ;
\draw [shift={(308.4,251.22)}, rotate = 127.13] [color={rgb, 255:red, 0; green, 0; blue, 0 }  ][fill={rgb, 255:red, 0; green, 0; blue, 0 }  ][line width=0.75]      (0, 0) circle [x radius= 3.35, y radius= 3.35]   ;
\draw [shift={(345.31,202.48)}, rotate = 127.13] [color={rgb, 255:red, 0; green, 0; blue, 0 }  ][fill={rgb, 255:red, 0; green, 0; blue, 0 }  ][line width=0.75]      (0, 0) circle [x radius= 3.35, y radius= 3.35]   ;
\draw    (308.4,251.22) -- (307.28,90) ;
\draw [shift={(307.28,90)}, rotate = 269.6] [color={rgb, 255:red, 0; green, 0; blue, 0 }  ][fill={rgb, 255:red, 0; green, 0; blue, 0 }  ][line width=0.75]      (0, 0) circle [x radius= 3.35, y radius= 3.35]   ;
\draw [shift={(308.4,251.22)}, rotate = 269.6] [color={rgb, 255:red, 0; green, 0; blue, 0 }  ][fill={rgb, 255:red, 0; green, 0; blue, 0 }  ][line width=0.75]      (0, 0) circle [x radius= 3.35, y radius= 3.35]   ;
\draw    (393,258) -- (307.28,90) ;
\draw [shift={(307.28,90)}, rotate = 242.97] [color={rgb, 255:red, 0; green, 0; blue, 0 }  ][fill={rgb, 255:red, 0; green, 0; blue, 0 }  ][line width=0.75]      (0, 0) circle [x radius= 3.35, y radius= 3.35]   ;
\draw [shift={(393,258)}, rotate = 242.97] [color={rgb, 255:red, 0; green, 0; blue, 0 }  ][fill={rgb, 255:red, 0; green, 0; blue, 0 }  ][line width=0.75]      (0, 0) circle [x radius= 3.35, y radius= 3.35]   ;
\draw    (393,258) -- (308.4,251.22) ;
\draw [shift={(308.4,251.22)}, rotate = 184.58] [color={rgb, 255:red, 0; green, 0; blue, 0 }  ][fill={rgb, 255:red, 0; green, 0; blue, 0 }  ][line width=0.75]      (0, 0) circle [x radius= 3.35, y radius= 3.35]   ;
\draw [shift={(393,258)}, rotate = 184.58] [color={rgb, 255:red, 0; green, 0; blue, 0 }  ][fill={rgb, 255:red, 0; green, 0; blue, 0 }  ][line width=0.75]      (0, 0) circle [x radius= 3.35, y radius= 3.35]   ;
\draw    (393,258) -- (438.94,258.89) ;
\draw [shift={(438.94,258.89)}, rotate = 1.11] [color={rgb, 255:red, 0; green, 0; blue, 0 }  ][fill={rgb, 255:red, 0; green, 0; blue, 0 }  ][line width=0.75]      (0, 0) circle [x radius= 3.35, y radius= 3.35]   ;
\draw [shift={(393,258)}, rotate = 1.11] [color={rgb, 255:red, 0; green, 0; blue, 0 }  ][fill={rgb, 255:red, 0; green, 0; blue, 0 }  ][line width=0.75]      (0, 0) circle [x radius= 3.35, y radius= 3.35]   ;
\draw [color={rgb, 255:red, 245; green, 166; blue, 35 }  ,draw opacity=1 ]   (438.94,258.89) -- (438,289) ;
\draw [shift={(438,289)}, rotate = 91.79] [color={rgb, 255:red, 245; green, 166; blue, 35 }  ,draw opacity=1 ][fill={rgb, 255:red, 245; green, 166; blue, 35 }  ,fill opacity=1 ][line width=0.75]      (0, 0) circle [x radius= 3.35, y radius= 3.35]   ;
\draw    (438,289) -- (390,289) ;
\draw [shift={(390,289)}, rotate = 180] [color={rgb, 255:red, 0; green, 0; blue, 0 }  ][fill={rgb, 255:red, 0; green, 0; blue, 0 }  ][line width=0.75]      (0, 0) circle [x radius= 3.35, y radius= 3.35]   ;
\draw [shift={(438,289)}, rotate = 180] [color={rgb, 255:red, 0; green, 0; blue, 0 }  ][fill={rgb, 255:red, 0; green, 0; blue, 0 }  ][line width=0.75]      (0, 0) circle [x radius= 3.35, y radius= 3.35]   ;
\draw [color={rgb, 255:red, 245; green, 166; blue, 35 }  ,draw opacity=1 ]   (390,289) -- (345,289) ;
\draw  [fill={rgb, 255:red, 0; green, 0; blue, 0 }  ,fill opacity=1 ] (349.5,289) .. controls (349.5,287.07) and (347.49,285.5) .. (345,285.5) .. controls (342.51,285.5) and (340.5,287.07) .. (340.5,289) .. controls (340.5,290.93) and (342.51,292.5) .. (345,292.5) .. controls (347.49,292.5) and (349.5,290.93) .. (349.5,289) -- cycle ;
\draw  [color={rgb, 255:red, 245; green, 166; blue, 35 }  ,draw opacity=1 ] (380.37,250.39) -- (312.96,244.83) -- (311.3,102.86) -- cycle ;

\draw (233,214.4) node [anchor=north west][inner sep=0.75pt]    {$4$};
\draw (173,215.4) node [anchor=north west][inner sep=0.75pt]    {$3$};
\draw (171,133.4) node [anchor=north west][inner sep=0.75pt]    {$1$};
\draw (169,165.4) node [anchor=north west][inner sep=0.75pt]    {$2$};
\draw (222,143.4) node [anchor=north west][inner sep=0.75pt]    {$5$};
\draw (332,301.4) node [anchor=north west][inner sep=0.75pt]    {$6$};
\draw (394,301.4) node [anchor=north west][inner sep=0.75pt]    {$7$};
\draw (436,297.4) node [anchor=north west][inner sep=0.75pt]    {$8$};
\draw (447,247.4) node [anchor=north west][inner sep=0.75pt]    {$9$};
\draw (394,238.4) node [anchor=north west][inner sep=0.75pt]    {$10$};
\draw (467,158.4) node [anchor=north west][inner sep=0.75pt]    {$11$};
\draw (412,176.4) node [anchor=north west][inner sep=0.75pt]    {$12$};
\draw (429,108.4) node [anchor=north west][inner sep=0.75pt]    {$13$};
\draw (374,102.4) node [anchor=north west][inner sep=0.75pt]    {$14$};
\draw (375,179.4) node [anchor=north west][inner sep=0.75pt]    {$15$};
\draw (345,118.4) node [anchor=north west][inner sep=0.75pt]    {$16$};
\draw (251,116.4) node [anchor=north west][inner sep=0.75pt]    {$17$};
\draw (251,204.4) node [anchor=north west][inner sep=0.75pt]    {$18$};
\draw (341,206.4) node [anchor=north west][inner sep=0.75pt]    {$19$};
\draw (289,254.4) node [anchor=north west][inner sep=0.75pt]    {$20$};
\draw (297,64.4) node [anchor=north west][inner sep=0.75pt]    {$21$};
\end{tikzpicture}}

        
        \end{subfigure}
\begin{subfigure}[t]{0.45\textwidth}
\centering
\resizebox{0.6\textwidth}{0.5\textwidth}{\begin{tikzpicture}[x=0.6pt,y=0.6pt,yscale=-1,xscale=1]

\draw    (270.38,138.74) -- (307.28,90) ;
\draw [shift={(307.28,90)}, rotate = 307.13] [color={rgb, 255:red, 0; green, 0; blue, 0 }  ][fill={rgb, 255:red, 0; green, 0; blue, 0 }  ][line width=0.75]      (0, 0) circle [x radius= 3.35, y radius= 3.35]   ;
\draw [shift={(270.38,138.74)}, rotate = 307.13] [color={rgb, 255:red, 0; green, 0; blue, 0 }  ][fill={rgb, 255:red, 0; green, 0; blue, 0 }  ][line width=0.75]      (0, 0) circle [x radius= 3.35, y radius= 3.35]   ;
\draw    (270.38,138.74) -- (345.31,138.74) ;
\draw [shift={(345.31,138.74)}, rotate = 0] [color={rgb, 255:red, 0; green, 0; blue, 0 }  ][fill={rgb, 255:red, 0; green, 0; blue, 0 }  ][line width=0.75]      (0, 0) circle [x radius= 3.35, y radius= 3.35]   ;
\draw [shift={(270.38,138.74)}, rotate = 0] [color={rgb, 255:red, 0; green, 0; blue, 0 }  ][fill={rgb, 255:red, 0; green, 0; blue, 0 }  ][line width=0.75]      (0, 0) circle [x radius= 3.35, y radius= 3.35]   ;
\draw [color={rgb, 255:red, 245; green, 166; blue, 35 }  ,draw opacity=1 ]   (270.38,138.74) -- (270.38,202.48) ;
\draw    (270.38,202.48) -- (345.31,202.48) ;
\draw [shift={(345.31,202.48)}, rotate = 0] [color={rgb, 255:red, 0; green, 0; blue, 0 }  ][fill={rgb, 255:red, 0; green, 0; blue, 0 }  ][line width=0.75]      (0, 0) circle [x radius= 3.35, y radius= 3.35]   ;
\draw [shift={(270.38,202.48)}, rotate = 0] [color={rgb, 255:red, 0; green, 0; blue, 0 }  ][fill={rgb, 255:red, 0; green, 0; blue, 0 }  ][line width=0.75]      (0, 0) circle [x radius= 3.35, y radius= 3.35]   ;
\draw [color={rgb, 255:red, 245; green, 166; blue, 35 }  ,draw opacity=1 ]   (345.31,138.74) -- (345.31,202.48) ;
\draw    (345.31,138.74) -- (307.28,90) ;
\draw [shift={(307.28,90)}, rotate = 232.04] [color={rgb, 255:red, 0; green, 0; blue, 0 }  ][fill={rgb, 255:red, 0; green, 0; blue, 0 }  ][line width=0.75]      (0, 0) circle [x radius= 3.35, y radius= 3.35]   ;
\draw [shift={(345.31,138.74)}, rotate = 232.04] [color={rgb, 255:red, 0; green, 0; blue, 0 }  ][fill={rgb, 255:red, 0; green, 0; blue, 0 }  ][line width=0.75]      (0, 0) circle [x radius= 3.35, y radius= 3.35]   ;
\draw    (308.4,251.22) -- (270.38,202.48) ;
\draw [shift={(270.38,202.48)}, rotate = 232.04] [color={rgb, 255:red, 0; green, 0; blue, 0 }  ][fill={rgb, 255:red, 0; green, 0; blue, 0 }  ][line width=0.75]      (0, 0) circle [x radius= 3.35, y radius= 3.35]   ;
\draw [shift={(308.4,251.22)}, rotate = 232.04] [color={rgb, 255:red, 0; green, 0; blue, 0 }  ][fill={rgb, 255:red, 0; green, 0; blue, 0 }  ][line width=0.75]      (0, 0) circle [x radius= 3.35, y radius= 3.35]   ;
\draw    (345.31,202.48) -- (308.4,251.22) ;
\draw [shift={(308.4,251.22)}, rotate = 127.13] [color={rgb, 255:red, 0; green, 0; blue, 0 }  ][fill={rgb, 255:red, 0; green, 0; blue, 0 }  ][line width=0.75]      (0, 0) circle [x radius= 3.35, y radius= 3.35]   ;
\draw [shift={(345.31,202.48)}, rotate = 127.13] [color={rgb, 255:red, 0; green, 0; blue, 0 }  ][fill={rgb, 255:red, 0; green, 0; blue, 0 }  ][line width=0.75]      (0, 0) circle [x radius= 3.35, y radius= 3.35]   ;
\draw [color={rgb, 255:red, 245; green, 166; blue, 35 }  ,draw opacity=1 ]   (308.4,251.22) -- (307.28,90) ;
\draw  [fill={rgb, 255:red, 0; green, 0; blue, 0 }  ,fill opacity=1 ] (426.5,169) .. controls (426.5,167.07) and (424.49,165.5) .. (422,165.5) .. controls (419.51,165.5) and (417.5,167.07) .. (417.5,169) .. controls (417.5,170.93) and (419.51,172.5) .. (422,172.5) .. controls (424.49,172.5) and (426.5,170.93) .. (426.5,169) -- cycle ;
\draw  [fill={rgb, 255:red, 0; green, 0; blue, 0 }  ,fill opacity=1 ] (389.5,133) .. controls (389.5,131.07) and (387.49,129.5) .. (385,129.5) .. controls (382.51,129.5) and (380.5,131.07) .. (380.5,133) .. controls (380.5,134.93) and (382.51,136.5) .. (385,136.5) .. controls (387.49,136.5) and (389.5,134.93) .. (389.5,133) -- cycle ;
\draw  [fill={rgb, 255:red, 0; green, 0; blue, 0 }  ,fill opacity=1 ] (400.5,290) .. controls (400.5,288.07) and (398.49,286.5) .. (396,286.5) .. controls (393.51,286.5) and (391.5,288.07) .. (391.5,290) .. controls (391.5,291.93) and (393.51,293.5) .. (396,293.5) .. controls (398.49,293.5) and (400.5,291.93) .. (400.5,290) -- cycle ;
\draw  [fill={rgb, 255:red, 0; green, 0; blue, 0 }  ,fill opacity=1 ] (442.5,260) .. controls (442.5,258.07) and (440.49,256.5) .. (438,256.5) .. controls (435.51,256.5) and (433.5,258.07) .. (433.5,260) .. controls (433.5,261.93) and (435.51,263.5) .. (438,263.5) .. controls (440.49,263.5) and (442.5,261.93) .. (442.5,260) -- cycle ;
\draw  [fill={rgb, 255:red, 0; green, 0; blue, 0 }  ,fill opacity=1 ] (194.5,175) .. controls (194.5,173.07) and (192.49,171.5) .. (190,171.5) .. controls (187.51,171.5) and (185.5,173.07) .. (185.5,175) .. controls (185.5,176.93) and (187.51,178.5) .. (190,178.5) .. controls (192.49,178.5) and (194.5,176.93) .. (194.5,175) -- cycle ;
\draw  [fill={rgb, 255:red, 0; green, 0; blue, 0 }  ,fill opacity=1 ] (233.5,207) .. controls (233.5,205.07) and (231.49,203.5) .. (229,203.5) .. controls (226.51,203.5) and (224.5,205.07) .. (224.5,207) .. controls (224.5,208.93) and (226.51,210.5) .. (229,210.5) .. controls (231.49,210.5) and (233.5,208.93) .. (233.5,207) -- cycle ;
\draw  [color={rgb, 255:red, 208; green, 2; blue, 27 }  ,draw opacity=1 ] (371.75,119.75) -- (398.25,119.75) -- (398.25,146.25) -- (371.75,146.25) -- cycle ;
\draw  [color={rgb, 255:red, 208; green, 2; blue, 27 }  ,draw opacity=1 ] (408.75,155.75) -- (435.25,155.75) -- (435.25,182.25) -- (408.75,182.25) -- cycle ;
\draw  [color={rgb, 255:red, 208; green, 2; blue, 27 }  ,draw opacity=1 ] (424.75,246.75) -- (451.25,246.75) -- (451.25,273.25) -- (424.75,273.25) -- cycle ;
\draw  [color={rgb, 255:red, 208; green, 2; blue, 27 }  ,draw opacity=1 ] (382.75,276.75) -- (409.25,276.75) -- (409.25,303.25) -- (382.75,303.25) -- cycle ;
\draw  [color={rgb, 255:red, 208; green, 2; blue, 27 }  ,draw opacity=1 ] (215.75,193.75) -- (242.25,193.75) -- (242.25,220.25) -- (215.75,220.25) -- cycle ;
\draw  [color={rgb, 255:red, 208; green, 2; blue, 27 }  ,draw opacity=1 ] (176.75,161.75) -- (203.25,161.75) -- (203.25,188.25) -- (176.75,188.25) -- cycle ;

\draw (195,218.4) node [anchor=north west][inner sep=0.75pt]    {$4$};
\draw (159,166.4) node [anchor=north west][inner sep=0.75pt]    {$2$};
\draw (422,305.4) node [anchor=north west][inner sep=0.75pt]    {$7$};
\draw (460,245.4) node [anchor=north west][inner sep=0.75pt]    {$9$};
\draw (450.75,157.65) node [anchor=north west][inner sep=0.75pt]    {$12$};
\draw (374,102.4) node [anchor=north west][inner sep=0.75pt]    {$14$};
\draw (345,118.4) node [anchor=north west][inner sep=0.75pt]    {$16$};
\draw (251,116.4) node [anchor=north west][inner sep=0.75pt]    {$17$};
\draw (251,204.4) node [anchor=north west][inner sep=0.75pt]    {$18$};
\draw (341,206.4) node [anchor=north west][inner sep=0.75pt]    {$19$};
\draw (289,254.4) node [anchor=north west][inner sep=0.75pt]    {$20$};
\draw (297,64.4) node [anchor=north west][inner sep=0.75pt]    {$21$};
\end{tikzpicture}}

\end{subfigure}
\caption{Look-ahead example}
\label{fig:lookahead_ct}
\end{figure}
As before, the yellow cliques represent the essential cliques given by a strictly complementary solution of~\eqref{LP-D}.
an MWSS of this graph has size $9$.
As in Example~\ref{ex:house-ct}, suppose $2,7,12$ are selected; we then obtain the right graph in Figure \ref{fig:lookahead_ct}.
It is not hard to see that any remaining vertex is a bad choice for the same reason from Example~\ref{ex:house-ct}.
Thus, even with look-ahead, Algorithm~\ref{alg:look ahead} fails on this graph, and Algorithm~\ref{alg:look ahead} with $\calV_{\SDP}$ and look-ahead also fails to obtain an MWSS.
\end{example}
By generalizing Example~\ref{ex:ct-lookahead} with longer paths, Algorithm~\ref{alg:look ahead} with a constant number of look-ahead steps can also be made to fail on perfect graphs: 
In the example, after picking $2$ and $7$, $12$ is a suboptimal choice.
By adding steps to the look ahead, we can detect that $14$ is suboptimal choice and hence also $12$ (after choosing $2$ and $7$).
However, by extending the paths $1-5$, $6-10$ and $11-15$ to paths with length $2k$ and picking vertices incident to the leaves, we need $(k+1)$ steps in the look-ahead to detect a suboptimal choice.

The next proposition verifies that this graph and the one used in Example \ref{ex:house-ct} are both perfect but not in the families covered by our results.
\begin{proposition}
\label{prop:ct-ngsgchordalcochordal}
The graphs in Examples \ref{ex:house-ct} and \ref{ex:ct-lookahead} are perfect but not generalized split, chordal, or co-chordal
\end{proposition}
   \begin{proof}
   Consider the graph $G$ in Example~\ref{ex:house-ct}.
       We first prove the perfectness.
        It is clear that $G$ is $3$-colorable by coloring $\{1,4,6,8,10\}$, $\{2,5,7,9,11\}$ and $\{3\}$ differently, and the clique number of $G$ is $3$.
        Then, except for $G$, the components of any other induced subgraph of $G$ are made of isolated vertices, paths, even holes or containing a triangle, so their chromatic numbers are $1,2,3$ respectively. 
        That is, the clique number of any subgraph of $G$ equals to its chromatic number, which implies the perfectness of $G$.

        It is not hard to see that any induced subgraph of a generalized split graph is also generalized split.
        Let $S=\{2,3,6,7,8,9,10,11\}$ and $G|_S$.
        $G|_S$ is shown to be not generalized split in~\cite{ESCHEN2014195}. 
        Thus, $G$ is not generalized split.

        Since $G|_{\{3,4,5,6\}}$ is an even hole, $G$ is not chordal.
        Similarly, $\bar{G}|_{\{1,2,10,11\}}$ is an even hole, so $\bar{G}$ is not chordal.
        Thus, $G$ is not chordal or co-chordal.

        Consider $G$ in Example~\ref{ex:ct-lookahead}.
        It is $3$-colorable by coloring $\{2,4,7,9,11,13,15,17,20\}$, $\{1,3,4,19,21\}$, $\{6,8,10,12,14,16,18\}$ differently.
        Similar arguments as above apply; the clique number of all induced subgraphs of $G$ equal to its chromatic number, so $G$ is perfect.

        Notice that graph in Example~\ref{ex:house-ct} is an induced subgraph of $G$, so $G$ is not generalized split.
        Similarly, it is not chordal or co-chordal.
   \end{proof}



\section{Rounding via SDP VFA}\label{sec:SDPVFA}

In the previous sections, we show that Algorithm~\ref{alg:look ahead} with $\calV_{\LP}$ returns an MWSS for generalized split graphs, chordal graphs, and co-chordal graphs. 
However, constructing $\calV_{\LP}$ requires a strictly complementary solution of~\eqref{LP-D}, which generally takes exponential time since there is a variable for each clique in~$G$. 
In this section we show that Algorithm~\ref{alg:look ahead} with $\calV_{\SDP}$ inherits the performance guarantees we obtain for $\calV_{\LP}$.
In particular, Algorithm~\ref{alg:look ahead} with $\calV_{\SDP}$ generates an MWSS for generalized split graphs, chordal graphs, and co-chordal graphs.

Throughout this section, $ G = (N, E) $ is a perfect graph; let $(x^*,\mu^*)$ denote a fixed pair of strictly complementary solutions of \eqref{LP-P}, \eqref{LP-D},
let $(\bar{x},\bar{X},\bar{q},\bar{Q})$ be a fixed tuple of optimal solutions of \eqref{SDP-P}, \eqref{SDP-D} in the relative interior of the optimal face,
and let $t^*$ be the optimal value of these problems.
Finally, let $\calV_{\LP}$ and $\calV_{\SDP}$ be the VFAs constructed from them.

\begin{theorem}
\label{thm:SDP_beat_LP}
    Let $G=(N,E)$ be a perfect graph and let $\calV_{\LP}$ and $\calV_{\SDP}$ be as defined above.
    If Algorithm~\ref{alg:look ahead} with $\calV_{\LP}$ returns an MWSS of $G$, irrespective of which vertex choices are made in line \ref{alg_step:select},
    then Algorithm~\ref{alg:look ahead} with $\calV_{\SDP}$ also returns an MWSS.
\end{theorem}

Our main result, namely Theorem~\ref{thm:main_SDP},
is a consequence of Theorems \ref{thm:main_LP} and \ref{thm:SDP_beat_LP}.
The key idea behind the proof of Theorem~\ref{thm:SDP_beat_LP} is to show that if $S$ is on an optimal trajectory and $I \subseteq N$ is the set of remaining vertices, then for every $i\in I$,
\begin{equation}\label{eq:implication}
\calV_{\LP}(I)-\calV_{\LP}(I\!\setminus\!(\{i\}\!\cup\!\delta_i)> w_i
\implies
\calV_{\SDP}(I)-\calV_{\SDP}(I\!\setminus\!(\{i\}\!\cup\!\delta_i)> w_i.
\end{equation}
In particular,~\eqref{eq:implication} shows that $\calV_{\SDP}$ discards every vertex discarded by $\calV_{\LP}$.
Furthermore, by the first property of Lemma~\ref{lem:fund_VFA_prop}, $\calV_{\SDP}$ might discard some suboptimal vertices that $\calV_{\LP}$ does not. 

Roughly, the argument to prove \eqref{eq:implication} proceeds as follows; we provide the details below. 
When $S$ is on an optimal trajectory and $ \calV_{\LP}(I)-\calV_{\LP}(I \setminus (\{i\} \cup \delta_i) > w_i $, there is an essential clique $C$ of $G|_I$ in $\delta_i\cap I$; 
this $C$ is a subset of an essential clique $\tilde{C}$ of $G$.
Given the LP solution $\mu^*$, we construct an optimal solution $(t^*,q_{\LP},Q_{\LP})$ for~\eqref{SDP-D}, and a vector $p_{\tilde{C}}\in \R^{n}$ with $p_{\tilde{C}} \in \range(Q_{\LP})$.
Since $(\bar{q},\bar{Q})$ is in the relative interior, $p_{\tilde{C}}\in\range(Q_{\LP})\subseteq\range(\bar{Q})$.
By analyzing $(\bar{q},\bar{Q})$, we show that $p_{\tilde{C}}\in\range(\bar{Q})$ implies \( \calV_{\SDP}(I)-\calV_{\SDP}(I \setminus (\{i\} \cup \delta_i) ) >  w_i \).

\subsection{Proof of Theorem~\ref{thm:SDP_beat_LP}}
We first analyze \PI of Algorithm~\ref{alg:look ahead} with $\calV_{\SDP}$.

\begin{lemma}
\label{lem:SDP PhaseI}
    For any $i\in V$, $\bar{x}_i>0$ if and only if $i$ is in an MWSS of $G$.
\end{lemma}
\begin{proof} $(\implies)$ Since $G$ is a perfect graph, by~\eqref{relation:perfect graph} and the optimality of $\bar{x}$, there exists an MWSS of $G$ including $i$.
$(\impliedby) $ For the sake of contradiction, suppose $\bar{x}_i=0$. 
By~\eqref{relation:perfect graph}, there exists an optimal solution $(x^*,X^*)$ such that $x^*_i=1$.
Then, for any $\lambda\in\R$, $(\bar{x},\bar{X})+\lambda(x^*-\bar{x},X^*-\bar{X})$ is in the affine hull of the optimal face of~\eqref{SDP-D}.
By $(\bar{x},\bar{X})$ being in the relative interior of the optimal face, $(\tilde{x},\tilde{X})=(\bar{x},\bar{X})-\epsilon(x^*-\bar{x},X^*-\bar{X})$ is in the optimal face for $\epsilon>0$ small enough.
But $\tilde{x}_i=\bar{x}_i-\epsilon x^*_i<0$, which is not feasible to~\eqref{SDP-D}, contradiction.  
\end{proof}
Thus, for $\calV_{\SDP}$, Algorithm~\ref{alg:look ahead} discards all vertices that are not in any MWSS during \PI.
Hence, $\calV_{\LP}$ and $\calV_{\SDP}$ reach the same set of vertices after \PI,
and it remains to analyze \PII. 
We now show how to construct an optimal solution of~\eqref{SDP-D} from one of~\eqref{LP-D}.
\begin{lemma}
\label{lem:LPDtoSDPD}
    Consider vectors $b \in \R^N$ and $p_C \in \R^N$ for $C\in \mathcal{C}(G)$, with entries
\begin{equation}
b_{i}=
\begin{cases}
    0,&\text{ if } i =0\\
    \sum_{C\ni i}\mu^*_C-w_i,&\text{ if }i > 0
\end{cases}
\quad\text{ and }\quad
    [p_C]_i:=\begin{cases}
    \sqrt{\mu^*_C},&\text{ if }i=0\\
    -\sqrt{\mu^*_C},&\text{ if }i\in C\\
    0,&\text{ otherwise}
\end{cases}
\end{equation}
Then
\(
M := \sum_{C\in \mathcal{C}(G)}p_C p_C^\top+ \Diag(b)
\)
is optimal for~\eqref{SDP-D}.
\end{lemma}
\begin{proof}
    First we verify feasibility. 
    By $\mu^*$ being feasible for~\eqref{LP-D}, we know $\sum_{C\ni i}\mu^*_C-w_i\geq 0$, so $\Diag(b)$ is PSD.
    Since each $p_C p_C^\top$ is PSD,  $M$ is PSD.
    Notice that $M_{ij}=0$ for $ij\notin E, i\neq j$,
    since no clique contains both $i,j$ and so $[p_Cp_C^\top]_{ij}=0$. 
    Consider the blocks
    \( M = \begin{psmallmatrix} t&p^\top\\ p&P \end{psmallmatrix} \).
    Since $P_{ii}=\sum_{C\ni i}\mu^*_C+\sum_{C\ni i}\mu^*_C-w_i=2\sum_{C\ni i}\mu^*_C-w_i$ and $p_i=-\sum_{C\ni i}\mu^*_C$,
    then $M$ is feasible for~\eqref{SDP-D}.
    Finally, $t=\sum_{C\in \mathcal{C}(G)}\mu_C^*=t^*$, so $M$ is optimal.
 \end{proof}
The next lemma establishes when $\calV_{\SDP}(I) = \calV_{\SDP}(S)$ for some $S\subseteq I$.
\begin{lemma}
\label{lem:row-dep}
Let
$\begin{psmallmatrix} A&C^\top\\ C&D \end{psmallmatrix}\succeq 0$,
with $A\in\symm^s$, $C\in\R^{(n-s)\times s}$ and $D\in\symm^{n-s}$,
and let
$a \in\range(A) \subset \R^s$, $d \in \R^{n-s}$.
Consider
\begin{align*}
    v_I:=\min\left\{t:\begin{psmallmatrix}
        t&a^\top&d^\top\\
        a&A&C^\top\\
        d&C&D
    \end{psmallmatrix}\succeq 0\right\}=\begin{psmallmatrix}
        a&d
    \end{psmallmatrix}\begin{psmallmatrix}
        A&C^\top\\
        C&D
    \end{psmallmatrix}^\dagger\begin{psmallmatrix}
        a\\
        d
    \end{psmallmatrix};~
    v_S:=\min\left\{t:
    \begin{psmallmatrix} t&a^\top\\ a&A \end{psmallmatrix}
    \succeq 0\right\}=a^\top A^\dagger a.
\end{align*}
Let $r=\rank(A)$ and assume that the upper $r\times r$ principal matrix is invertible.
So we may write
$
A=\begin{psmallmatrix} \tilde{A}&\tilde{C}^\top\\ \tilde{C}&\tilde{D} \end{psmallmatrix},
~
\tilde{A}\in\symm^{r}_{++},\
\tilde{C}=U\tilde{A} \in \R^{(s-r)\times r},~
\tilde{D}=U\tilde{A}U^\top \in \symm^{s-r}
$
for some $U \in \R^{(s-r)\times r}$.
Using that $a \in \range(A)$, we may also write
$
a=\begin{psmallmatrix} \tilde{a}\\ \tilde{d} \end{psmallmatrix},~
\tilde a\in\R^{r}, \tilde d = U \tilde a\in\R^{s-r},~
C=\begin{psmallmatrix} C_1&C_2 \end{psmallmatrix},~
C_1\in\R^{(n-s)\times r},
C_2\in\R^{(n-s)\times(s-r)}.
$
Then, $v_I=v_S$ if and only if there exists some $V\in\R^{(n-r)\times r}$ such that $\begin{psmallmatrix}
    \tilde{d}\\d
\end{psmallmatrix}=V\tilde{a}$, $\begin{psmallmatrix}
    \tilde{C}\\C_1
\end{psmallmatrix}=V\tilde{A}$, $\begin{psmallmatrix}
    \tilde{D}&C_2^\top\\
    C_2&D
\end{psmallmatrix}\succeq V\tilde{A}V^\top$.
\end{lemma}
\begin{proof}\
$(\impliedby)$ Assume such a $V$ exists. 
        The inequality $v_I \geq v_S$ always holds,
        so it remains to see that $v_S \geq v_I$.
        Let $t$ be feasible for the SDP defining $v_S$, i.e.,
        $\begin{psmallmatrix} t&\tilde{a}^\top\\ \tilde{a}&\tilde{A}\\ \end{psmallmatrix}\succeq 0$. 
    Since
    \[\begin{psmallmatrix}
        t&a^\top&d^\top\\
        a&A&C^\top\\
        d&C&D
    \end{psmallmatrix}=\begin{psmallmatrix}
        1&0\\
        0&I\\
        0&V
    \end{psmallmatrix}\begin{psmallmatrix}
        t&\tilde{a}^\top\\
        \tilde{a}&\tilde{A}\\
    \end{psmallmatrix}\begin{psmallmatrix}
        1&0&0\\
        0&I&V^\top
    \end{psmallmatrix}+\begin{psmallmatrix}
        0&0&0\\
        0&0&0\\
        0&0&\begin{psmallmatrix}
    \tilde{D}&C_2^\top\\
    C_2&D
\end{psmallmatrix}-V\tilde{A}V^\top
    \end{psmallmatrix} , \]
the above matrix is PSD, so $t$ is also feasible for the SDP defining $v_I$; 
hence, $v_S \geq v_I$. 

$(\implies)$ Suppose now that $v_I=v_S$.
    Consider first the case that $r=s$, i.e., $A$ is positive definite. 
    Notice that if $v_I=v_S$, then the same still holds if we replace $D$ by $D+P$ for any PSD matrix $P$.
    Hence, we may assume that $S:=D-CA^{-1}C^\top$ is also positive definite. 
    Recall the inverse formula for a block matrix:
    \[\begin{psmallmatrix}
        A&C^\top\\
        C&D
    \end{psmallmatrix}^{-1}=\begin{psmallmatrix}
        A^{-1}+A^{-1}C^\top S^{-1}CA^{-1}&-A^{-1}C^\top S^{-1}\\
        -S^{-1}CA^{-1}&S^{-1}
    \end{psmallmatrix}.\]
    Using
    $v_I=\begin{psmallmatrix} a&d \end{psmallmatrix}
    \begin{psmallmatrix} A&C^\top\\ C&D \end{psmallmatrix}^{-1}
    \begin{psmallmatrix} a\\ d \end{psmallmatrix},
    v_S = a^\top A^{-1} a$,
    we obtain 
    \begin{align*}
        v_I-v_S&=a^\top A^{-1}C^{\top}S^{-1}CA^{-1}a-2a^\top A^{-1}C^{\top}S^{-1}d+d^\top S^{-1}d
        =\|S^{-1/2}CA^{-1}a-S^{-1/2}d\|^2.
    \end{align*}
    Since $v_I=v_S$, then $d=CA^{-1}a$. 
    The wanted matrix is
    $V=\begin{psmallmatrix} U\\ CA^{-1} \end{psmallmatrix}$.

    Consider now the case that $r = \rank(A) <s$. 
    Let
    \(v_{\tilde S}:=\min\left\{t:\begin{psmallmatrix}
        t&\tilde{a}^\top\\
        \tilde{a}&\tilde{A}\\
    \end{psmallmatrix}\succeq 0\right\}=\tilde{a}^\top \tilde{A}^{-1} \tilde{a}.\)
    By the first part of this proof,
    we get that $v_S=v_{\tilde{S}}$. 
    Since $v_I=v_{\tilde{S}}$ and $\tilde A \succ 0$,
    we are back in the positive definite case,
    and we can find a matrix $\tilde{V}\in\R^{(n-r)\times r}$.
    Letting
    $V=\begin{psmallmatrix} \tilde{C}\\ C_1 \end{psmallmatrix}\tilde{A}^{-1}
    =\begin{psmallmatrix} U\\ \tilde{V} \end{psmallmatrix}$,
    we have 
    $\begin{psmallmatrix} \tilde{d}\\d \end{psmallmatrix}=V\tilde{a}$,
    $\begin{psmallmatrix} \tilde{C}\\C_1 \end{psmallmatrix}=V\tilde{A}$,
    $\begin{psmallmatrix} \tilde{D}&C_2^\top\\ C_2&D \end{psmallmatrix}\succeq V\tilde{A}V^\top$,
    as required. 
\end{proof}

The following lemma is the key to prove Theorem~\ref{thm:SDP_beat_LP}.

\begin{lemma}
\label{thm:SDP_beat_LP2}
    Let $G=(N,E)$ and consider Algorithm~\ref{alg:look ahead} applied to $\calV_{\LP}$ and $\calV_{\SDP}$. 
    Consider an arbitrary iteration of \PII\ on an optimal trajectory, starting with the selected set of vertices $S$ and the set of remaining vertices $I$.
    Then for $J\subsetneq I$,  
    if $\calV_{\LP}(I)>\calV_{\LP}(J)$, we have $\calV_{\SDP}(I)>\calV_{\SDP}(J)$.
\end{lemma}

Before proving Lemma~\ref{thm:SDP_beat_LP2}, we first show that it implies Theorem~\ref{thm:SDP_beat_LP}.

\begin{proof}[Lemma~\ref{thm:SDP_beat_LP2}$\implies$ Theorem~\ref{thm:SDP_beat_LP}]
    Let $S=\{v_1,\ldots,v_t\}$ be a stable set returned by Algorithm~\ref{alg:look ahead} with $\calV_{\SDP}$, where $v_j$ is the $j$-th vertex selected in line \ref{alg_step:select}. 
    We claim that $v_1, \dots, v_t$ are also valid choices of selected vertices when using Algorithm~\ref{alg:look ahead} with $\calV_{\LP}$.
    If we prove this claim, then $S$ would be an MWSS by the assumption on Algorithm~\ref{alg:look ahead} with~$\calV_{\LP}$.

    Let $S_i=\{v_1,\ldots,v_{i-1}\}$ be the set of selected vertices
    and let $I^{\SDP}_i \subset N\setminus S_i$ be the set of remaining vertices
    before $v_i$ is selected by Algorithm~\ref{alg:look ahead} with $\calV_{\SDP}$.
    We have that $S_{i+1}=S_{i}\cup \{v_i\}$ and $I^{\SDP}_{i+1}\subseteq I^{\SDP}_{i}\setminus (v_i\cup\delta_{v_i})$.
    We will prove by induction on $i$ that Algorithm~\ref{alg:look ahead} with $\calV_{\LP}$ can select the vertices in $S_i$ for the first $i-1$ rounds,
    and that the set of remaining vertices $I_{i}^{\LP}$ is a superset of $I_{i}^{\SDP}$.
    
    For the base case, $I_{1}^{\LP}$, $I_{1}^{\SDP}$ are the sets of remaining vertices after \PI of Algorithm~\ref{alg:look ahead}.
    Since $S_1=\emptyset$ is on an optimal trajectory, the LP solution is strictly complementary, and the SDP solution is in the relative interior of the optimal face,
    by Section~\ref{sec:select and discard} and Lemma~\ref{lem:SDP PhaseI}
    we have $I_{1}^{\LP}=I_{1}^{\SDP}$. 
    Thus, Algorithm~\ref{alg:look ahead} with $\calV_{\LP}$ can select $v_1\in I_{1}^{\LP}$.
    
    Assume now that Algorithm~\ref{alg:look ahead} with $\calV_{\LP}$ selects the vertices in $S_i$,
    and let $I_{i}^{\LP} \supseteq I_{i}^{\SDP}$ be the set of remaining vertices.
    Since $v_i \in I_i^{\SDP} \subset I_i^{\LP}$,
    then we can select vertex $v_i$ in Algorithm~\ref{alg:look ahead} with~$\calV_{\LP}$.
    Notice that \eqref{eq:implication} holds,
    as it is a special case of Lemma~\ref{thm:SDP_beat_LP2}.
    Hence, after $v_i$ is selected,
    any vertex in $I_{i}^{\SDP}$ (a subset of $I_i^{\LP}$) that is discarded using $\calV_{\LP}$
    can also be discarded when using $\calV_{\SDP}$.
    It follows that $I_{i+1}^{\SDP}\subseteq I_{i+1}^{\LP}$.
\end{proof}

\begin{proof}[Proof of Lemma~\ref{thm:SDP_beat_LP2}]
For the sake of contradiction, suppose that $\calV_{\SDP}(I)=\calV_{\SDP}(J)$.
The proof strategy consists of constructing two matrices $M_1,M_2$ with $M_1\succeq M_2$ but such that $\calV_{\LP}(J)=\calV_{\LP}(I)$ implies $M_1\not{\succeq} M_2$.

Let $J'\subseteq J$ be such that $\bar{Q}_{J'}$ is positive definite and $\rank(\bar{Q}_{J'})=\rank(\bar{Q}_{J})$.
Define $\delta_{I}:=\{i\in N\setminus I:\exists j\in I, ij\in E(G)\}$.
By Lemma~\ref{lem:row-dep}, there exists a matrix $V\in\R^{|I\setminus J'|\times |J'|}$ such that
\begin{align*}
M_1:=\begin{psmallmatrix}
    t^*& \bar{q}_{J'}^{\top} &\bar{q}_{J'}^{\top} V^\top &\bar{q}_{\delta_I}^\top \\
    \multicolumn{3}{c}{R}&\bar{Q}_{J',\delta_I} \\
    V\bar{q}_{J'} &V\bar{Q}_{J'}&\bar{Q}_{I\setminus J'}&\bar{Q}_{I\setminus J',\delta_I}\\
    \bar{q}_{\delta_I} &\bar{Q}_{\delta_I,J'}&\bar{Q}_{\delta_I,I\setminus J'}&\bar{Q}_{\delta_I}
\end{psmallmatrix}\succeq M_2:=\begin{psmallmatrix}
    t^*& \bar{q}_{J'}^{\top} &\bar{q}_{J'}^{\top} V^\top &\bar{q}_{\delta_I}^\top \\
    \multicolumn{3}{c}{R}&\bar{Q}_{J',\delta_I} \\
    \multicolumn{3}{c}{VR}&\bar{Q}_{I\setminus J',\delta_I}\\
    \bar{q}_{\delta_I} &\bar{Q}_{\delta_I,J'}&\bar{Q}_{\delta_I,I\setminus J'}&\bar{Q}_{\delta_I}
\end{psmallmatrix}.
\end{align*}
where 
\(
  \bar{Q}_{I\setminus J'}\succeq V \bar{Q}_{J'}V^\top
\), \(R:=\begin{psmallmatrix}    
\bar{q}_{J'}& \bar{Q}_{J'} &\bar{Q}_{J'}V^\top\end{psmallmatrix}\)

Our goal is to show that there exists $\nu\in\R^{\{0\}\cup J'\cup (I\setminus J')\cup\delta_I}$ such that
$\nu^\top M_1\nu<\nu^\top M_2\nu$, which would be a contradiction.
We proceed to construct the vector~$\nu$.
By the way $\calV_{\LP}$ is constructed and $\calV_{\LP}(J)<\calV_{\LP}(I)$, we know that there is an essential clique $C'\in \mathcal{C}(G|_I)$ such that $C'\cap J=\emptyset$.
And there exists a strictly essential clique $\bar{C}$ of $G$ such that $C'\subseteq \bar{C}$.
With this $\bar{C}$ and $\mu^*_{\bar{C}}>0$, we construct a vector $p_{\bar{C}}\in\R^{\{0\}\cup N}$ and a matrix $M$ as in Lemma~\ref{lem:LPDtoSDPD},
where $M$  is shown to be an optimal solution of \eqref{SDP-D} over~$G$.
By definition of $M$, it is clear that 
\( p_{\bar{C}}\in\range\left(M\right) \) by ~\cite[Theorem 1]{pseudo-inv}.
Let $\mathcal{H}(G)$ be the feasible set of~\eqref{SDP-D} defined over $G$.
Since
$\begin{psmallmatrix} t^* &\bar{q}^\top\\ \bar{q}&\bar{Q} \end{psmallmatrix}$
is in the relative interior of the optimal face of $\mathcal{H}(G)$,
or equivalently, the optimal face is the minimal face of $\mathcal{H}(G)$ containing
$\begin{psmallmatrix} t^* &\bar{q}^\top\\ \bar{q}&\bar{Q} \end{psmallmatrix}$, 
by \cite[Lemma 4]{diag_dom_matrices} we have
\[
p_{\bar{C}}\in\range\left(M\right)\subseteq\range
\begin{psmallmatrix} t^* &\bar{q}^\top\\ \bar{q}&\bar{Q}\end{psmallmatrix}.
\]
For the rest of the proof, let $\Xi$ be $N\setminus(I\cup\delta_I)$,
and consider the block structure 
\[
\begin{psmallmatrix}
    t^* &\bar{q}^\top\\
    \bar{q}&\bar{Q}
\end{psmallmatrix}=\begin{psmallmatrix}
    t^*& \bar{q}_{I}^\top &\bar{q}_{\delta_I}^\top & \bar{q}_{\Xi}^\top\\
    \bar{q}_{I}& \bar{Q}_{I} &\bar{Q}_{I,\delta_I}& 0\\
    \bar{Q}_{\delta_I}&\bar{Q}_{\delta_I,I} &\bar{Q}_{\delta_I} & \bar{Q}_{\delta_I,\Xi}\\
    \bar{q}_{\Xi}^\top& 0 & \bar{Q}_{\Xi,\delta_I}& \bar{Q}_{\Xi}
\end{psmallmatrix}.
\]

Let $\tilde{p}_{C'}\in \R^{\{0\}\cup {I\cup\Xi}}$ be the restriction of $p_{\bar{C}}$ to $\{0\}\cup I\cup \Xi$.
Since
$p_{\bar{C}}\in \range
\begin{psmallmatrix} t^* &\bar{q}^\top\\ \bar{q}&\bar{Q} \end{psmallmatrix}$,
then
$\tilde{p}_{C}\in\range
\begin{psmallmatrix}
    t^*& \bar{q}_{I}^\top & \bar{q}_{\Xi}^\top\\
    \bar{q}_{I}& \bar{Q}_{I} & 0\\
    \bar{q}_{\Xi}^\top& 0 & \bar{Q}_{\Xi}
\end{psmallmatrix}$.
Hence, there exists $\tilde{y}\in\R^{\{0\}\cup I\cup\Xi}$ such that
$\tilde{p}_{C} = \begin{psmallmatrix}
    t^*& \bar{q}_{I}^\top & \bar{q}_{\Xi}^\top\\
    \bar{q}_{I}& \bar{Q}_{I} & 0\\
    \bar{q}_{\Xi}^\top& 0 & \bar{Q}_{\Xi}
\end{psmallmatrix}\tilde{y}$.
Consider the vectors
\begin{align*}
\ell:=
\begin{psmallmatrix}
    1&
    \tilde{x}^\top_{J'}&
    \tilde{x}^\top_{I\setminus J'}&
    \tilde{x}^\top_{\delta_I}&
\end{psmallmatrix}^\top,\quad
\bar{y}:=
\begin{psmallmatrix}
    \tilde{y}_{\{0\}}^\top&
    \tilde{y}_{J'}^\top&
    \tilde{y}_{I\setminus J'}^\top&
    0
\end{psmallmatrix},\quad
\nu_{\gamma} = \bar y + \gamma \ell
\in
\R^{\{0\}\cup J'\cup (I\setminus J')\cup \delta_I}
\end{align*}
It remains to show that
$\nu_{\gamma}^\top M_1\nu_{\gamma}<\nu_{\gamma}^\top M_2\nu_{\gamma}$ for a suitable choice of~$\gamma$.

Let us first evaluate $M_1 \bar y$ and $M_2 \bar y$.
Consider the equation
$\begin{psmallmatrix}
    t^*& \bar{q}_{I}^\top & \bar{q}_{\Xi}^\top\\
    \bar{q}_{I}& \bar{Q}_{I} & 0\\
    \bar{q}_{\Xi}^\top& 0 & \bar{Q}_{\Xi}
\end{psmallmatrix}\tilde{y} = \tilde{p}_{C}$.
By replacing $I$ by $J'\cup (I\setminus J')$ and adding rows and columns corresponding to $\delta_I$,
we can rewrite this equation as
\[\begin{psmallmatrix}
    t^*& \bar{q}_{J'}^{\top} &\bar{q}_{J'}^{\top} V^\top &\bar{q}_{\delta_I}^\top &\bar{q}_{\Xi}^\top\\
    \bar{q}_{J'}& \bar{Q}_{J'} &\bar{Q}_{J'}V^\top&\bar{Q}_{J',\delta_I}&0 \\
    V\bar{q}_{J'} &V\bar{Q}_{J'}&\bar{Q}_{I\setminus J'}&\bar{Q}_{I\setminus J',\delta_I}&0\\
    \bar{q}_{\delta_I} &\bar{Q}_{\delta_I,J'}&\bar{Q}_{\delta_I,I\setminus J'}&\bar{Q}_{\delta_I}&\bar{Q}_{\delta_I,\Xi}\\
    \bar{q}_{\Xi}&0&0&\bar{Q}_{\Xi,\delta_I}&\bar{Q}_{\Xi}
\end{psmallmatrix}
\begin{psmallmatrix}
    \tilde{y}_{\{0\}}\\\tilde{y}_{J'}\\\tilde{y}_{I\setminus J'}\\0\\ \tilde{y}_{\Xi}
\end{psmallmatrix}
=(
    \underbrace{\sqrt{\mu^*_{\bar{C}}}}_{\{0\}}~\underbrace{0}_{J'}~\underbrace{-\sqrt{\mu^*_{\bar{C}}}~\cdots~-\sqrt{\mu^*_{\bar{C}}}}_{C'}~\underbrace{0}_{I\setminus J'\setminus C'}\underbrace{*\cdots*}_{N\setminus I}
)^\top,\]
where $*\cdots *$ represent irrelevant entries.
Since $\bar{y}=\begin{psmallmatrix}
    \tilde{y}_{\{0\}}&\tilde{y}_{J'}&\tilde{y}_{I\setminus J'}&0
\end{psmallmatrix}$, then
\[
M_1\bar{y}=\begin{psmallmatrix}
    t^*& \bar{q}_{J'}^{\top} &\bar{q}_{J'}^{\top} V^\top &\bar{q}_{\delta_I}^\top \\
    \bar{q}_{J'}& \bar{Q}_{J'} &\bar{Q}_{J'}V^\top&\bar{Q}_{J',\delta_I}\\
    V\bar{q}_{J'} &V\bar{Q}_{J'}&\bar{Q}_{I\setminus J'}&\bar{Q}_{I\setminus J',\delta_I}\\
    \bar{q}_{\delta_I} &\bar{Q}_{\delta_I,J'}&\bar{Q}_{\delta_I,I\setminus J'}&\bar{Q}_{\delta_I}
\end{psmallmatrix}\bar{y}=(
    \underbrace{\alpha}_{\{0\}}\quad \underbrace{0}_{J'}\quad\underbrace{-\sqrt{\mu^*_{\bar{C}}}~\cdots~-\sqrt{\mu^*_{\bar{C}}}}_{C'}\quad \underbrace{0}_{I\setminus J'\setminus C'}\underbrace{*\cdots*}_{\delta_I}
)^\top ,
\]
for a constant $\alpha$.
While for $M_2$, by $[M_1]_{J',:}~\bar{y}=[R\quad \bar{Q}_{J',\delta_I}]\bar{y}=0$, the row dependence of $M_2$ and $\bar{y}_{\delta_I}=0$, we have 
\[M_2\bar{y}=(
    {\alpha}\quad {0}\quad{0}\quad {0}\quad {*\cdots*}
)^\top.\]
 Since we are following the optimal trajectory, all vertices in $\delta_I$ are discarded, and $I\setminus J$ contains an essential clique.

We now evaluate $M_1 \ell$ and $M_2 \ell$.
Since we are following an optimal trajectory and reach $I$, we know there exists an MWSS containing no vertices in $\delta_I$.
Thus, by~\eqref{relation:perfect graph}, there exists an optimal solution $( \tilde{x}, \tilde{X} )$ 
of ~\eqref{SDP-P} such that 
$\tilde{x}_{\delta_I}=0$.
By complementary slackness,
\(
\begin{psmallmatrix} t^* &\bar{q}^\top\\ \bar{q}&\bar{Q} \end{psmallmatrix}
\begin{psmallmatrix} 1\\ \tilde{x} \end{psmallmatrix} = 0.
\)
Also, since $\bar{C}$ is strictly essential over $G$, $\bar{C}\setminus C'\subseteq \delta_I$ and $\tilde{x}_{\delta_I}=0$, $\tilde{x}_{j}>0$ for some $j\in C'\subseteq {I\setminus J}$.
Since $\ell:=\begin{psmallmatrix}
    1&
    \tilde{x}_{J'}^\top&
    \tilde{x}_{I\setminus J'}^\top&
    \tilde{x}_{\delta_I}^\top&
\end{psmallmatrix}^\top,$
then $\ell_{j}>0$ and $\ell_{\delta_I}=0$.
Since
\(\bar{Q}_{I,\Xi}=0\), 
then
\[
\begin{psmallmatrix} t^* &\bar{q}^\top\\ \bar{q}&\bar{Q} \end{psmallmatrix}
\begin{psmallmatrix} 1\\ \tilde{x} \end{psmallmatrix}=0
\quad\implies\quad
\begin{psmallmatrix}
    \multicolumn{3}{c}{R}&\bar{Q}_{J',\delta_I} \\
    V\bar{q}_{J'} &V\bar{Q}_{J'}&\bar{Q}_{I\setminus J'}&\bar{Q}_{I\setminus J',\delta_I}
\end{psmallmatrix}\ell=0 ,
\]
and by the row dependence of $M_2$, 
we have \(M_1\ell=(
    {\lambda}\;\; {0}\;\;{0}\;\; {0}\;\; {*\cdots*}
)^\top=M_2\ell\)
for a constant~$\lambda$.

We proceed to evaluate
$\nu_{\gamma}^\top M_1\nu_{\gamma}$ and $\nu_{\gamma}^\top M_2\nu_{\gamma}$.
Since $\ell_{\delta_I}=\tilde{x}_{\delta_I}=0$,
\begin{align*}
\nu_{\gamma}^{\top}M_1\nu_{\gamma}
&=\bar{y}^\top M_1\bar{y}+2\gamma\ell^\top M_1\bar{y}+\gamma^2\ell^\top M_1\ell
=\bar{y}^\top M_1\bar{y}+2\gamma\alpha+2\gamma\sum_{i\in C'}\left(-\sqrt{\mu^*_{\bar{C}}}\right)\ell_i+0+\lambda\gamma^2\\
\nu_{\gamma}^{\top}M_2\nu_{\gamma}&=\bar{y}^\top M_2\bar{y}+2\gamma\ell^\top M_2\bar{y}+\gamma^2\ell^\top M_2\ell=\bar{y}^\top M_2\bar{y}+2\gamma\alpha+\lambda\gamma^2.\end{align*}
Since $-\sqrt{{\mu^*}_{\bar{C}}}<0$, $\ell\geq 0$ and $\ell_j=\tilde{x}_j>0$, $j\in C'\subseteq I\setminus J$, 
then $\nu_{\gamma}^\top M_1\nu<\nu_{\gamma}^\top M_2\nu_{\gamma}$ for large enough $\gamma$.
This contradicts the fact that $M_1 \succeq M_2$.
\end{proof}
We emphasize that even though the analysis of Algorithm \ref{alg:look ahead} with $\calV_{\SDP}$ relies on $\calV_{\LP}$, there is no need to compute $\calV_{\LP}$ in practice.

\subsection{Computing all Maximum Weighted Stable Sets in Polynomial Time}

We now show how to modify Algorithm \ref{alg:look ahead} to efficiently obtain all MWSS for unipolar, chordal and co-chordal graphs. In addition to the number of vertices $n$, we measure the algorithm's efficiency with the number of solutions it generates; Algorithm \ref{alg:all MWSS} describes the approach.

\begin{algorithm}[htb]
\begin{algorithmic}[1]
\State\textbf{Input:} $G=(N,E)$, a weight function $w$, optimal $x^*\in\STAB(G)$, and a tight VFA $\calV$.
\item[\textbf{Phase I:}]
\State $\mathcal{S} \gets \emptyset$
\Comment{List of stable sets, initially empty.}
\State $I \gets \{ i \in N : x_i^* >0 \}$\label{alg_step:subopt by primal}
\Comment{Discard vertices not in any optimal set.}
\State Index elements in $I$ from $1$ to $|I|$.
\State $T \gets \left[(I,\emptyset)\right]$\Comment{List of unvisited tree nodes.}

\item[\textbf{Phase II:}]
\While{$T\neq\emptyset$}
    \State $(J,S)\gets \operatorname{pop}(T)$\Comment{Get the last element in $T$ and discard it.}
    \If{$J=\emptyset$}
    \State $\mathcal{S}\gets\mathcal{S}\cup\{S\}$
    \Else
    \State{$v\gets u\in J$ with the smallest index}
        \State{$I'\gets J\setminus\{v\}$} \Comment{Test deleting $v$}
        \While{$\exists i\in I'$ with $\calV(I')-\calV(I'\setminus(\{i\}\cup\delta_i))> w_i$}\label{alg_step:discard}
            \State $I' \gets I'\setminus \{i\}$
        \EndWhile
        \If{$\calV(J)=\calV(I')$}
            \State $T\gets \operatorname{append}(T,(I',S))$ \Comment{Append the element to the end of $T$}
        \EndIf 
        \State $I' \gets J$
        \Comment{Test adding $v$ to $S$}
        \State $I' \gets I'\setminus(\{v\}\cup\delta_v)$ 
        \While{$\exists i\in I'$ with $\calV(I')-\calV(I'\setminus(\{i\}\cup\delta_i))> w_i$}\label{alg_step:discard}
            \State $I' \gets I'\setminus \{i\}$
        \EndWhile
        \If{$\calV(J)=\calV(I')+w_v$}
            \State $T\gets \operatorname{append}(T,(I',S\cup\{v\}))$.\Comment{Append the element to the end of $T$}
        \EndIf
    \EndIf
\EndWhile
\State\textbf{Output:} Return $\mathcal{S}$, a list of stable sets of $G$.
\end{algorithmic}
\caption{Retrieving all MWSS from a tight VFA with one-step look ahead} 
\label{alg:all MWSS}
\end{algorithm}

To establish the algorithm's correctness,
we first show that, for unipolar, chordal and co-chordal graphs, if a vertex is essential, deleting it decreases the VFA.
\begin{lemma}
\label{lem: not delete ess. vertex}
    Let $G=(N,E)$ be either unipolar, chordal or co-chordal and consider Algorithm~\ref{alg:look ahead} applied to $\calV_{\LP}$.
    Consider an iteration of \PII starting with set of vertices $I$; if $v$ is in every MWSS of $G|_I$, then 
    $\calV_\LP(I)>\calV_\LP(I\setminus\{v\})$.
\end{lemma}
\begin{proof}
    For contradiction, suppose that $\calV_\LP(I)=\calV_\LP(I\setminus\{v\})$ with the set of selected vertices $S$.
    By the proof of Theorem~\ref{thm:main_LP}~\ref{thm:main:unipolar}, for each later iteration starting with the set of vertices $J$, there exists a vertex $u\in J$ such that 
    $\calV_\LP(J) = \calV_\LP(J\setminus(\{u\}\cup\delta_u)$.
    Thus, by $\calV_\LP$ being tight, the returned stable set $\hat{S}$ is an MWSS of $G$.
    However, since $v\notin \hat{S}$ and every MWSS of $G|_I$ contains $v$, we know $\hat{S}\setminus S$ is not an MWSS of $G|_I$, so $\hat{S}$ is not an MWSS of $G$, a contradiction.
\end{proof}

\begin{corollary}
\label{cor: not delete ess. vertex}
        Let $G=(N,E)$ be either unipolar, chordal or co-chordal and 
        consider Algorithm~\ref{alg:look ahead} applied to $\calV_{\SDP}$.
    Consider an iteration of \PII starting with set of vertices $I$; if $v$ is in every MWSS of $G|_I$, then 
    $\calV_\SDP(I)>\calV_\SDP(I\setminus\{v\})$.
\end{corollary}
\begin{proof}
    For contradiction, suppose that $\calV_\SDP(I)=\calV_\SDP(I\setminus\{v\})$.
    Let $\hat{S}$ be a stable set returned by Algorithm~\ref{alg:look ahead}.
    By Lemma~\ref{thm:SDP_beat_LP2}, $\calV_\LP(I)=\calV_\LP(I\setminus\{v\})$.
    Thus by the proof of Lemma~\ref{thm:SDP_beat_LP2} implying Theorem~\ref{thm:SDP_beat_LP}, we know $\hat{S}$ can be returned by Algorithm~\ref{alg:look ahead} with $\calV_\LP$, so it is an MWSS of $G$.
    Using the same argument in the proof of Lemma~\ref{lem: not delete ess. vertex}, we see that $\hat{S}$ is not an MWSS of $G$, a contradiction.
\end{proof}

Thus, if one is on an optimal trajectory with $\calV_\SDP(I)=\calV_\SDP(I\setminus\{v\})$, then one stays on an optimal trajectory after deleting $v$.
Now we show the main result of this subsection.
\begin{theorem}
\label{thm:all mwss}
     For unipolar, chordal and co-chordal graphs, if there are $\ell$ different MWSS, Algorithm~\ref{alg:all MWSS} with $\calV_\SDP$ computes all $\ell$ MWSS in $\mathcal{O}(n^2\ell)$ VFA evaluations.
\end{theorem}
\begin{proof}
Let $G=(N,E)$ be  unipolar, chordal or co-chordal, and 
let $S_1,\ldots,S_\ell$ be the collection of all MWSS.
    By Lemma~\ref{lem:SDP PhaseI}, $I:=\bigcup_{i=1}^{\ell}S_i$;
    without loss of generality, assume $|I|=n$ and $I=\{1,\ldots,n\}$.

Consider a binary tree $T$ where each node $(J,S)$ indicates that $J$ is the vertex set of the remaining graph and $S\subseteq I \setminus J$ is the selected partial stable set. 
For the node $(J,S)$, let $v$ be the vertex in $J$ with the smallest index. 
    Algorithm~\ref{alg:all MWSS} creates children for this node in the following ways.
\begin{enumerate}
\item Consider deleting $v$.
After tentatively deleting $v$, if $\calV_\SDP$ does not decrease in value, create a child node $(I',S)$, where $I'$ is the remaining vertex set after deleting $v$ and deleting all bad vertices from $J$.
        If $\calV_\SDP$ decreases in value, by Corollary~\ref{cor: not delete ess. vertex}, $v$ is an element of every MWSS of the remaining set; do not create a child node for deleting $v$.
        
\item Consider adding $v$ to $S$ and applying a look-ahead. After tentatively choosing $v$, if $v$ is a good choice, create a child node $(I',S\cup\{v\})$, where $I'$ is the remaining vertex set after choosing $v$, deleting its neighbors, and deleting all bad vertices from $J$.
If $v$ is a bad choice, do not create a child node for adding $v$.
\end{enumerate}
Thus, at most two children are created, one for selecting $v$ and one for deleting $v$. 
Since Algorithm~\ref{alg:all MWSS} considers and discards the lowest-indexed remaining vertex in each iteration, it does not create duplicate nodes with the same pair $(J,S)$. Each node is considered exactly once, and leaves are nodes of the form $(\emptyset, S)$.

\begin{claim}
Every MWSS is represented by a unique leaf in the binary tree.
\end{claim}
\begin{proof}
Let $\hat{S}=\{v_1,v_2,v_3,\dotsc,v_k\}$ be an arbitrary MWSS, where $v_1 < v_2 < v_3 < \dotsb < v_k $.
    By Theorem~\ref{thm:main_LP} \ref{thm:main:unipolar} and Lemma \ref{lem:fund_VFA_prop}, $\hat{S}$ can be possibly returned by Algorithm~\ref{alg:look ahead} and hence by Algorithm~\ref{alg:all MWSS}.
    And since $\hat{S}$ is an MWSS, $J$ is empty after $v_k$ is selected, so it is represented uniquely by a leaf. 
    In particular, it is represented uniquely by the path from the root to the leaf following the algorithm's order of selecting and deleting vertices according to the indexing.
\end{proof}

\begin{claim}
Every leaf of the binary tree represents an MWSS of the graph.
\end{claim}
\begin{proof}
By Theorem~\ref{thm:main_LP}~\ref{thm:main:unipolar}, Theorem~\ref{thm:SDP_beat_LP}, Lemma~\ref{lem: not delete ess. vertex} and Corollary~\ref{cor: not delete ess. vertex}, for all possible actions (children) at each node, Algorithm~\ref{alg:all MWSS} stays on an optimal trajectory.
Thus, when the remaining vertex set $J$ is not empty, there exists at least one optimal choice in $J$ to select. 
If the smallest-indexed vertex $v\in J$ is an optimal choice, then a child node for ``selecting $v$'' is in the tree; if it is not in every MWSS of $G'$, then a child node for ``deleting $v$'' is also created. 
Thus, for every leaf in the binary tree, $J$ is empty.
The path from the root to that leaf represents a unique stable set for $G$ made of all nodes selecting vertices and it can be returned by Algorithm~\ref{alg:look ahead} and Algorithm~\ref{alg:all MWSS}.
By Theorem~\ref{thm:main_LP} \ref{thm:main:unipolar} and Theorem~\ref{thm:SDP_beat_LP}, it is an MWSS.
\end{proof}

By the two claims above, we can compute all MWSS by running Algorithm~\ref{alg:all MWSS}.
Since each node of the binary tree $T$ is visited exactly once, the number of iterations is the number of nodes.
The depth of $T$ is at most $n$ and it has $\ell$ leaves, so the number of nodes is $\mathcal{O}(n\ell)$.
Each node involves $\mathcal{O}(n)$ VFA evaluations,
resulting in $\mathcal{O}(n^2\ell)$ total VFA evaluations for Algorithm~\ref{alg:all MWSS}.
\end{proof}

Since $\calV_\SDP$ can be computed in polynomial time given the SDP solution, Theorem~\ref{thm:all mwss} shows that all MWSS can be returned in time that is polynomial time in the number of vertices and the number of different MWSS.


\section{Computational Experiments}
\label{sec:comp_exp}

In this section, we discuss computational experiments on a variant of Algorithm~\ref{alg:look ahead} applied to arbitrary graphs; 
we detail this version in Algorithm~\ref{alg:VFA2}. 
This algorithm variant does not assume that the VFA is tight,
so it can be applied with imperfect graphs.
In particular, it does not discard vertices $j$ for which
$\calV(I) > \calV(I\setminus(\{j\}\cup\delta_j)) + w_j$,
but instead selects a vertex by maximizing
$\calV(I\setminus(\{j\}\cup\delta_j)) + w_j$;
this is how VFAs are typically used to generate heuristic solutions in approximate DP. 
To reduce the number of VFA evaluations,
Algorithm~\ref{alg:VFA2} selects isolated vertices or optimal leaves of $G|_I$ whenever possible; the latter are leaves whose weight matches or exceeds that of their only neighbor.


\begin{algorithm}[htb]
\begin{algorithmic}
\State\textbf{Input:} $G=(N,E)$, a weight function $w$ and a VFA $\calV$.
\State $S \gets \emptyset$ and
$I \gets N$
\Comment{Start with an empty stable set.}
\While{$I\neq\emptyset$}
    \State 
    $i \gets \begin{cases}
        \text{an isolated vertex or optimal leaf of }G|_I & \textbf{if one exists}\\
        \argmax_{j\in I} (\calV(I\setminus(\{j\}\cup\delta_j)) + w_j) & \textbf{otherwise}
    \end{cases}$
    \State $S \gets S \cup \{i\}$
    \Comment{Select $i$ to join the stable set.}
    \State $I \gets I\setminus(\{i\}\cup\delta_i)$
    \Comment{Discard $i$ and $\delta_i$ from remaining vertices.}
\EndWhile
\State\textbf{Output:} Return $S$, a stable set of $G$.
\end{algorithmic}
\caption{Retrieving a stable set from an arbitrary VFA}\label{alg:VFA2}
\end{algorithm}

We performed the experiments on a 2021 MacBook Pro with an 8-core Apple M1 Pro CPU and 16 GB of memory.
We solve the primal-dual pair \eqref{SDP-P}, \eqref{SDP-D} using either the commercial solver COPT \cite{copt} for dense graphs
or the SDP solver HALLaR \cite{monteiro2024lowrankaugmentedlagrangianmethod} for sparse graphs;
we implement Algorithm~\ref{alg:VFA2} in Julia.
We solve all SDPs within a relative precision of $\varepsilon_{\SDP} = 10^{-5}$.
We evaluate $\calV_{\SDP}$ using the quadratic minimization characterization from Lemma~\ref{lem:comp equiv}, $\calV_{\SDP}(S)=-\min_{y\in\R^S}(y^\top Q_S y - 2 q_S^\top y)$,
including a regularization term $\lambda \lVert y \rVert ^2$ with $\lambda = 10^{-4}$.
We solve the VFA quadratic minimization to a precision of $\varepsilon_{\mathrm{VFA}} = 10^{-6}$ with an iterative method,
warm-started with the previous solution;
this characterization allows us to trade precision for efficiency.
Note that it is possible to get better results by tuning the precision parameters to each individual instance; 
however, we present results using uniform parameters for consistency.

In our first set of experiments, we test Algorithm~\ref{alg:VFA2} on special classes of perfect graphs.
More precisely, we generate chordal and co-chordal graphs using the random generators with algorithms \texttt{"growing"}, \texttt{"connecting"} and \texttt{"pruned"} from SageMath~\cite{sagemath} with default parameters.
We also generate uniformly random generalized split graphs using the algorithm by McDiarmid and Yolov \cite{mcdiarmid2017randomperfectgraphs}.
In total, we tested 300 chordal and co-chordal graphs (100 generated by each algorithm)
and 100 generalized split graphs, using the cardinality objective, $w_i=1$ for every $i\in N$. 
For all instances, we obtain an optimal solution, which agrees with our theoretical results.

In the second set of experiments, we test Algorithm~\ref{alg:VFA2} on graphs that are not necessarily perfect, including
graphs from DIMACS2 \cite{ams1996dimacs}, GSet \cite{gset} DIMACS10 \cite{dimacs10}, and SNAP \cite{snapnets}; 
for DIMACS2, we use graph complements, as the instances were designed for the maximum clique problem.
As before, we use the cardinality objective.
We compare Algorithm~\ref{alg:VFA2} against the Benson-Ye (BY) rounding method~\cite{BensonYe}
and the Walteros-Buchanan (WB) fixed-parameter tractable algorithm \cite{walteros_why_2020}. 
We run BY on the same graphs as Algorithm~\ref{alg:VFA2}, and WB on the complements, as it is designed for the maximum clique problem.
BY is randomized and inexpensive, so
we run it $|N|$ times as suggested in \cite{BensonYe} and report its best and average performance.
For all three methods,
we first apply a pre-processing step so that the graph is connected and has no leaves:
if it is disconnected, we consider each component separately,
and if there is a leaf, we select it.
We implement BY and use the WB code provided in \cite{walteros_why_2020}.
The latter is an exact algorithm
taking exponential time in the worst case,
but it performs well in practice in many large instances.
However, some instances are challenging for WB; for more details, we refer the reader to \cite{walteros_why_2020}.
We only include results for instances that are solved in 30 minutes and leave the solution time blank otherwise. 

Tables~\ref{table:comp_dimacs2} and~\ref{table:comp_gset} respectively summarize the results for DIMACS2 (78 graphs) and GSet (71 graphs). 
Table~\ref{table:comp_dimacs10} summarizes the results for the larger instances from DIMACS10 (19 graphs) and SNAP (7 graphs). 
Overall, Algorithm~\ref{alg:VFA2} significantly outperforms the BY average in every single instance, and also beats the BY best solution (often significantly) except in instances \texttt{san200-0-7-2} and \texttt{san200-0-9-3} from DIMACS2. 
For these two instances, we can obtain an optimal solution by tuning the parameters
$\lambda, \varepsilon_{\mathrm{VFA}}$,
but we do not include those results for consistency.
WB returns the optimal value whenever it terminates within the time limit, but does not terminate for many of our instances.  
We highlight in bold the instances where 
Algorithm~\ref{alg:VFA2} matches $\alpha$ or the best known bound per \cite{prosser2012exactalgorithmsmaximumclique,walteros_why_2020}.

Our algorithm's performance is noteworthy, particularly in the large instances. The stable set generated by the algorithm frequently attains the stability number $\alpha$ when this quantity is known; in instances where it falls short or where $\alpha$ is unknown, our method still significantly outperforms the BY method.  
Finally, note that the computational effort to run Algorithm~\ref{alg:VFA2} is typically much smaller than the cost of solving the SDP, so we expect it to perform much more efficiently than~\cite{GROTSCHEL1984325} in practice.

To highlight the versatility of our algorithm, we conducted a third set of experiments on instances with a weighted objective.  
We reuse the 300 chordal and co-chordal graphs and the 100 generalized split graphs generated in the first set of experiments.
For an instance with name $G$ and $n$ vertices, we initialized a random number generator \texttt{MersenneTwister} in Julia with the per-instance seed \texttt{hash((20260226, G, n))} and sampled $n$ weights uniformly from $\{1,\dots,10\}$.
For all these instances, Algorithm~\ref{alg:VFA2} obtains a maximum weighted stable set, which agrees with our theoretical results.

We also conducted an analogous set of experiments with weighted-objective instances based on the DIMACS2 dataset. 
The computational setup is identical to the earlier experiments with the following exceptions. 
We solved instance \texttt{MANN-a45} using COSMO \cite{Garstka_2021} because COPT encountered numerical issues and reported an erroneous bound;
similarly, we solved instances \texttt{c2000.5, c2000.9, c4000.5, keller6} with COSMO, to within a relative precision of $\varepsilon_\SDP=10^{-4}$, because the weighted objective significantly increased computation times in these instances. 
The results are summarized in Table~\ref{table:weighted_comp_dimacs_full}. 
As before, these graphs are not necessarily perfect, so the Lovász theta function bound $\vartheta(G, w)$ 
is only an upper bound benchmark. 
Since the exact weighted stability numbers for these instances are unknown and the WB algorithm 
is limited to the unweighted variant, we omit the corresponding baseline columns. 
We highlight in bold the instances where Algorithm~\ref{alg:VFA2} matches the upper bound $\vartheta$. 
In all reported instances where $\vartheta(G,w)$ is integral, Algorithm~\ref{alg:VFA2} matches $\vartheta(G,w)$ and is therefore optimal.

\section{Conclusions and Future Directions}

We provide a novel rounding scheme for the Lov\'asz theta function to solve the MWSS problem.
Algorithm~\ref{alg:look ahead} relies on a VFA constructed from the optimal solution of the SDP.
Theorem~\ref{thm:main_SDP} guarantees that Algorithm~\ref{alg:look ahead} paired with VFA $\calV_{\SDP}$ solves the MWSS problem for several important subclasses of perfect graphs, which asymptotically cover almost all perfect graphs;
the algorithm requires only the initial SDP solution. 
The algorithm can also be modified to efficiently output all MWSS for unipolar, chordal and co-chordal graphs.
To the best of our knowledge, this is the only known rounding strategy for the Lov\'asz theta function that solves the MWSS problem for large sub-classes of perfect graphs.

Our computational experiments show that Algorithm~\ref{alg:VFA2},
a simplified variant of Algorithm~\ref{alg:look ahead},
works well in practice.
Algorithm~\ref{alg:VFA2} recovers an MWSS for all instances of perfect graphs we tested.
The algorithm performs very well even for imperfect graphs, matching the best known bound in many of the instances we tested.
We believe that Algorithm~\ref{alg:VFA2} should perform well on graphs for which the theta function $\vartheta(G)$ is close to the stability number $\alpha(N)$.
Importantly, the computational cost of our rounding procedure is much lower than the cost of solving the SDP with a state-of-the-art method.

None of our computational experiments use look-ahead.
We conjecture that the requirement of using look-ahead is only an artifact of our proof technique,
and that Algorithm~\ref{alg:look ahead} without look-ahead still returns an MWSS for generalized split graphs, chordal, and co-chordal graphs.

Our analysis of $\calV_{\SDP}$ uses the LP-based VFA $\calV_{\LP}$.
The need for $\calV_{\LP}$ comes from its combinatorial interpretation,
and we do not have a similar interpretation for $\calV_{\SDP}$.
An interesting future direction is finding a combinatorial interpretation of $\calV_{\SDP}$, or perhaps constructing a different VFA that optimizes the stable set problem for a larger family of perfect graphs. 

Further research includes applying similar techniques to related problems, such as the dynamic stable set problem proposed in \cite{DNP}. 

{
    \begin{table}
    \normalsize
    \centering
    \resizebox{0.85\columnwidth}{!}{\begin{tabular}{ccccc|cccc|Hcccc}
        \toprule
        
        \multicolumn{5}{c|}{Graph}& \multicolumn{4}{c|}{Solution value} & \#VFA & \multicolumn{4}{c}{Time (s)} \\
        Name & $|N|$ & $|E|$ & $\vartheta$& $\alpha$ & Alg.\ref{alg:VFA2} & BY${}_{\text{avg}}$ & BY${}_{\text{best}}$ & WB & calls & SDP & Alg.\ref{alg:VFA2} & BY${}_{\text{total}}$ &WB${}$\\
        \midrule
        \textbf{MANN-a27} & 378 & 702 & 132.77 & 126 & 126 & 113.2 & 124 & 126 & 3521 & 7.92 & 1.02 & 0.02 & 0.16 \\ 
        MANN-a45 & 1035 & 1980 & 356.05 & 345 & 343 & 307.06 & 339 & 345 & 17591 & 38.89 & 10.76 & 0.14 &  129.97\\ 
        MANN-a81 & 3321 & 6480 & 1126.64 & 1100 & 1098 & 985.88 & 1089 & 1100 & 76335 & 543.92 & 151.41 & 1.91 &  55958.3\\ 
        \textbf{MANN-a9} & 45 & 72 & 17.48 & 16 & 16 & 14.64 & 16 & 16 & 138 & 0.2 & 1.99 & 0.06 & $<$0.01 \\ 
        brock200-1 & 200 & 5066 & 27.46 & 21 & 19 & 7.11 & 13 & 21 & 789 & 3.4 & 0.85 & 0.05 & 12.25 \\ 
        brock200-2 & 200 & 10024 & 14.23 & 12 & 10 & 2.66 & 6 & 12 & 426 & 18.87 & 0.57 & 0.08 & 0.19 \\ 
        brock200-3 & 200 & 7852 & 18.82 & 15 & 13 & 4.22 & 8 & 15 & 536 & 9.59 & 0.65 & 0.06 & 0.99 \\ 
        brock200-4 & 200 & 6811 & 21.29 & 17 & 14 & 4.73 & 9 & 17 & 631 & 6.93 & 0.74 & 0.05 & 3.22 \\ 
        brock400-1 & 400 & 20077 & 39.7 & 27 & 22 & 6.4 & 12 &  & 1682 & 142.46 & 2.98 & 0.3 &  \\ 
        brock400-2 & 400 & 20014 & 39.56 & 29 & 21 & 6.46 & 12 &  & 1840 & 150.39 & 3.04 & 0.32 &  \\ 
        brock400-3 & 400 & 20119 & 39.48 & 31 & 24 & 6.57 & 12 &  & 1750 & 142.73 & 2.66 & 0.31 &  \\ 
        brock400-4 & 400 & 20035 & 39.6 & 33 & 22 & 6.57 & 12 &  & 1772 & 143.57 & 3.73 & 0.32 &  \\ 
        brock800-1 & 800 & 112095 & 42.22 & 23 & 20 & 3.78 & 9 &  & 2610 &183.06& 10.32 & 3.34 &  \\ 
        brock800-2 & 800 & 111434 & 42.47 & 24 & 18 & 3.78 & 9 &  & 2728 & 177.83 & 10.85 & 3.31 &  \\ 
        brock800-3 & 800 & 112267 & 42.24 & 25 & 19 & 3.66 & 9 &  & 2733 & 180.22 & 10.51 & 3.37 &  \\ 
        brock800-4 & 800 & 111957 & 42.35 & 26 & 20 & 3.58 & 9 &  & 2656 & 177.56 & 11.53 & 3.34 &  \\ 
        \textbf{c-fat200-1} & 200 & 18366 & 12 & 12 & 12 & 3.46 & 12 & 12 & 210 & 44.17 & 0.81 & 0.16 & $<$0.01\\ 
        \textbf{c-fat200-2} & 200 & 16665 & 24 & 24 & 24 & 23.91 & 24 & 24 & 210 & 0.45 & 0.27 & 0.16 &  $<$0.01\\ 
        \textbf{c-fat200-5} & 200 & 11427 & 60.35 & 58 & 58 & 31.38 & 58 & 58 & 260 & 1.14 & 0.04 & 0.1 & $<$0.01 \\ 
        \textbf{c-fat500-10} & 500 & 78123 & 126 & 126 & 126 & 114.32 & 126 & 126 & 565 & 12.40 & 1.1 & 1.66 &  0.02\\ 
        \textbf{c-fat500-1} & 500 & 120291 & 14 & 14 & 14 & 3.18 & 14 & 14 & 510 & 4.17 & 1.83 & 2.12 & $<$0.01 \\ 
        \textbf{c-fat500-2} & 500 & 115611 & 26 & 26 & 26 & 5.86 & 26 & 26 & 510 & 18.12 & 1.27 & 2.03 & $<$0.01 \\ 
        \textbf{c-fat500-5} & 500 & 101559 & 64 & 64 & 64 & 59.12 & 64 & 64 & 510 & 201.85 & 0.84 & 1.99 & $<$0.01 \\ 
        c1000.9 & 1000 & 49421 & 123.49 & $\geq 68$ & 62 & 18.42 & 33 &  & 9690 & 675.64 & 32.1 & 1.88 &  \\ 
        \textbf{c125.9} & 125 & 787 & 37.81 & 34 & 34 & 21.76 & 32 &  & 530 & 3.24 & 0.28 & $<$0.01 & 0.45 \\ 
        c2000.5 & 2000 & 999164 & 44.87 & 16 & 14 & 1.84 & 7 &  & 4351 & 1556.79 & 117.96 & 82.91 &  \\ 
        c2000.9 & 2000 & 199468 & 178.93 & $\geq80$ & 68 & 16.41 & 31 &  & 21745 & 35044.21 & 272.68 & 19.54 &  \\ 
        c250.9 & 250 & 3141 & 56.24 & 44 & 40 & 21.36 & 31 &  & 1648 & 6.53 & 0.75 & 0.03 &  \\ 
        c4000.5 & 4000 & 3997732 & 64.57 & 18 & 15 & 1.68 & 6 &  & 8619 & 518753.57 & 4840.99 & 680.11 &  \\ 
        c500.9 & 500 & 12418 & 84.2 & $\geq 57$ & 52 & 20.51 & 35 &  & 4105 & 49.55 & 4.78 & 0.25 &  \\ 
        dsjc1000\_5 & 1000 & 249674 & 31.89 & 15 & 12 & 2.03 & 6 & 15 & 2304 & 13610.94 & 75.32 & 9.63 &  3045.79\\ 
        dsjc500\_5 & 500 & 62126 & 22.74 & 13 & 11 & 2.31 & 6 & 13 & 1178 & 890/11  & 9.96 & 1.11 & 179.96 \\ 
        {gen200\_p0.9\_44} & 200 & 1990 & 44.01 & 44 & 38 & 22.66 & 38 &  & 1214 & 5.10 & 0.98 & 0.02 & 361.19 \\ 
        \textbf{gen200\_p0.9\_55} & 200 & 1990 & 55 & 55 & 55 & 54.7 & 55 &  & 345 & 1.81 & 0.24 & 0.02 & 172.09 \\ 
        gen400\_p0.9\_55 & 400 & 7980 & 55 & 55 & 45 & 20.88 & 33 &  & 2803 & 72.46 & 1.55 & 0.14 &  \\ 
        gen400\_p0.9\_65 & 400 & 7980 & 65.02 & 65 & 44 & 21.18 & 34 &  & 3746 & 28.22 & 3.01 & 0.16 &  \\ 
        \textbf{gen400\_p0.9\_75} & 400 & 7980 & 75 & 75 & 75 & 74.54 & 75 &  & 617 & 44.17 & 1.66 & 0.14 &  \\ 
        \textbf{hamming10-2} & 1024 & 5120 & 512 & 512 & 512 & 512 & 512 &  & 2105 & 0.1 & 4.07 & 0.34 & 8.62 \\ 
        hamming10-4 & 1024 & 89600 & 51.2 & $\geq 40$ & 35 & 3.56 & 18 &  & 2958 & 104.78 & 18.7 & 3.75 &  \\ 
        \textbf{hamming6-2} & 64 & 192 & 32 & 32 & 32 & 32 & 32 & 32 & 101 & 0.11 & 0.01 & $<$0.01 & $<$0.01 \\ 
        \textbf{hamming6-4} & 64 & 1312 & 5.33 & 4 & 4 & 1.53 & 3 & 4 & 95 & 0.07 & 0.01 & $<$0.01 & $<$0.01 \\ 
        \textbf{hamming8-2} & 256 & 1024 & 128 & 128 & 128 & 128 & 128 & 128 & 491 & 0.05 & 0.21 & 0.02 & 0.05 \\ 
        \textbf{hamming8-4} & 256 & 11776 & 16 & 16 & 16 & 2.25 & 9 & 16 & 321 & 34.76 & 0.46 & 0.12 &  6.83\\ 
        \textbf{johnson16-2-4} & 120 & 1680 & 8 & 8 & 8 & 2.75 & 6 & 8 & 173 & 0.33 & 0.01 & 0.01 & 3.33 \\ 
        \textbf{johnson32-2-4} & 496 & 14880 & 16 & 16 & 16 & 3.87 & 12 &  & 621 & 45.22 & 0.06 & 0.32 &  \\ 
        \textbf{johnson8-2-4} & 28 & 168 & 4 & 4 & 4 & 2.11 & 4 & 4 & 45 & 0.02 & 2.22 & 0.07 & 0.01 \\ 
        \textbf{johnson8-4-4} & 70 & 560 & 14 & 14 & 14 & 5.74 & 14 & 14 & 125 & 0.06 & 0.09 & $<$0.01 & $<$0.01 \\ 
        \textbf{keller4} & 171 & 5100 & 14.01 & 11 & 11 & 2.95 & 7 & 11 & 443 & 2.77 & 1.2 & 0.04 & 0.82 \\ 
        keller5 & 776 & 74710 & 31 & 27 & 21 & 3.82 & 12 &  & 981 & 79.41 & 4.34 & 2.26 &  \\ 
        keller6 & 3361 & 1026582 & 66.89 & 59 & 47 & 9.49 & 39 &  & 23189 & 9662.47 & 1275.15 & 147.79 &  \\ 
        p-hat1000-1 & 1000 & 377247 & 17.61 & 10 & 9 & 1.68 & 6 & 10 & 1652 & 241.39 & 26.67 & 14.36 & 54.62 \\ 
        p-hat1000-2 & 1000 & 254701 & 55.61 &  & 44 & 19.14 & 38 &  & 2511 & 10622.85 & 49.4 & 10.32 &  \\ 
        p-hat1000-3 & 1000 & 127754 & 84.8 &  & 63 & 24.8 & 50 &  & 4533 & 6577.31 & 27.14 & 5.03 &  \\ 
        p-hat1500-1 & 1500 & 839327 & 22.01 & 12 & 10 & 1.61 & 5 & 12 & 2533 & 663.11 & 95.44 & 50.08 & 1122.13\\ 
        p-hat1500-2 & 1500 & 555290 & 77.56 & 65 & 63 & 26.04 & 53 & &3758 &66738.36 & 171.87 & 35.82 \\ 
        p-hat1500-3 & 1500 & 277006 & 115.43 & 94 & 91 & 34.58 & 74 & &7027&68850.02 & 623.32 & 32.34 \\ 
        p-hat300-3&300&11460&41.17&36&34&16.45&	29&&1101&38.51&	7.22&0.73&167.74\\
        {p-hat500-1}
        &500&93181&13.07&9&8&1.75&	6&	&839&32.59&30.22&1.79&0.54\\
        p-hat500-2
        &500&61804&38.97&36&34&	17.72&	34&	&1149&1492.97&53.84&1.33&208.03\\
        p-hat500-3 & 500 & 30950 & 58.57 &  50 & 48 & 22.28&	42&	&1966&	1100.74&35.55&	0.66\\
        p-hat700-1 & 700 & 183651 & 15.12 & 11 & 8&1.67&7&	&1179&95.68&67.09&	4.94&8.09\\
        {p-hat700-2}&700&122922&49.02&44&43&21.45&39&&1657&4091.29&155.49&4.69\\
         \textbf{p-hat700-3} & 700 & 61640 & 72.7 & 62& 62 & 26.98 & 54 & &2724 & 4256.37 & 2.84 & 2.10 \\
        san1000 & 1000 & 249000 & 15 & 15 & 10 &1.67&8&	&1066&75.06&	22.29&16.00&\\
        \textbf{san200-0-7-1} & 200 & 5970 & 30 & 30 & 30 & 30 & 30 &
        &246&5.12&0.95&0.06&674.78\\
        san200-0-7-2 & 200 & 5970 & 18 &  & 14 & 17.95 & 18 &  & 517 & 7.39 & 2.14 & 0.05 &  \\ 
         \textbf{san200-0-9-1} & 200 & 1990 & 70 & 70 & 70 &
        70&70&&480&0.78&1.34&0.05&3.82\\
        \textbf{san200-0-9-2} & 200 & 1990 & 60 & 60 & 60 & 60&60&&354&0.68&0.96&0.04&130.23\\
        san200-0-9-3 & 200 & 1990 & 44 &  & 37 & 43.88 & 44 &  & 786 & 1.07 & 2.04 & 0.02 & 1351.19 \\ 
        \textbf{san400-0-5-1} & 400 & 39900 & 13 & 13 & 13 & 
        12.92&13&&437&24.38&	20.09&0.63\\
        \textbf{san400-0-7-1} & 400 & 23940 & 40 & 40 & 40 & 40 & 40 & 
        &509&219.62&	9.33&	0.54\\
        \textbf{san400-0-7-2} & 400 & 23940 & 30 & 30 & 30 & 30 & 30 & 
        &518&245.79&	18.95&	0.48\\
        san400-0-7-3 & 400 & 23940 & 22 &  & 16 & 4.15 & 9 &  & 490 & 275.76 & 0.14 & 0.36 &  \\ 
        \textbf{san400-0-9-1} & 400 & 7980 & 100 & 100 &
        100&100&100&&572&12.80&9.56&	0.22\\
        sanr200-0-7 & 200 & 6032 & 23.84 & 18 & 16 & 5.76&12&	&709&	5.76&1.99&	0.10 &4.89\\
        sanr200-0-9 & 200 & 2037 & 49.3 & 42 &41 &21.92& 33 & &915 & 0.75 & 0.36& 0.02&966.90 \\
        sanr400-0-5 & 400 & 39816 & 20.32 & 13 & 12 &
        2.49&7&&	903&31.23&	11.08&	0.63&12.54\\
        sanr400-0-7 & 400 & 23931 & 34.28 & 21 & 19 &
        5.27&10&&1505&247.17&	20.65&	0.66
    \end{tabular}}
\caption{ 
Comparison of Algorithm~\ref{alg:VFA2} with Benson-Ye (BY) and Walteros-Buchanan (WB) on graphs from the DIMACS2 dataset \cite{ams1996dimacs}. 
All instances are complemented first; $|E|$ is the number of edges after taking the complement.
Algorithm~\ref{alg:VFA2} attains optimality for the bolded instances. 
We run the randomized BY method $ \lvert N \rvert $ times and report the average and best value obtained.}
\label{table:comp_dimacs2}
\end{table}
}

\begin{table}[htb]
    \centering
    \normalsize
    \resizebox{0.8\columnwidth}{!}{\begin{tabular}{cccc|cccc|cccc}
        \toprule
        \multicolumn{4}{c|}{Graph}& \multicolumn{4}{c|}{Solution value} & \multicolumn{4}{c}{Time (s)} \\
        Name & $|N|$ & $|E|$ & $\vartheta$ & Alg.\ref{alg:VFA2} & BY${}_{\text{avg}}$ & BY${}_{\text{best}}$ & WB & SDP & Alg.\ref{alg:VFA2} & BY${}_{\text{total}}$ &WB${}$\\
        \midrule
        G1 & 800 & 19176 & 145.03 & 86 & 34.61 & 54 &  & 144.62 & 5.26 & 0.95 \\ 
        G2 & 800 & 19176 & 145.35 & 86 & 34.38 & 54 &  & 144.35 & 2.65 & 1.26 \\
        G3 & 800 & 19176 & 145.3 & 86 & 34.25 & 57 &  & 136.46 & 2.83 & 1.01 \\
        G4 & 800 & 19176 & 145.37 & 85 & 34.87 & 54 &  & 142.12 & 2.29 & 0.86 \\
        G5 & 800 & 19176 & 145.36 & 85 & 34.91 & 53 &  & 139.21 & 2.52 & 0.73 \\
        G6 & 800 & 19176 & 145.03 & 86 & 34.1 & 55 &  & 147.05 & 2.49 & 0.76 \\
        G7 & 800 & 19176 & 145.35 & 86 & 34.43 & 54 &  & 148.82 & 3.1 & 1.27 \\ 
        G8 & 800 & 19176 & 145.3 & 86 & 34.21 & 56 &  & 144.48 & 3.21 & 0.75 \\ 
        G9 & 800 & 19176 & 145.37 & 85 & 34.58 & 54 &  & 142.7 & 2.42 & 0.8 \\ 
        G10 & 800 & 19176 & 145.36 & 85 & 34.83 & 53 &  & 137.69 & 2.39 & 0.75 \\ 
        \textbf{G11} & 800 & 1600 & 400.02 & 400 & 400 & 400 & 400 & 2.21 & 1.55 & 0.14 &0.27\\ 
        \textbf{G12} & 800 & 1600 & 400.01 & 400 & 400 & 400 & 400 & 0.13 & 0.11 & 0.09 &0.28\\ 
        G13 & 800 & 1600 & 398.42 & 384 & 381.42 & 384 & & 1.37 & 0.17 & 0.09 \\ 
        \textbf{G14} & 800 & 4694 & 279 & 279 & 275.45 & 279 &  & 42.65 & 0.34 & 0.17 \\ 
        \textbf{G15} & 800 & 4661 & 283.85 & 283 & 270.32 & 282 &  & 220.98 & 26.45 & 0.18 \\ 
        \textbf{G16} & 800 & 4672 & 285.16 & 285 & 266.85 & 284 &  & 220.98 & 81.01 & 0.24 \\
        \textbf{G17} & 800 & 4667 & 286.23 & 286 & 271.54 & 285 &  & 191.85 & 24.56 & 0.18 \\ 
        \textbf{G18} & 800 & 4694 & 279 & 279 & 275.45 & 279 &  & 41.01 & 0.27 & 0.17 \\ 
        \textbf{G19} & 800 & 4661 & 283.85 & 283 & 269.63 & 283 &  & 226.4 & 27.16 & 0.19 \\ 
        \textbf{G20} & 800 & 4672 & 285.16 & 285 & 267.93 & 285 &  & 224.62 & 78.44 & 0.19 \\
        \textbf{G21} & 800 & 4667 & 286.23 & 286 & 269.91 & 286 &  & 191.94 & 24.94 & 0.21 \\ 
        G22 & 2000 & 19990 & 578.51 & 410 & 238.82 & 305 &  & 377.25 & 36.83 & 3.74 \\ 
        G23 & 2000 & 19990 & 577.93 & 415 & 237.18 & 299 &  & 378.67 & 33.63 & 4.21 \\ 
        G24 & 2000 & 19990 & 580.01 & 411 & 239.74 & 308 &  & 378.68 & 35.27 & 3.78 \\ 
        G25 & 2000 & 19990 & 578.09 & 406 & 237.85 & 301 &  & 379.84 & 33.88 & 4.08 \\ 
        G26 & 2000 & 19990 & 577.99 & 411 & 238.09 & 299 &  & 368.94 & 35.1 & 3.86 \\ 
        G27 & 2000 & 19990 & 578.51 & 410 & 238.37 & 309 &  & 366 & 33.38 & 6.37 \\ 
        G28 & 2000 & 19990 & 577.93 & 415 & 238.72 & 300 &  & 367.88 & 34.05 & 3.95 \\ 
        G29 & 2000 & 19990 & 580.01 & 411 & 239.24 & 303 &  & 369.62 & 37.82 & 3.77 \\ 
        G30 & 2000 & 19990 & 578.09 & 406 & 237.12 & 301 &  & 370.36 & 38.12 & 4.42 \\ 
        G31 & 2000 & 19990 & 577.99 & 411 & 237.74 & 301 &  & 371.64 & 34.09 & 4.02 \\ 
        \textbf{G32} & 2000 & 4000 & 1000 & 1000 & 1000 & 1000 & 1000 & 0.35 & 0.64 & 0.53 &20.95\\ 
        G33 & 2000 & 4000 & 996.04 & 960 & 945.7 & 960 &  & 191.85 & 304.3 & 0.55 \\
        \textbf{G34} & 2000 & 4000 & 1000 & 1000 & 1000 & 1000 & 1000 & 0.27 & 0.56 & 0.52 &20.44\\ 
        \textbf{G35} & 2000 & 11778 & 718.27 & 718 & 693.53 & 716 &  & 378.97 & 0.89 & 1.13 \\ 
        G36 & 2000 & 11766 & 696.03 & 695 & 671.48 & 695 &  & 586.57 & 101.99 & 1.11 \\ 
        \textbf{G37} & 2000 & 11785 & 708.02 & 708 & 687.27 & 707 &  & 306.7 & 0.95 & 1.16 \\ 
        \textbf{G38} & 2000 & 11779 & 716.02 & 716 & 694.83 & 715 &  & 298.5 & 0.91 & 1.11 \\ 
        \textbf{G39} & 2000 & 11778 & 718.27 & 718 & 693.53 & 716 &  & 372.92 & 1.17 & 1.12 \\  
        G40 & 2000 & 11766 & 696.03 & 695 & 670.97 & 695 &  & 571.55 & 103.25 & 1.1 \\  
        \textbf{G41} & 2000 & 11785 & 708.02 & 708 & 687.27 & 707 &  & 299.61 & 1.04 & 1.13 \\ 
        \textbf{G42} & 2000 & 11779 & 716.02 & 716 & 694.83 & 715 &  & 306.85 & 1.22 & 1.11 \\ 
        G43 & 1000 & 9990 & 280.62 & 199 & 114.89 & 151 &  & 42.47 & 6.6 & 0.72 \\ 
        G44 & 1000 & 9990 & 280.58 & 206 & 117.66 & 152 &  & 42.35 & 4.72 & 0.68 \\ 
        G45 & 1000 & 9990 & 280.19 & 198 & 115.24 & 149 &  & 41.16 & 5.03 & 1.03 \\ 
        G46 & 1000 & 9990 & 279.84 & 199 & 115.75 & 155 &  & 40.47 & 7.43 & 0.66 \\ 
        G47 & 1000 & 9990 & 281.89 & 203 & 117.81 & 152 &  & 40.65 & 4.91 & 0.66 \\ 
        \textbf{G48} & 3000 & 6000 & 1500 & 1500 & 1500 & 1500 & 1500 & 0.48 & 142.93 & 1.18 &23.40 \\ 
        \textbf{G49} & 3000 & 6000 & 1500 & 1500 & 1500 & 1500 & 1500 & 0.47 & 0.86 & 1.17 &23.99\\ 
        G50 & 3000 & 6000 & 1494.06 & 1440 & 1418.94 & 1440 &  & 226.4 & 1104.64 & 1.27 \\ 
        \textbf{G51} & 1000 & 5909 & 349.01 & 349 & 333.54 & 348 & & 184.94 & 0.38 & 0.29 \\ 
        \textbf{G52} & 1000 & 5916 & 348.43 & 348 & 332.37 & 348 & & 220.98 & 2.16 & 0.36 \\ 
        G53 & 1000 & 5914 & 348.39 & 346 & 324.99 & 346 & & 255.35 & 0.4 & 0.29 \\ 
        \textbf{G54} & 1000 & 5916 & 341.01 & 341 & 330.08 & 341 & & 120.39 & 23.68 & 0.29 \\ 
        G55 & 5000 & 12498 & 2324.17 & 2172 & 1877.34 & 2060 &  & 120.82 & 19.77 & 3.73 \\ 
        G56 & 5000 & 12498 & 2324.17 & 2172 & 1877.34 & 2060 &  & 124.67 & 18.9 & 3.7 \\ 
        \textbf{G57} & 5000 & 10000 & 2500 & 2500 & 2500 & 2500 &  & 0.96 & 1.87 & 3.53 \\ 
        \textbf{G58} & 5000 & 29570 & 1782.62 & 1782 & 1714.53 & 1778 & & 1291.86 & 8.22 & 7.62 \\ 
        \textbf{G59} & 5000 & 29570 & 1782.62 & 1782 & 1714.53 & 1778 & & 1286.69 & 7 & 6.9 \\ 
        G60 & 7000 & 17148 & 3265.04 & 3056 & 2643.67 & 2883 &  & 251.85 & 47 & 7.53 \\ 
        G61 & 7000 & 17148 & 3265.04 & 3056 & 2643.67 & 2883 &  & 246.22 & 47.75 & 8.32 \\ 
        \textbf{G62} & 7000 & 14000 & 3500 & 3500 & 3500 & 3500 & & 2.89 & 3.78 & 6.93 \\ 
        G63 & 7000 & 41459 & 2493.42 & 2486 & 2381.93 & 2484 &  & 1831.07 & 21.6 & 13.96 \\ 
        G64 & 7000 & 41459 & 2493.42 & 2486 & 2381.93 & 2484 & & 6426.76 & 21.35 & 14.05 \\ 
        \textbf{G65} & 8000 & 16000 & 4000 & 4000 & 4000 & 4000 & & 2.97 & 199.94 & 8.53 \\ 
        \textbf{G66} & 9000 & 18000 & 4500 & 4500 & 4500 & 4500 & & 4.46 & 239.33 & 11.19 \\ 
        \textbf{G67} & 10000 & 20000 & 5000 & 5000 & 5000 & 5000 & & 3.32 & 291.93 & 13.94 \\  
        \textbf{G70} & 10000 & 9999 & 6077.24 & 6077 & 6077 & 6077 & & 0.05 & $<$0.01 & $<$0.01 \\ 
        \textbf{G72} & 10000 & 20000 & 5000 & 5000 & 5000 & 5000 &  & 4.61 & 296.21 & 14.32 \\  
        \textbf{G77} & 14000 & 28000 & 7000 & 7000 & 6999.5 & 7000 & & 5.34 & 21.42 & 28.55 \\ 
        \textbf{G81} & 20000 & 40000 & 10000 & 10000 & 10000 & 10000 &  & 11.28 & 1046.98 & 56.23 \\
        \bottomrule
    \end{tabular}}
    \caption{Same experiments as in Table~\ref{table:comp_dimacs2}, using graphs in the GSet dataset~\cite{gset}. We do not include $ \alpha $, as this number is not included for this dataset.}
    \label{table:comp_gset}
\end{table}

\begin{table}[htb]
    \centering
    \normalsize
    \resizebox{\columnwidth}{!}{\begin{tabular}{ccccc|cccc|Hcccc}
        \toprule
        \multicolumn{5}{c|}{Graph}& \multicolumn{4}{c|}{Solution value} & \#VFA & \multicolumn{4}{c}{Time (s)} \\
         Name & $|N|$ & $|E|$ & $\vartheta$& $\alpha$ & Alg.\ref{alg:VFA2} & BY${}_{\text{avg}}$ & BY${}_{\text{best}}$ & WB &calls & SDP & Alg.\ref{alg:VFA2} & BY${}_{\text{total}}$&WB${}$ \\
        \midrule
        \multicolumn{13}{c}{DIMACS10}\\
        \midrule
        uk & 4824 & 6837 & 2286.98 & ~ & 2173 & 2066.26 & 2115 && 216643 & 53.32 & 18.33 & 2.68 \\ 
        data & 2851 & 15093 & 690.01 & ~ & 683 & 621.34 & 674 & &23602 & 118 & 5.07 & 2.37 \\ 
        fe\_4elt2 & 11143 & 32818 & 3631.78 & ~ & 3544 & 2780.99 & 3239 & &291529 & 371.03 & 93.07 & 24 \\ 
        \textbf{vsp\_p0291\_seymourl\_iiasa} & 10498 & 53868 & 6301.12 & 6301 & 6301 & 6232.51 & 6301 & &16279 & 753.19 & 6.57 & 20.62 \\ 
        cti & 16840 & 48232 & 8198.1 & ~ &8083&	7798.45&8080&&755812	&3420.47&152.84&68.99\\
        \textbf{fe\_sphere} & 16386 & 49152 & 5462.00&5462&5462& 4279.16&	5462&&110344	&49.70&	2575.14&	51.34\\
        cs4 & 22499 & 43858 & 9738.86 & ~ & 8987 & 8099.92 & 8273 & &6573556 & 534.88 & 2221.8 & 84.03 \\ 
        hi2010 & 25016 & 62063 & 11022.07 & ~ & 11012 & 10743.45 & 10926 & &758792 & 1464.83 & 81.59 & 21.19 \\ 
        ri2010 & 25181 & 62875 & 10819.96 & ~ & 10792 & 10460.01 & 10653 & &3157157 & 1281.81 & 1307.78 & 82.03 \\ 
        vt2010 & 32580 & 77799 & 15127.13 & ~ & 15118 & 14826.15 & 14980 & &2733536 & 1564.35 & 1287.42 & 103.47 \\ 
        nh2010 & 48837 & 117275 & 22890.74 & ~ & 22878 & 22452.91 & 22687 & &4898868 & 2467.79 & 3164.02 & 170.91 \\ 
        delaunay\_n14 & 16384 & 49122 & 5230.5 & ~ & 5132 & 4141.92 & 4503 & &609191 & 548.09 & 229.26 & 54.14 \\ 
        \midrule
        \multicolumn{12}{c}{SNAP}\\
        \midrule
        \textbf{ca-CondMat} & 23133 & 93497 & 9612.17 & 9612 & 9612 & 9477.4 & 9585 & &37818 & 8328.38 & 29.85 & 47.59& \\ 
        \textbf{p2p-Gnutella31} & 62586 & 76950 & 47974 & 47974 & 47974 & 47974 & 47974 &47974 &23 & 0.13 & $<$0.01 & $<$0.01 &$<$0.01\\ 
        \textbf{Oregon-2} & 11806 & 32730 & 9889 & 9889 & 9889 & 9888.94 & 9889 & 9889&137 & 2.53 & 5.01 & 0.39 &$<$0.01\\ 
        \textbf{p2p-Gnutella25} & 22687 & 30751 & 17116 & 17116 & 17116 & 17116 & 17116 & 17116&23 & 0.02 & $<$0.01 & $<$0.01 &$<$0.01\\ 
        \textbf{p2p-Gnutella24} & 26518 & 35828 & 19872 & 19872 & 19872 & 19872 & 19872 & 19872&8 & 0.01 & $<$0.01 & $<$0.01 &$<$0.01\\ 
        \textbf{as-caida\_G\_001} & 31379 & 32955 & 29037 & 29037 & 29037 & 29037 & 29037 & 29037&19 & 1.13 & $<$0.01 & $<$0.01 &$<$0.01\\ 
        \textbf{p2p-Gnutella30} & 36682 & 48507 & 28094 & 28094 & 28094 & 28094 & 28094 & 28094&32 & 0.1 & $<$0.01 & $<$0.01 &$<$0.01\\ 
        \bottomrule
    \end{tabular}}
    \caption{Same experiments as in Table~\ref{table:comp_dimacs2}, using graphs from the DIMACS10 and SNAP datasets~\cite{dimacs10,snapnets}}
    \label{table:comp_dimacs10}
\end{table}
\begin{table}
    \normalsize
    \centering
    \resizebox{0.7\columnwidth}{!}{\begin{tabular}{cccc|ccc|Hcccc}
        \toprule
        
        \multicolumn{4}{c|}{Graph}& \multicolumn{3}{c|}{Solution value} & \#VFA & \multicolumn{3}{c}{Time (s)} \\
Name & $|N|$ & $|E|$ & $\vartheta$ & Alg.\ref{alg:VFA2} & BY${}_{\text{avg}}$ & BY${}_{\text{best}}$ & calls & SDP & Alg.\ref{alg:VFA2} & BY${}_{\text{total}}$\\
\midrule
\textbf{MANN-a27} & 378 & 702 & 938 & 938 & 937.08 & 938 & 1621 & 7.62 & 8.79 & 0.49 \\
\textbf{MANN-a45} & 1035 & 1980 & 2596 & 2596 & 814.16 & 1791 & & 21.00 & 18.82 & 0.28 \\
\textbf{MANN-a81} & 3321 & 6480 & 8631 & 8631 & 2704.07 & 5939 & 54300 & 43.75 & 109.05 & 4.31 \\
\textbf{MANN-a9} & 45 & 72 & 109 & 109 & 109 & 109 & 82 & 0.02 & 0.01 & $<0.01$ \\
brock200-1 & 200 & 5066 & 173.47 & 144 & 55.23 & 106 & 635 & 5.16 & 1.49 & 0.06 \\
brock200-2 & 200 & 10024 & 89.68 & 71 & 21.55 & 56 & 412 & 19.3 & 1.04 & 0.1 \\
brock200-3 & 200 & 7852 & 118.69 & 93 & 32.55 & 76 & 477 & 11.85 & 1.14 & 0.07 \\
brock200-4 & 200 & 6811 & 128.47 & 103 & 36.69 & 95 & 547 & 8.83 & 1.36 & 0.06 \\
brock400-1 & 400 & 20077 & 246.12 & 161 & 51.3 & 102 & 1526 & 151.81 & 5.06 & 0.34 \\
brock400-2 & 400 & 20014 & 245.5 & 166 & 49 & 99 & 1548 & 146.95 & 5.4 & 0.49 \\
brock400-3 & 400 & 20119 & 250.03 & 169 & 53.28 & 111 & 1462 & 130.22 & 5.86 & 0.34 \\
brock400-4 & 400 & 20035 & 243.72 & 152 & 49.13 & 106 & 1572 & 150.03 & 9.02 & 0.34 \\
brock800-1 & 800 & 112095 & 259.25 & 122 & 27.79 & 83 & 2480 & 83.08 & 25.17 & 5.38 \\
brock800-2 & 800 & 111434 & 260.98 & 112 & 28.85 & 86 & 2587 & 83.92 & 28.51 & 3.88 \\
brock800-3 & 800 & 112267 & 266.33 & 143 & 30.09 & 75 & 2435 & 86.77 & 25.87 & 4.31 \\
brock800-4 & 800 & 111957 & 265.39 & 123 & 29.14 & 76 & 2545 & 83.73 & 25.94 & 4.06 \\
\textbf{c-fat200-1} & 200 & 18366 & 94 & 94 & 94 & 94 & 120 & 2 & 9.6 & 0.76 \\
\textbf{c-fat200-2} & 200 & 16665 & 147 & 147 & 146.27 & 147 & 210 & 2.44 & 14.83 & 0.28 \\
\textbf{c-fat200-5} & 200 & 11427 & 374 & 374 & 374 & 374 & 210 & 14.84 & 0.2 & 0.13 \\
\textbf{c-fat500-10} & 500 & 78123 & 732 & 732 & 727.59 & 732 & 510 & 55.87 & 15.27 & 2.67 \\
\textbf{c-fat500-1} & 500 & 120291 & 97 & 97 & 96.81 & 97 & 422 & 5.2 & 5.71 & 2.39 \\
\textbf{c-fat500-2} & 500 & 115611 & 165 & 165 & 165 & 165 & 510 & 29.64 & 199.23 & 2.71 \\
\textbf{c-fat500-5} & 500 & 101559 & 410 & 410 & 409.55 & 410 & 510 & 230.39 & 1.43 & 2.03 \\
c1000.9 & 1000 & 49421 & 759.58 & 464 & 140.96 & 272 & 7452 & 1697.06 & 1400.01 & 9.85 \\
c125.9 & 125 & 787 & 218.98 & 195 & 138.33 & 189 & 499 & 8.02 & 8.61 & 0.55 \\
c2000.5 & 2000 & 999164 & 280.52 & 99 & 15.35 & 66 & 4342 & 33155.19 & 879.2 & 227.15 \\
c2000.9 & 2000 & 199468 & 1105.27 & 344 & 48.83 & 242 & 1946 & 72603.8 & 3033.9 & 22.22 \\
c250.9 & 250 & 3141 & 341.49 & 271 & 146.54 & 233 & 1364 & 1.67 & 2.7 & 0.07 \\
c4000.5 & 4000 & 3997732 & 395.4 & 106 & 13.75 & 57 & 8535 & 462610.07 & 3140.67 & 1181.26 \\
c500.9 & 500 & 12418 & 523.52 & 365 & 143.28 & 270 & 3287 & 50.33 & 60.73 & 0.46 \\
dsjc1000\_5 & 1000 & 249674 & 200.04 & 92 & 18.05 & 59 & 2165 & 109.71 & 813.45 & 25.78 \\
dsjc500\_5 & 500 & 62126 & 146.51 & 96 & 18.95 & 57 & 1069 & 28.27 & 173.71 & 1.6 \\
gen200\_p0.9\_44 & 200 & 1990 & 291.78 & 246 & 142.57 & 229 & 888 & 1.28 & 18.26 & 0.06 \\
\textbf{gen200\_p0.9\_55} & 200 & 1990 & 338 & 338 & 338 & 338 & 399 & 1.17 & 11.09 & 0.17 \\
gen400\_p0.9\_55 & 400 & 7980 & 423.61 & 324 & 177.53 & 289 & 2064 & 22.25 & 333.86 & 0.32 \\
gen400\_p0.9\_65 & 400 & 7980 & 446.12 & 352 & 168.08 & 282 & 2286 & 21.6 & 218.57 & 0.38 \\
gen400\_p0.9\_75 & 400 & 7980 & 455.6 & 320 & 187.45 & 409 & 2426 & 21.11 & 230.69 & 0.34 \\
\textbf{hamming10-2} & 1024 & 5120 & 2800 & 2800 & 2792.28 & 2800 & 2063 & 11.42 & 23.95 & 2.33 \\
hamming10-4 & 1024 & 89600 & 364.71 & 285 & 122.63 & 238 & 2973 & 9918.86 & 45700.81 & 8.57 \\
\textbf{hamming6-2} & 64 & 192 & 173 & 173 & 173 & 173 & 128 & 8.59 & 3.94 & 0.08 \\
hamming6-4 & 64 & 1312 & 38.1 & 37 & 21.05 & 37 & 87 & 0.48 & 0.1 & $<0.01$ \\
\textbf{hamming8-2} & 256 & 1024 & 737 & 737 & 723.47 & 737 & 518 & 0.21 & 0.32 & 0.02 \\
\textbf{hamming8-4} & 256 & 11776 & 123 & 123 & 123 & 123 & 320 & 32.29 & 4.93 & 0.16 \\
\textbf{johnson16-2-4} & 120 & 1680 & 77 & 77 & 77 & 77 & 173 & 0.5 & 0.35 & 0.02 \\
\textbf{johnson32-2-4} & 496 & 14880 & 157 & 157 & 148.84 & 157 & 621 & 91.22 & 148.94 & 0.9 \\
\textbf{johnson8-2-4} & 28 & 168 & 38 & 38 & 38 & 38 & 45 & 0.04 & 0.02 & $<0.01$ \\
\textbf{johnson8-4-4} & 70 & 560 & 105 & 105 & 105 & 105 & 131 & 0.14 & 0.03 & $<0.01$ \\
keller4 & 171 & 5100 & 95.73 & 79 & 31.5 & 76 & 374 & 3.47 & 1.55 & 0.04 \\
keller5 & 776 & 74710 & 234.45 & 176 & 44.28 & 125 & 2038 & 1134.35 & 125.81 & 3.14 \\
keller6 & 3361 & 1026582 & 538.62 & 277 & 64.92 & 203 & 2500 & 212884.82 & 11605.11 & 259.51 \\
p-hat1000-1 & 1000 & 377247 & 109.19 & 64 & 13.85 & 47 & 1596 & 1641.78 & 66.42 & 15.51 \\
p-hat1000-2 & 1000 & 254701 & 335.69 & 277 & 121.58 & 245 & 2373 & 2272.97 & 254.03 & 12.79 \\
p-hat1000-3 & 1000 & 127754 & 519.58 & 406 & 151.7 & 321 & 4379 & 3162.75 & 77.13 & 6.62 \\
p-hat1500-1 & 1500 & 839327 & 136.33 & 56 & 12.92 & 52 & 1734 & 12036.66 & 216.02 & 59.58 \\
p-hat1500-2 & 1500 & 555290 & 460.14 & 292 & 146.28 & 297 & 862 & 10208.36 & 284.65 & 104.05 \\
p-hat1500-3 & 1500 & 277006 & 697.66 & 332 & 203.96 & 467 & 1336 & 36496.32 & 3961.68 & 26.99 \\
p-hat300-3 & 300 & 11460 & 245.55 & 221 & 97.72 & 187 & 825 & 44.24 & 12.05 & 0.68 \\
p-hat500-1 & 500 & 93181 & 82.56 & 58 & 15.74 & 58 & 817 & 202.27 & 9.87 & 1.89 \\
p-hat500-2 & 500 & 61804 & 239.35 & 220 & 110.89 & 191 & 969 & 554.96 & 19.46 & 1.22 \\
p-hat500-3 & 500 & 30950 & 353.66 & 317 & 149.25 & 303 & 1584 & 659.87 & 8.71 & 0.62 \\
p-hat700-1 & 700 & 183651 & 94.58 & 59 & 15.44 & 46 & 1127 & 526.38 & 25.86 & 4.62 \\
p-hat700-2 & 700 & 122922 & 300.27 & 250 & 117.79 & 230 & 1654 & 890.9 & 36.62 & 3.89 \\
p-hat700-3 & 700 & 61640 & 436.92 & 367 & 33.74 & 248 & 2622 & 2497.79 & 108.12 & 2.6 \\
san1000 & 1000 & 249000 & 114.51 & 87 & 19.82 & 76 & 1780 & 2698.35 & 125.21 & 10.73 \\
\textbf{san200-0-7-1} & 200 & 5970 & 159 & 159 & 159 & 159 & 282 & 5.39 & 0.58 & 0.07 \\
san200-0-7-2 & 200 & 5970 & 125.31 & 101 & 63.13 & 108 & 595 & 6.6 & 4.49 & 0.05 \\
\textbf{san200-0-9-1} & 200 & 1990 & 398 & 398 & 398 & 398 & 498 & 0.78 & 0.42 & 0.02 \\
san200-0-9-2 & 200 & 1990 & 295.59 & 251 & 163.16 & 268 & 1003 & 0.89 & 2.38 & 0.02 \\
san200-0-9-3 & 200 & 1990 & 286.18 & 246 & 158.95 & 241 & 912 & 0.88 & 2.3 & 0.04 \\
san400-0-5-1 & 400 & 39900 & 83.17 & 72 & 23.88 & 63 & 685 & 849.72 & 11.59 & 0.61 \\
san400-0-7-1 & 400 & 23940 & 210.23 & 173 & 62.9 & 155 & 1304 & 280.5 & 6.05 & 0.58 \\
san400-0-7-2 & 400 & 23940 & 191.04 & 129 & 47.58 & 120 & 1244 & 281.6 & 6.61 & 0.42 \\
san400-0-7-3 & 400 & 23940 & 171.06 & 128 & 43.33 & 96 & 1129 & 281.84 & 10.56 & 0.4 \\
san400-0-9-1 & 400 & 7980 & 489.08 & 386 & 282.29 & 489 & 2574 & 17.21 & 8.33 & 0.18 \\
sanr200-0-7 & 200 & 6032 & 149.63 & 121 & 44.62 & 112 & 580 & 5.78 & 1 & 0.05 \\
sanr200-0-9 & 200 & 2037 & 298.27 & 262 & 136.82 & 217 & 914 & 0.97 & 1.74 & 0.02 \\
sanr400-0-5 & 400 & 39816 & 129.77 & 75 & 20.16 & 57 & 881 & 15.4 & 3.49 & 0.61 \\
sanr400-0-7 & 400 & 23931 & 211.66 & 134 & 43.06 & 97 & 1378 & 240.72 & 5.37 & 0.45 \\
    \end{tabular}}
\caption{
Comparison of Algorithm~\ref{alg:VFA2} with Benson-Ye (BY) on the DIMACS benchmark instances in the order of Table~\ref{table:comp_dimacs2}, with weights independently sampled from $\operatorname{Uniform}([10])$. All instances are complemented first; $|E|$ is the number of edges after taking the complement. Algorithm~\ref{alg:VFA2} reaches optimality or the best known bound for bolded instances. We run the randomized BY method $|N|$ times and report the average and best value obtained.}
\label{table:weighted_comp_dimacs_full}
\end{table}

\section*{Declarations}
The authors’ work was partially funded by the U.S. Office of Naval Research, N00014-23-1-2631.

The authors have no other competing interests to declare that are relevant to the content of this article.
\newpage
{
\small
\bibliographystyle{splncs04}
\bibliography{SDPVFA}

@article{10.1145/227683.227684,
author = {Goemans, Michel X. and Williamson, David P.},
title = {Improved Approximation Algorithms for Maximum Cut and Satisfiability Problems Using Semidefinite Programming},
year = {1995},
month = nov,
publisher = {Association for Computing Machinery},
address = {New York, NY, USA},
volume = {42},
number = {6},
journal = {Journal of the ACM},
pages = {1115–1145},
numpages = {31}
}

@article{ellipsoid_MSS,
	author = {Gr{\"o}tschel, M. and Lov{\'a}sz, L. and Schrijver, A.},
	date = {1981/06/01},
	id = {Gr{\"o}tschel1981},
	journal = {Combinatorica},
	number = {2},
	pages = {169--197},
	title = {The ellipsoid method and its consequences in combinatorial optimization},
	volume = {1},
	year = {1981},
}

@incollection{GROTSCHEL1984325,
title = {Polynomial Algorithms for Perfect Graphs},
series = {North-Holland Mathematics Studies},
publisher = {North-Holland},
volume = {88},
pages = {325-356},
year = {1984},
booktitle = {Topics on Perfect Graphs},
author = {M. Grötschel and L. Lovász and A. Schrijver}
}

@book{Grötschel1993,
author="Gr{\"o}tschel, Martin
and Lov{\'a}sz, L{\'a}szl{\'o}
and Schrijver, Alexander",
title="Geometric Algorithms and Combinatorial Optimization",
bookTitle="Geometric Algorithms and Combinatorial Optimization",
year="1993",
publisher="Springer Berlin Heidelberg",
address="Berlin, Heidelberg",
}

@article{NP-hard,
author = {Garey, M. R. and Johnson, D. S.},
title = {`` Strong '' NP-Completeness Results: Motivation, Examples, and Implications},
year = {1978},
issue_date = {July 1978},
publisher = {Association for Computing Machinery},
address = {New York, NY, USA},
volume = {25},
number = {3},
journal = {Journal of the ACM},
month = {July},
pages = {499–508},
numpages = {10}
}

@inproceedings{sublinear-time-Alizadeh,
author = {Alizadeh, Farid},
title = {A sublinear-time randomized parallel algorithm for the maximum clique problem in perfect graphs},
year = {1991},
publisher = {Society for Industrial and Applied Mathematics},
address = {USA},
booktitle = {Proceedings of the Second Annual ACM-SIAM Symposium on Discrete Algorithms},
pages = {188–194},
numpages = {7},
location = {San Francisco, California, USA},
series = {SODA '91}
}

@article{Yildirim-Fan,
	author = {Yildirim, E. Alper and Fan-Orzechowski, Xiaofei},
	date = {2006/03/01},
	id = {Yildirim2006},
	journal = {Computational Optimization and Applications},
	number = {2},
	pages = {229--247},
	title = {{On Extracting Maximum Stable Sets in Perfect Graphs Using Lov{\'a}sz's Theta Function}},
	volume = {33},
	year = {2006},
}

@article{Prömel_Steger_1992, 
title={{Almost all Berge Graphs are Perfect}}, 
volume={1}, 
number={1}, 
journal={Combinatorics, Probability and Computing}, 
author={Prömel, Hans Jürgen and Steger, Angelika}, 
year={1992}, 
pages={53–79}}

@article {StrongPerfectGraph,
    AUTHOR = {Chudnovsky, Maria and Robertson, Neil and Seymour, Paul and
              Thomas, Robin},
     TITLE = {The strong perfect graph theorem},
   JOURNAL = {Annals of Mathematics (2)},
  FJOURNAL = {Annals of Mathematics. Second Series},
    VOLUME = {164},
      YEAR = {2006},
    NUMBER = {1},
     PAGES = {51--229},
   MRCLASS = {05C17},
  MRNUMBER = {2233847},
MRREVIEWER = {Carsten\ Thomassen},
}

@article{thetabody,
author = {Lov\'{a}sz, L. and Schrijver, A.},
title = {Cones of Matrices and Set-Functions and 0–1 Optimization},
journal = {SIAM Journal on Optimization},
volume = {1},
number = {2},
pages = {166-190},
year = {1991}
}

@Inbook{Karp1972,
author="Karp, Richard M.",
editor="Miller, Raymond E.
and Thatcher, James W.
and Bohlinger, Jean D.",
title="Reducibility among Combinatorial Problems",
bookTitle="Complexity of Computer Computations: Proceedings of a symposium on the Complexity of Computer Computations, held March 20--22, 1972, at the IBM Thomas J. Watson Research Center, Yorktown Heights, New York, and sponsored by the Office of Naval Research, Mathematics Program, IBM World Trade Corporation, and the IBM Research Mathematical Sciences Department",
year="1972",
publisher="Springer US",
address="Boston, MA",
pages="85--103",
}

@book{nesterov1994interior,
author = {Nesterov, Yurii and Nemirovskii, Arkadii},
title = {Interior-Point Polynomial Algorithms in Convex Programming},
publisher = {Society for Industrial and Applied Mathematics},
year = {1994},
address = {},
edition   = {},
}

@ARTICLE{1055985,
  author={Lovasz, L.},
  journal={IEEE Transactions on Information Theory}, 
  title={On the Shannon capacity of a graph}, 
  year={1979},
  volume={25},
  number={1},
  pages={1-7},
}

@article{doi:10.1137/20M1383732,
author = {Abrishami, Tara and Chudnovsky, Maria and Pilipczuk, Marcin and Rz{a}\.{z}ewski, Pawe\l{} and Seymour, Paul},
title = {Induced Subgraphs of Bounded Treewidth and the Container Method},
journal = {SIAM Journal on Computing},
volume = {53},
number = {3},
pages = {624-647},
year = {2024}
}

@article{10.1145/2629600,
author = {Faenza, Yuri and Oriolo, Gianpaolo and Stauffer, Gautier},
title = {Solving the Weighted Stable Set Problem in Claw-Free Graphs via Decomposition},
year = {2014},
issue_date = {July 2014},
publisher = {Association for Computing Machinery},
address = {New York, NY, USA},
volume = {61},
number = {4},
journal = {Journal of the ACM},
month = {July},
articleno = {20},
numpages = {41}
}

@inproceedings{10.5555/2095116.2095218,
author = {Yuri Faenza and Gianpaolo Oriolo and Gautier Stauffer},
title = {Separating stable sets in claw-free graphs via Padberg-Rao and compact linear programs},
year = {2012},
booktitle = {Proceedings of the 2012  Annual ACM-SIAM Symposium on Discrete Algorithms (SODA)},
chapter = {},
pages = {1298--1308}
}

@article{2odd-cycle-free,
	author = {Conforti, Michele and Fiorini, Samuel and Huynh, Tony and Weltge, Stefan},
	date = {2022/03/01},
	journal = {Mathematical Programming},
	number = {1},
	pages = {547--566}, 
	title = {Extended formulations for stable set polytopes of graphs without two disjoint odd cycles},
	volume = {192},
	year = {2022},
}

@article{walteros_why_2020,
	title = {Why {Is} {Maximum} {Clique} {Often} {Easy} in {Practice}?},
	volume = {68},
	language = {en},
	number = {6},
	journal = {Operations Research},
	author = {Walteros, Jose L. and Buchanan, Austin},
	month = {November},
	year = {2020},
	pages = {1866--1895},
}

@article{Gavril1973AlgorithmsFA,
author = {Gavril, F.},
title = {Algorithms for a maximum clique and a maximum independent set of a circle graph},
journal = {Networks},
volume = {3},
number = {3},
pages = {261-273},
year = {1973}
}

@article{Gavril1974AlgorithmsOC,
author = {Gavril, F.},
title = {Algorithms on circular-arc graphs},
journal = {Networks},
volume = {4},
number = {4},
pages = {357-369},
year = {1974}
}

@article{DBLP:journals/networks/GuptaLL82,
  author       = {Udaiprakash I. Gupta and
                  D. T. Lee and
                  Joseph Y.{-}T. Leung},
  title        = {Efficient algorithms for interval graphs and circular-arc graphs},
  journal      = {Networks},
  volume       = {12},
  number       = {4},
  pages        = {459--467},
  year         = {1982},
}

@article{wo_long_odd_cycle,
	author = {Hsu, Wen-lian and Ikura, Yoshiro and Nemhauser, George L.},
	date = {1981/12/01},
	journal = {Mathematical Programming},
	number = {1},
	pages = {225--232},
	title = {A polynomial algorithm for maximum weighted vertex packings on graphs without long odd cycles},
	volume = {20},
	year = {1981},
}

@article{Babel,
	author = {Babel, L. },
	date = {1991/12/01},
	id = {Babel1991},
	journal = {Computing},
	number = {4},
	pages = {321--341},
	title = {Finding maximum cliques in arbitrary and in special graphs},
	volume = {46},
	year = {1991},
}

@article{Babel_Tinhofer,
	author = {Babel, L. and Tinhofer, G.},
	date = {1990/05/01},
	id = {Babel1990},
	journal = {Zeitschrift f{\"u}r Operations Research},
	number = {3},
	pages = {207--217},
	title = {A branch and bound algorithm for the maximum clique problem},
	volume = {34},
	year = {1990},
}

@article{doi:10.1137/0215075,
author = {Balas, Egon and Yu, Chang Sung},
title = {Finding a Maximum Clique in an Arbitrary Graph},
journal = {SIAM Journal on Computing},
volume = {15},
number = {4},
pages = {1054-1068},
year = {1986}
}

@article{FRIDEN1990437,
title = {Tabaris: An exact algorithm based on tabu search for finding a maximum independent set in a graph},
journal = {Computers and Operations Research},
volume = {17},
number = {5},
pages = {437-445},
year = {1990},
author = {C. Friden and A. Hertz and D. {de Werra}}
}

@article{PARDALOS1992363,
title = {A branch and bound algorithm for the maximum clique problem},
journal = {Computers and Operations Research},
volume = {19},
number = {5},
pages = {363-375},
year = {1992},
author = {Panos M. Pardalos and Gregory P. Rodgers}
}

@article{CARRAGHAN1990375,
title = {An exact algorithm for the maximum clique problem},
journal = {Operations Research Letters},
volume = {9},
number = {6},
pages = {375-382},
year = {1990},
author = {Randy Carraghan and Panos M. Pardalos}
}

@article{Mannino_Sassano,
	author = {Mannino, Carlo and Sassano, Antonio},
	date = {1994/07/01},
	journal = {Computational Optimization and Applications},
	number = {3},
	pages = {243--258},
	title = {An exact algorithm for the maximum stable set problem},
	volume = {3},
	year = {1994}}

@article{Balas1991MinimumWC,
  title={Minimum Weighted Coloring of Triangulated Graphs, with Application to Maximum Weight Vertex Packing and Clique Finding in Arbitrary Graphs},
  author={Egon Balas and Jue Xue},
  journal={SIAM Journal on Computing},
  year={1991},
  volume={20},
  pages={209-221},
}

@article{NemSigismondi,
author = {Nemhauser, George and Sigismondi, G.},
year = {1992},
month = {05},
pages = {443-457},
title = {A Strong Cutting Plane/Branch-and-Bound Algorithm for Node Packing},
volume = {43},
journal = {Journal of Operational Research Society},
}

@article{NemTrotter,
	author = {Nemhauser, G. L. and Trotter, L. E.},
	date = {1975/12/01},
	journal = {Mathematical Programming},
	number = {1},
	pages = {232--248},
	title = {Vertex packings: Structural properties and algorithms},
	volume = {8},
	year = {1975}}

@Article{BalasSamuelsson,
  author={Egon Balas and Haakon Samuelsson},
  title={{A node covering algorithm}},
  journal={Naval Research Logistics Quarterly},
  year=1977,
  volume={24},
  number={2},
  pages={213-233},
  month={June},
}

@article{BensonYe,
author = {Benson,S. and Ye,Yinyu},
year = {2000},
month = {July},
pages = {},
title = {Approximating Maximum Stable Set And Minimum Graph Coloring Problems With The Positive Semidefinite Relaxation},
journal = {Applications and Algorithms of Complementarity},
}

@article{ESCHEN2014195,
title = {Algorithms for unipolar and generalized split graphs},
journal = {Discrete Applied Mathematics},
volume = {162},
pages = {195-201},
year = {2014},
author = {Elaine M. Eschen and Xiaoqiang Wang}
}

@misc{monteiro2024lowrankaugmentedlagrangianmethod,
      title={A low-rank augmented Lagrangian method for large-scale semidefinite programming based on a hybrid convex-nonconvex approach}, 
      author={Renato D. C. Monteiro and Arnesh Sujanani and Diego Cifuentes},
      year={2024},
      archivePrefix={arXiv},
      primaryClass={math.OC},
      url={https://arxiv.org/abs/2401.12490}, 
}

@article{mcdiarmid2017randomperfectgraphs,
author = {McDiarmid, Colin and Yolov, Nikola},
title = {Random perfect graphs},
journal = {Random Structures \& Algorithms},
volume = {54},
number = {1},
pages = {148-186},
year = {2019}
}

@article{DNP,
	author = {Muir, Christopher and Toriello, Alejandro},
	date = {2022/11/01},
	id = {Muir2022},
	journal = {Mathematical Programming},
	number = {1},
	pages = {875--906},
	title = {Dynamic node packing},
	volume = {196},
	year = {2022},}

@article{diag_dom_matrices,
author = {Barker, George and Carlson, David},
year = {1975},
month = {March},
pages = {15-32},
title = {Cones of diagonally dominant matrices},
volume = {57},
journal = {Pacific Journal of Mathematics},
}

@article{doi:10.1137/0805002,
author = {Alizadeh, Farid},
title = {Interior Point Methods in Semidefinite Programming with Applications to Combinatorial Optimization},
journal = {SIAM Journal on Optimization},
volume = {5},
number = {1},
pages = {13-51},
year = {1995},
}

@book{ams1996dimacs,
  author    = {Johnson, David Stifler and  Trick, Michael A.},
  title     = {Cliques, Coloring, and Satisfiability},
  year      = {1996},
  publisher = {American Mathematical Society},
  series    = {DIMACS Series in Discrete Mathematics and Theoretical Computer Science},
  volume    = {26},
}

@misc{gset,
  author       = {Yinyu Ye},
  title        = {Gset dataset of random graphs},
  howpublished = {\url{www.cise.ufl.edu/research/sparse/matrices/Gset/}},
  note         = {Accessed: 2023-10-25}
}

@manual{sagemath,
  Key          = {SageMath},
  Author       = {{The Sage Developers}},
  Title        = {{S}ageMath, the {S}age {M}athematics {S}oftware {S}ystem ({V}ersion 10.4)},
  url         = {https://www.sagemath.org},
  Year         = {2024},
}

@misc{snapnets,
                    author       = {Jure Leskovec and Andrej Krevl},
                    title        = {{SNAP Datasets}: {Stanford} Large Network Dataset Collection},
                    howpublished = {\url{http://snap.stanford.edu/data}},
                    month        = {June},
                    year         = 2014
                  }

@Article{prosser2012exactalgorithmsmaximumclique,
AUTHOR = {Prosser, Patrick},
TITLE = {Exact Algorithms for Maximum Clique: A Computational Study },
JOURNAL = {Algorithms},
VOLUME = {5},
YEAR = {2012},
NUMBER = {4},
PAGES = {545--587},
}

@article{pjm/1102995572,
author = {D. R. Fulkerson and O. A. Gross},
title = {{Incidence matrices and interval graphs.}},
volume = {15},
journal = {Pacific Journal of Mathematics},
number = {3},
publisher = {Pacific Journal of Mathematics, A Non-profit Corporation},
pages = {835 -- 855},
year = {1965},
}

@book{dimacs10,
  title        = {Graph Partitioning and Graph Clustering. 10th DIMACS Implementation Challenge Workshop},
  author        = {David A. Bader and Henning Meyerhenke and Peter Sanders and Dorothea Wagner and editors},
  series       = {Contemporary Mathematics},
  volume       = {588},
  publisher    = {American Mathematical Society},
  year         = {2013},
}

@article{WU2015693,
title = {A review on algorithms for maximum clique problems},
journal = {European Journal of Operational Research},
volume = {242},
number = {3},
pages = {693-709},
year = {2015},
author = {Qinghua Wu and Jin-Kao Hao}
}

@misc{marino2024shortreviewnovelapproaches,
      title={A Short Review on Novel Approaches for Maximum Clique Problem: from Classical algorithms to Graph Neural Networks and Quantum algorithms}, 
      author={Raffaele Marino and Lorenzo Buffoni and Bogdan Zavalnij},
      year={2024},
      archivePrefix={arXiv},
      primaryClass={cs.AI},
      url={https://arxiv.org/abs/2403.09742}, 
}

@article{pseudo-inv,
 author = {Arthur Albert},
 journal = {SIAM Journal on Applied Mathematics},
 number = {2},
 pages = {434--440},
 publisher = {Society for Industrial and Applied Mathematics},
 title = {Conditions for Positive and Nonnegative Definiteness in Terms of Pseudoinverses},
 volume = {17},
 year = {1969}
}

@article{10.1137/0201013,
author = {Gavril, F\u{a}nic\u{a}},
title = {Algorithms for Minimum Coloring, Maximum Clique, Minimum Covering by Cliques, and Maximum Independent Set of a Chordal Graph},
year = {1972},
issue_date = {Jun 1972},
publisher = {Society for Industrial and Applied Mathematics},
address = {USA},
volume = {1},
number = {2},
journal = {SIAM J. Comput.},
month = jun,
pages = {180–187},
numpages = {8}
}

@Article{Garstka_2021,
  author  = {Michael Garstka and Mark Cannon and Paul Goulart},
  journal = {Journal of Optimization Theory and Applications},
  title   = {{COSMO}: A Conic Operator Splitting Method for Convex Conic Problems},
  volume  = {190},
  number  = {3},
  pages   = {779--810},
  year    = {2021},
  publisher = {Springer},
  doi     = {10.1007/s10957-021-01896-x},
  url     = {https://doi.org/10.1007/s10957-021-01896-x}
}

@article{doi:10.1287/opre.16.3.682,
author = {Murty, Katta G.},
title = {Letter to the Editor—An Algorithm for Ranking all the Assignments in Order of Increasing Cost},
journal = {Operations Research},
volume = {16},
number = {3},
pages = {682-687},
year = {1968},
doi = {10.1287/opre.16.3.682},

URL = { 
    
        https://doi.org/10.1287/opre.16.3.682
    
    

},
eprint = { 
    
        https://doi.org/10.1287/opre.16.3.682
    
    

}
}

@article{doi:10.1287/mnsc.18.7.401,
author = {Lawler, Eugene L.},
title = {A Procedure for Computing the K Best Solutions to Discrete Optimization Problems and Its Application to the Shortest Path Problem},
journal = {Management Science},
volume = {18},
number = {7},
pages = {401-405},
year = {1972},
doi = {10.1287/mnsc.18.7.401},

URL = { 
    
        https://doi.org/10.1287/mnsc.18.7.401
    
    

},
eprint = { 
    
        https://doi.org/10.1287/mnsc.18.7.401
    
    

}
}

@article{OKAMOTO2008229,
title = {Counting the number of independent sets in chordal graphs},
journal = {Journal of Discrete Algorithms},
volume = {6},
number = {2},
pages = {229-242},
year = {2008},
note = {Selected papers from CompBioNets 2004},
issn = {1570-8667},
doi = {https://doi.org/10.1016/j.jda.2006.07.006},
url = {https://www.sciencedirect.com/science/article/pii/S1570866707000330},
author = {Yoshio Okamoto and Takeaki Uno and Ryuhei Uehara}
}

@INPROCEEDINGS{143921,
  author={Liang, Y.D. and Dhall, S.K. and Lakshmivarahan, S.},
  booktitle={[Proceedings] 1991 Symposium on Applied Computing}, 
  title={On the problem of finding all maximum weight independent sets in interval and circular-arc graphs}, 
  year={1991},
  volume={},
  number={},
  pages={465-470},
  doi={10.1109/SOAC.1991.143921}}

@misc{copt,
  author={Dongdong Ge and Qi Huangfu and Zizhuo Wang and Jian Wu and Yinyu Ye},
  title={Cardinal {O}ptimizer {(COPT)} user guide},
  howpublished={https://guide.coap.online/copt/en-doc},
  year=2023
}
}

\newpage
\appendix
\section{Additional proofs}
\begin{lemma}
\label{lem:PSD and range}
    Given fixed $Q\in\PSD^{n}$, $q\in\R^{n}$, there exists $\varphi>0$ such that $\varphi qq^\top\preceq Q$ if and only if $q\in\range(Q)$.
\end{lemma}
\begin{proof}
    Suppose there exists $\varphi>0$ such that $\varphi qq^\top\preceq Q$. 
    For contradiction, assume $q\notin\range(Q)$; then $q=a+b$, where $a\in\range(Q)$ and $b\in\Null(Q)\setminus\{0\}$ by the orthogonal decomposition. Then
\[
0<\varphi\|b\|_2^2=\varphi b^\top qq^\top b\leq b^\top Qb=0 ,
\]
which gives a contradiction.

Suppose $q\in\range(Q)$; then we can write $q=Qy$ for some $y$. 
By the singular value decomposition, there is an orthonormal matrix $U$ and a diagonal matrix $D$ such that $Q=UDU^\top$.
Without loss of generality, assume $D_{11}\geq \ldots D_{rr}>0$, for $r:=\rank(Q)$.
Thus, we have 
\[
U^\top qq^\top U=DU^\top yy^\top UD,
\]
where $[U^\top qq^\top U]_{ij}=0$ if $i>r$ or $j>r$.
Hence, $ U^\top qq^\top U\preceq \varphi D $ for some $\varphi>0$, and $ \varphi^{-1}qq^\top\preceq UDU^\top $.
   \end{proof}

\end{document}